\documentclass[twoside,a4paper,12pt,centertags]{amsart}
\usepackage{amsmath,amssymb,verbatim,vmargin}
\usepackage[bookmarks=true]{hyperref}   

\theoremstyle{plain}
\newtheorem{thm}{Theorem}[section]
\newtheorem{lem}[thm]{Lemma}
\newtheorem{cor}[thm]{Corollary}
\newtheorem{prop}[thm]{Proposition}

\theoremstyle{definition}
\newtheorem{defn}[thm]{Definition}
\newtheorem{rem}[thm]{Remark}

\newtheorem{ex}[thm]{Example}

\numberwithin{equation}{section}

\title[Functional calculus and complex perturbations]{Functional 
calculus of Dirac operators and complex perturbations of Neumann and Dirichlet problems}
\author{Pascal Auscher} \author{Andreas Axelsson} \author{Steve Hofmann}
\address{Pascal Auscher, Universit\'e de Paris-Sud, UMR du CNRS 8628, 91405 Orsay Cedex, France}
\email{pascal.auscher@math.u-psud.fr}
\address{Andreas Axelsson, Matematiska institutionen, Stockholms universitet, 106 91 Stockholm, Sweden}
\email{andax@math.su.se}
\address{Steve Hofmann, Mathematics Department, University of Missouri, Columbia 65211, USA}
\email{hofmann@math.missouri.edu}

\mathchardef\semic="303B
\newcommand{\Mcc}{{M\raise.55ex\hbox{\lowercase{c}}}}
\newcommand{\dirac}{{\mathbf D}}
\newcommand{\wedg}{\mathbin{\scriptstyle{\wedge}}}
\newcommand{\lctr}{\mathbin{\lrcorner}}
\newcommand{\R}{{\mathbf R}}
\newcommand{\C}{{\mathbf C}}
\newcommand{\Z}{{\mathbf Z}}
\newcommand{\mH}{{\mathcal H}}
\newcommand{\mK}{{\mathcal K}}
\newcommand{\mX}{{\mathcal X}}
\newcommand{\mY}{{\mathcal Y}}
\newcommand{\mL}{{\mathcal L}}
\newcommand{\mF}{{\mathcal F}}
\newcommand{\mP}{{\mathcal P}}
\newcommand{\V}{{\mathcal V}}
\DeclareMathOperator{\re}{Re}

\newcommand{\sett}[2]{ \{ #1 \, \semic \, #2 \} }

\newcommand{\dual}[2]{\langle #1,#2 \rangle}
\newcommand{\brac}[1]{\langle #1 \rangle}
\newcommand{\supp}{\text{{\rm supp}}\,}
\newcommand{\dist}{\text{{\rm dist}}\,}

\newcommand{\nul}{\textsf{N}}
\newcommand{\ran}{\textsf{R}}
\newcommand{\dom}{\textsf{D}}
\newcommand{\graf}{\textsf{G}}
\newcommand{\clos}[1]{\overline{#1}}
\newcommand{\conj}[1]{\overline{#1}}
\newcommand{\dyadic}{\triangle}
\newcommand{\sgn}{\text{{\rm sgn}}}
\newcommand{\barint}{\mbox{$ave \int$}}
\newcommand{\divv}{{\text{{\rm div}}}}
\newcommand{\curl}{{\text{{\rm curl}}}}

\newcommand{\dd}[2]{\frac{\partial #1}{\partial #2}}
\newcommand{\tdd}[2]{\tfrac{\partial #1}{\partial #2}}
\newcommand{\ud}{\underline{d}}
\newcommand{\hut}[1]{\check #1}
\newcommand{\PP}{{\mathbf P}}
\newcommand{\tb}[1]{\| \hspace{-0.42mm} | #1 \| \hspace{-0.42mm} |}
\newcommand{\tbb}[1]
{\left\| \hspace{-0.42mm} \left| #1 \right\| \hspace{-0.42mm} \right|}
\newcommand{\wt}{\widetilde}
\newcommand{\ta}{{\scriptscriptstyle \parallel}}
\newcommand{\no}{{\scriptscriptstyle\perp}}
\newcommand{\pd}{\partial}
\newcommand{\ep}{{\mathcal E}}
\newcommand{\op}{\text{op}}
\newcommand{\off}{\text{off}}
\newcommand{\hB}{{\hat B}}
\newcommand{\hA}{{\hat A}}
\newcommand{\m}{m}

\def\barint_#1{\mathchoice
            {\mathop{\vrule width 6pt
height 3 pt depth -2.5pt
                    \kern -8.8pt
\intop}\nolimits_{#1}}%
            {\mathop{\vrule width 5pt height
3 pt depth -2.6pt
                    \kern -6.5pt
\intop}\nolimits_{#1}}%
            {\mathop{\vrule width 5pt height
3 pt depth -2.6pt
                    \kern -6pt
\intop}\nolimits_{#1}}%
            {\mathop{\vrule width 5pt height
3 pt depth -2.6pt
          \kern -6pt \intop}\nolimits_{#1}}}

\begin{document}

\begin{abstract}
We prove that the Neumann, Dirichlet and regularity problems for divergence
form elliptic equations in the half space are well posed in $L_2$
for small complex $L_\infty$ perturbations of a coefficient matrix
which is either real symmetric, of block form or constant.
All matrices are assumed to be independent of the transversal 
coordinate.
We solve the Neumann, Dirichlet and regularity problems through a new
boundary operator method which makes use of operators in the functional
calculus of an underlaying first order Dirac type
operator. We establish quadratic estimates for this Dirac operator, which implies
that the associated Hardy projection operators are bounded and depend 
continuously on the coefficient matrix.
We also prove that certain transmission problems for
$k$-forms are well posed for small perturbations of 
block matrices.
\end{abstract}

\maketitle

\tableofcontents

%
%
%
\section{Introduction}   \label{section1}

In this paper we prove that the Neumann, Dirichlet and regularity problems 
are well posed
in $L_2(\R^n)$ for divergence form second order elliptic equations
\begin{equation}  \label{eq:divform}
  \divv_{t,x} A(x) \nabla_{t,x} U(t,x) =0
\end{equation}
on the half space $\R^{n+1}_+ := \sett{(t,x)\in\R\times \R^n}{t>0}$, $n\ge 1$,
when $A$ is a small complex $L_\infty$ perturbation of either a block matrix, 
a constant matrix or a real symmetric matrix.
Furthermore, the matrix $A=(a_{ij}(x))_{i,j=0}^n\in L_\infty(\R^n;\mL(\C^{n+1}))$ 
is assumed to be $t$-independent with complex coefficients and accretive,
with quantitative bounds $\|A\|_\infty$ and $\kappa_A$, where 
$\kappa_A>0$ is the largest constant such that
$$
  \re (A(x)v,v) \ge \kappa_A |v|^2,
  \qquad \text{for all } v\in\C^{n+1}, 
  \,\, x\in\R^n.
$$
We shall approach the equation (\ref{eq:divform}) from a first order point of view,
rewriting it as the first order system
\begin{align} \label{eq:Laplacein1order}
\begin{cases}   
  \divv_{t,x} A(x) F(t,x)  =0, \\
  \curl_{t,x} F(t,x) =0,
\end{cases}
\end{align}
where $F(t,x)=\nabla_{t,x}U(t,x)$. 
Recall that a vector field $F=F_0e_0+F_1e_1+\ldots+F_ne_n$ can be written in this way as a gradient 
if and only if $\curl_{t,x} F=0$, by which we understand that 
$\pd_jF_i= \pd_i F_j$, for all $i$, $j=0,\ldots,n$.
We write $\{e_0,e_1,\ldots,e_n\}$ for the standard basis for $\R^{n+1}$ with 
$e_0$ upward pointing into $\R^{n+1}_+$, and write $t=x_0$ for the vertical
coordinate. For the vertical derivative, we write $\partial_0 = \partial_t$.
Denote also by $F_\ta:= F_1e_1+\ldots+F_ne_n$, the tangential part of $F$, and
write $\curl_{x} F_\ta=0$ if 
$\pd_jF_i= \pd_i F_j$, for all $i$, $j=1,\ldots,n$.

In the formulation of the boundary value problems below, we assume that 
$A=A(x)$ is a given coefficient matrix with properties as above.
Furthermore, by saying that $F_t(x)=F(t,x)$ satisfies (\ref{eq:Laplacein1order})
we shall mean that $F_t\in C^1(\R_+; L_2(\R^n;\C^{n+1}))$ and that for each 
fixed $t>0$, we have 
$\divv_x(AF)_\ta= - (A(x)\pd_0 F)_0$, 
$\nabla_x F_0 = \pd_0 F_\ta$ and $\curl_x F_\ta=0$,
where the derivatives on the left hand sides are taken in the sense of distributions.
If this holds for $F$, then in particular we can write $F=\nabla_{t,x}U$, with 
$U\in W^1_{2,\text{loc}}(\R^{n+1}_+)$, and we see that $U$ satisfies 
(\ref{eq:divform}) in the sense that 
$$
  \iint_{\R^{n+1}_+}  \big(A(x)\nabla_{t,x}U(t,x),\nabla_{t,x}\varphi(t,x)\big)\, dtdx=0,\qquad
  \text{for all } \varphi\in C^\infty_0(\R^{n+1}_+).
$$
\vspace{2mm}
\noindent THE NEUMANN PROBLEM (Neu-$A$).

Given a function $\phi(x)\in L_2(\R^n;\C)$,
find a vector field 
$F_t(x) =F(t,x)$ in $\R^{n+1}_+$ 
such that $F_t\in C^1(\R_+; L_2(\R^n;\C^{n+1}))$ and
$F$ satisfies (\ref{eq:Laplacein1order}) for $t>0$,
and furthermore $\lim_{t\rightarrow \infty}F_t=0$ and $\lim_{t\rightarrow 0}F_t=f$
in $L_2$ norm,
where the conormal part of $f$ satisfies the boundary condition
$$
  e_0\cdot(A f) = \sum_{j=0}^n A_{0j} f_j= \phi, \qquad \text{on }\R^n =\partial \R^{n+1}_+.
$$
For $U$, this means that the conormal derivative
$\tdd U{\nu_A}(0,x)=\phi(x)$ in $L_2(\R^n)$.

\vspace{2mm}
\noindent THE REGULARITY PROBLEM (Reg-$A$).

Given a function $\psi:\R^n\rightarrow\C$ with tangential
gradient $\nabla_x \psi\in L_2(\R^n;\C^n)$,
find a vector field 
$F_t(x) =F(t,x)$ in $\R^{n+1}_+$ 
such that $F_t\in C^1(\R_+; L_2(\R^n;\C^{n+1}))$ and
$F$ satisfies (\ref{eq:Laplacein1order}) for $t>0$,
and furthermore $\lim_{t\rightarrow \infty}F_t=0$ and $\lim_{t\rightarrow 0}F_t=f$
in $L_2$ norm, 
where the tangential part of $f$ satisfies the boundary condition
$$
  f_\ta= f_1 e_1+\ldots + f_n e_n = \nabla_x\psi, 
  \qquad \text{on }\R^n =\partial \R^{n+1}_+.
$$
For $U$, this means that $\nabla_x U(x,0)=\nabla_x\psi(x)$, i.e. $U(x,0)=\psi(x)$ in $\dot W^1_2(\R^n)$.

\vspace{2mm}
\noindent THE DIRICHLET PROBLEM (Dir-$A$).

Given a function $u(x)\in L_2(\R^n;\C)$,
find a function
$U_t(x) =U(t,x)$ in $\R^{n+1}_+$ 
such that $U_t\in C^2(\R_+; L_2(\R^n;\C))$, $\nabla_{t,x} U_t\in C^1(\R_+; L_2(\R^n;\C^{n+1}))$ and
$\nabla_{t,x} U_t$ satisfies (\ref{eq:Laplacein1order}) for $t>0$,
and furthermore $\lim_{t\rightarrow \infty}U_t=0$, $\lim_{t\rightarrow \infty}\nabla_{t,x} U_t=0$ 
and $\lim_{t\rightarrow 0}U_t=u$
in $L_2$ norm.

\vspace{2mm}
We shall also use first order methods based on (\ref{eq:Laplacein1order}) to solve the Dirichlet problem.
However, here we use a different relation between (\ref{eq:divform}) and (\ref{eq:Laplacein1order}) 
which we now describe.
Assume $F(t,x)=\sum_{k=0}^n F_i(t,x) e_i$ is a vector field satisfying (\ref{eq:Laplacein1order}).
Applying $\pd_t$ to the first equation and using that $\curl_{t,x} F =0$ yields
$$
  0= \pd_t(\divv_{t,x}A(x)F)= \divv_{t,x} A(x)(\pd_t F)= \divv_{t,x} A(x)(\nabla_{t,x} F_0),
$$
since the coefficients are assumed to be $t$-independent.
Thus the normal component $U:= F_0$ satisfies (\ref{eq:divform}).
Note that when $A=I$, the functions $F_1,\ldots , F_n$ are conjugates to $U$ in the sense of
Stein and Weiss~\cite{SW}.
From this we see that solvability of (Dir-$A$) is a direct consequence of solvability of the 
following auxiliary Neumann problem.

\vspace{2mm}
\noindent THE NEUMANN PROBLEM (Neu$^\perp$-$A$).

Given a function $\phi(x)\in L_2(\R^n;\C)$,
find a vector field 
$F_t(x) =F(t,x)$ in $\R^{n+1}_+$ 
such that $F_t\in C^1(\R_+; L_2(\R^n;\C^{n+1}))$ and
$F$ satisfies (\ref{eq:Laplacein1order}) for $t>0$,
and furthermore $\lim_{t\rightarrow \infty}F_t=0$ and $\lim_{t\rightarrow 0}F_t=f$
in $L_2$ norm,
where the normal part of $f$ is $e_0\cdot f=f_0=\phi$.

\vspace{2mm}
The main result of this paper is the following $L_\infty$ perturbation result 
for the boundary value problems.
\begin{thm}  \label{thm:main}
Let $A_0(x)= ((a_0)_{ij}(x))_{i,j=0}^n$ be a $t$-independent, complex,
accretive coefficient matrix function.
Furthermore assume that $A_0$ has one of the following extra properties.
\begin{itemize}
\item[{\rm (b)}]
  $A_0$ is a block matrix, i.e. $(a_0)_{0i}(x) = (a_0)_{i0}(x)=0$ 
  for all $1\le i \le n$ and all $x\in \R^n$.
\item[{\rm (c)}]
  $A_0$ is a constant coefficient matrix, i.e. $A_0(x) =A_0(y)$ 
  for all $x$, $y\in \R^n$.
\item[{\rm (s)}]
  $A_0$ is a real symmetric matrix, i.e. 
  $(a_0)_{ij}(x) = (a_0)_{ji}(x)\in \R$ 
  for all $0\le i,j \le n$ and all $x\in \R^n$.
\end{itemize}
Then there exists $\epsilon>0$ depending only on $\|A_0\|_\infty$, the
accretivity constant $\kappa_{A_0}$ and the dimension $n$, such that if 
$A\in L_\infty(\R^n;\mL(\C^{n+1}))$ is $t$-independent and satisfies 
$\|A-A_0\|_{\infty} <\epsilon$, then
the Neumann and regularity problems (Neu-$A$), (Neu$^\perp$-$A$) and (Reg-$A$) above
have a unique solution $F(t,x)$ with the required properties
for every boundary function $g(x)$, being $\phi(x)$ and $\nabla_x\psi(x)$ respectively.
Furthermore the Dirichlet problem (Dir-$A$) above
has a unique solution $U(t,x)$ with the required properties
for every boundary function $u(x)$.

The solutions depend continuously on the data with the following equivalences of norms.
If we define the triple bar norm $\tb{G_t}^2:= \int_0^\infty \|G_t\|_2^2\, t^{-1}dt$ and 
the non-tangential maximal function
$$
  \widetilde N_*(F)(x):= \sup_{t>0} \left(
\barint_{\hspace{-6pt} |s-t|<c_0t} \barint_{\hspace{-6pt} |y-x|<c_1t}
  |F(s,y)|^2 ds\,dy \right)^{1/2},
$$
where $\barint_{\hspace{-3pt} E}:= |E|^{-1}\int_E$ and $c_0\in(0,1)$, $c_1>0$ are constants, 
then for the Neumann and regularity problems we have
$$
  \|g\|_2\approx \|f\|_2\approx \sup_{t>0} \|F_t\|_2
  \approx \tb{t\pd_t F_t}
  \approx \|\widetilde N_* (F)\|_2,
$$
and for the Dirichlet problem we have
$$
  \|u\|_2\approx \sup_{t>0} \|U_t\|_2
  \approx \tb{t\pd_t U_t}
  \approx \tb{t\nabla_x U_t}
  \approx \|\widetilde N_* (U)\|_2.
$$
Moreover, the solution operators $S_A$, being $S_A(g)= F$ or $S_A(u)=U$ respectively, 
depend Lipschitz continuously on $A$, i.e. there exists $C<\infty$ such that 
$$
  \|S_{A_2}-S_{A_1}\|_{L_2(\R^n)\rightarrow \mX}\le C \|A_2-A_1\|_{L_\infty(\R^n)}
$$
when $\|A_i-A_0\|_{\infty} <\epsilon$, $i=1,2$, where $\|F\|_\mX$ or 
$\|U\|_\mX$ denotes any of the norms above.
\end{thm}
Throughout this paper, we use the notation $X \approx Y$ and $X \lesssim Y$
to mean that there exists a constant $C>0$ so that $ X/C \le Y \le CX$ and 
$X\le C Y$, respectively. The value of $C$ varies from one usage
to the next, but then is always fixed.

Let us now review the history of works on these boundary value problems,
starting with the case of matrices of the form $\tilde A$
which we now describe.
By standard arguments, Theorem~\ref{thm:main}
also shows well posedness of the corresponding boundary value problems on 
the region $\Omega$ above a Lipschitz graph
$\Sigma= \sett{(t,x)}{t= g(x)}$, where $g:\R^n\rightarrow \R$
is a Lipschitz function.
Indeed, if the function $U(t,x)$ satisfies 
$\divv_{t,x} A(x) \nabla_{t,x} U=0$ in $\Omega$ then $\wt U(t,x):= U(t+g(x),x)$
satisfies $\divv_{t,x} \tilde A(x) \nabla_{t,x}\wt U=0$ in $\R^{n+1}_+$, 
where
$$
  \tilde A(x) := 
     \begin{bmatrix}
       1 & -(\nabla_x g(x))^t \\
       0 & I
     \end{bmatrix}
     A(x)
     \begin{bmatrix}
       1 & 0 \\
       -\nabla_x g(x) & I
     \end{bmatrix}.
$$
Thus Theorem~\ref{thm:main} gives conditions on $A$ 
for which the Neumann problem (conormal derivative 
$e_0\cdot \tilde A\nabla_{t,x}\wt U= (e_0-\nabla_x g)\cdot A\nabla_{t,x}U$ 
given), the regularity problem (tangential gradient 
$\nabla_x\wt U= (\partial_t U)\nabla_x g+ \nabla_x U$ given),
and the Dirichlet problem ($\wt U=U$ given) are well posed.
Note that $A$ is real symmetric if and only if $\tilde A$ is,
but that $\tilde A$ being constant or of block form does not imply the
same for $A$.
For the Laplace equation $A=I$ in $\Omega$, solvability of 
(Neu-$\tilde I$) and (Reg-$\tilde I$) was first proved by
Jerison and Kenig~\cite{JK2}, and solvability of (Dir-$\tilde I$) was first
proved by Dahlberg~\cite{D}. Later Verchota~\cite{V} showed 
that these boundary value problems are solvable with the layer potential
integral equation method.

For general real symmetric matrices $A$, not being of the ''Jacobian type'' $\tilde I$
above, the well posedness of (Dir-$A$) was first proved by 
Jerison and Kenig~\cite{JK1}, and (Neu-$A$) and (Reg-$A$)
by Kenig and Pipher~\cite{KP}.
These results make use of the Rellich estimate technique. For the Neumann
and regularity problems, this integration by parts technique yields an
equivalence
\begin{equation}  \label{eq:rellichkato}
  \big\|\tdd U{\nu_A}\big\|_{L_2(\R^n)} \approx \|\nabla_x U \|_{L_2(\R^n)},
\end{equation}
which is seen to be equivalent with the first estimate $\|f\|\approx \|g\|$ 
in the theorem above, and shows that the boundary trace $f$ splits into
two parts of comparable size.

Turning to the unperturbed case where $A=A_0$ satisfies (b), then (\ref{eq:rellichkato}) 
is still valid, but the proof is far
deeper than Rellich estimates. In fact, it is equivalent with 
the Kato square root estimate proved by Auscher, Hofmann, Lacey, \Mcc Intosh and Tchamitchian
in \cite{AHLMcT}.
(For the non divergence form case $a_{00}\ne 1$, see \cite{AKMc}.)
For details concerning this equivalence between the Kato problem
and the boundedness and invertibility of the Dirichlet-to-Neumann map
$\nabla_x U\mapsto \dd U\nu$ we refer to 
Kenig~\cite[Remark 2.5.6]{kenig}, where also many further references in the field
can be found.

We now consider what is previously known in the case when $A$ does not satisfy
(b), (c) or (s).
Here (Dir-$A$) has been showed to be well posed by Fabes, Jerison and Kenig
~\cite{FJK} for small perturbations of (c), using the method of multilinear 
expansions.
More recently, the boundary value problems have been studied in the $L_p$
setting and for real but non-symmetric matrices in the plane, i.e. $n=1$. 
Here Kenig, Koch, Pipher and Toro~\cite{KKPT}
have obtained solvability of the Dirichlet problem for sufficiently large $p$,
and Kenig and Rule~\cite{KR} have shown solvability of the Neumann and regularity 
problems for sufficiently small dual exponent $p'$.

In the perturbed case $A \approx A_0$ when $A_0$ satisfies 
(c) or (s), the well posedness of (Neu-$A$), (Reg-$A$) and
(Dir-$A$) is also proved in \cite{AAAHK} by Alfonseca, Auscher, Axelsson, 
Hofmann and Kim.
With the further assumption of pointwise resolvent kernel bounds,
perturbation of case (b) is also implicit in \cite{AAAHK}.  
It is worth comparing the present methods to those of \cite{AAAHK}.  
In \cite{AAAHK}, due to the presence of kernel bounds, 
the solvability of the boundary value problems is meant 
in the sense of non-tangential maximal estimates at the boundary 
and this follows from the use of layer potentials.   
The first main result in \cite{AAAHK} in the unperturbed case (s), 
is the proof via singular integral operator theory of boundedness 
and invertibility of layer potentials. 
The second main result in \cite{AAAHK} is the stability of the 
simultaneous occurrence of both boundedness and invertibility, 
which hold in the unperturbed cases (c), (s) and (b). 
Solvability then follows.  

Here, we setup a different resolution algorithm 
(forcing us to introduce some substantial material), 
which consists in solving the first order system (\ref{eq:Laplacein1order}) instead of 
(\ref{eq:divform}), also by a boundary operator method, but acting on the 
gradient of solutions involving a generalised Cauchy operator $E_A$, 
the goal being to establish boundedness of $E_A$ and invertibility
of related operators $E_A\pm N_A$. 
Boundedness of $E_A=\sgn(T_A)$ is obtained via quadratic estimates of an underlaying 
first order differential operator $T_A$, and the deep fact 
is here that those quadratic estimates alone are stable under 
perturbations. Stability of invertibility is then easy.
The perturbation argument requires sophisticated harmonic analysis techniques inspired 
by the strategy of \cite{AAAHK}. In particular, the latter uses extensively the technology 
of the solution of the Kato problem for second order operators in \cite{AHLMcT}, whereas 
we utilise here the work of Axelsson, Keith and \Mcc Intosh~\cite{AKMc}, which adapts 
and extends this technology to first order operators of Dirac type. Indeed, we note 
that our Dirac type operators $T_A$ are of the form $\Pi_A$ of \cite{AKMc} in the case 
of block matrices (b) of Theorem~\ref{thm:main}. But $T_A$ has a more complicated 
structure when $A$ is not a block matrix and we understand how to prove boundedness 
of $E_A$ at the moment only in the cases specified by Theorem~\ref{thm:main}. We also 
note that the present paper, like \cite{AKMc}, makes no use of kernel bounds and only 
needs $L_2$ off-diagonal bounds for the operators, which always holds.

The boundary operator method for first order Dirac type operators, used here to solve 
second order boundary value problems, was developed in the thesis of 
Axelsson~\cite{Ax}, which has been published as the four papers \cite{Ax1, Ax2, AMc, Ax3}. 
It covers operators on Lipschitz domains as described above and in Example~\ref{ex:diraclip}.
The result in \cite{AKMc} pursued the program initiated  by
Auscher, \Mcc Intosh and Nahmod in \cite{AMcN}, consisting of connecting the Kato problem 
and the functional calculus of first order differential operators of Dirac type. 
As said, it thus applies to the boundary value problems for operators of case (b). 
What is new here is the setup for full matrices encompassing the above. 
We prove also a sort of meta-theorem (see Theorem~\ref{thm:mainpert}) which roughly 
says that the set of matrices for which the needed quadratic estimates on $T_A$ hold, is open.


We also show that non-tangential maximal estimates hold for our solutions. 
By uniqueness in the class of  
solutions of (\ref{eq:divform}) with non-tangential maximal estimates,
this implies that our solutions are the same as those in \cite{AAAHK} 
for perturbations of the real symmetric and constant cases.
The non-tangential maximal estimate here also yields an  
indirect proof of non-tangential limits of solutions of (\ref{eq:Laplacein1order}) 
which hold for the solutions of (\ref{eq:divform}) in \cite{AAAHK}. 
We do not know how to prove this fact directly in the framework of 
this article.
Note also that we prove here that the non-tangential maximal functions have 
comparable $L_2$-norms for different values of the parameters $c_0$ and $c_1$,
and that the slightly different non-tangential maximal function used in
\cite{AAAHK} therefore has comparable norm.

Before turning to the method of proof for Theorem~\ref{thm:main}, we would like
to stress the importance of the final result that the solution operators 
$g\mapsto F$ and $u\mapsto U$ depend Lipschitz continuously on $L_\infty$ changes
of the matrix $A$ around $A_0$. This is an important motivation for considering complex
$A$, as the authors do not know any proof of this perturbation result which does not
make use of boundedness of the operators in a complex neighbourhood of $A_0$.
We also remark that we in fact prove that $A\mapsto S_A$ is holomorphic, 
from which we deduce Lipschitz continuity as a corollary.

In this paper, we shall mainly focus on the boundary value problems (Neu-$A$) and (Reg-$A$). 
The reason is that, assuming the Cauchy operator $E_A$ is bounded, we prove
in section~\ref{section2.6} that well posedness follows as 
$$
  \text{(Reg-$A^*$)} \quad\Longleftrightarrow\quad \text{(Neu$^\perp$-$A$)} \quad\Longrightarrow\quad \text{(Dir-$A$)}.
$$
That (Reg-$A^*$) implies (Dir-$A$) has been proved by Kenig and Pipher~\cite[Theorem 5.4]{KP} in the case of 
real matrices $A$.

\subsection{Operators and vector fields}   \label{section1.1}

We now explain the basic ideas of the method we use for
the proof of Theorem~\ref{thm:main}. 
The appropriate Hilbert space on the boundary $\R^n$ is
$$
  \hat\mH^1 := \sett{f\in L_2(\R^n;\C^{n+1})}{\curl_x(f_\ta)=0}.
$$
The condition on $f$ means that its tangential part is curl free.
Indeed, the trace $f(x)$ of a vector field $F(t,x)$ 
solving (\ref{eq:Laplacein1order}) belongs to $\hat\mH^1$ 
due to the second equation in (\ref{eq:Laplacein1order})
The basic picture, building on ideas from \cite{Ax1}, is that the Hilbert space 
splits into two different pairs of complementary subspaces as
\begin{equation}  \label{eq:thetwosplittings}
  \hat\mH^1 = E^+_A \hat\mH^1 \oplus E^-_A \hat\mH^1 = N^+_A \hat\mH^1 \oplus N^-_A \hat\mH^1.
\end{equation}
We first discuss the splitting into the Hardy type subspaces $E^\pm_A \hat\mH^1$,
consisting of $L_2$ boundary traces of vector fields $F^\pm$ solving 
(\ref{eq:Laplacein1order}) in $\R^{n+1}_\pm$ respectively.
Our main work in this paper is to establish boundedness of the projection
operators $E_A^\pm$ for certain $A$.
These projections can be written $E_A^\pm= \tfrac 12(I\pm E_A)$, where
$E_A$ for simple $A$ is a singular integral operator of Cauchy type.
However, in the general case $E_A$ may fail to be a singular integral operator. 
To handle the projections $E_A^\pm$ we make use of functional calculus of closed Hilbert space operators,
and show that $E_A^\pm = \chi_\pm (T_A)$ are the spectral projections of 
an underlaying bisectorial operator $T_A$ in $\hat\mH^1$.
The functions $\chi_\pm(z)$ are the characteristic functions for the
right and left complex half planes.
To find $T_A$, assume $F(t,x)$ satisfies (\ref{eq:Laplacein1order}) in $\R^{n+1}_+$ and 
solve for the vertical derivative
\begin{align*}
  \partial_t F_0 &= -a_{00}^{-1}\Big( \sum_{i=1}^n a_{0i}\partial_i F_0
  +\partial_i(AF)_i \Big), \\
  \partial_t F_i &= \partial_i F_0, \qquad i=1,\ldots,n.
\end{align*}
The right hand side defines an operator 
$-T_A$ in $\hat\mH^1$ which on $F(t,x)$, for fixed $t>0$, satisfies
$$
  \pd_t F + T_A F =0.
$$
Concretely, if we identify $f=f_0 e_0 + f_\ta$ with $(f_0, f_\ta)^t$, 
where $f_\ta$ is a tangential curl free vector field, then
\begin{equation}   \label{eq:blockTA}
     T_A f=
     \begin{bmatrix}
        A_{00}^{-1}((A_{0\ta},\nabla_x)+\divv_x A_{\ta 0} ) &
        A_{00}^{-1}\divv_x A_{\ta\ta} \\
        -\nabla_x & 0
     \end{bmatrix}
     \begin{bmatrix}
       f_0 \\ f_\ta
     \end{bmatrix},
\end{equation}
where
$\dom(T_A)=\sett{f=(f_0,f_\ta)^t\in\hat\mH^1}{\nabla_x f_0\in L_2, \divv_x(Af)_\ta\in L_2}$
and
$A= \begin{bmatrix}
       A_{00} & A_{0\ta} \\ A_{\ta 0} & A_{\ta\ta}
     \end{bmatrix}
$.
If $F(t,x)$ is a vector field in $\R_+^{n+1}$ satisfying (\ref{eq:Laplacein1order}),
then using this operator $T_A$, we can reproduce $F$ provided we know the full trace $f=F|_{\R^n}$,
through a Cauchy type reproducing formula
$F(t,x)=(e^{-t |T_A|}f)(x)$.
However, in (Neu-$A$) and (Reg-$A$) only ''half'' of the trace $f$ is known since
the boundary conditions for $f$ are $e_0\cdot Af=\phi$ and 
$e_0\wedg f= e_0\wedg \nabla \psi$ respectively.

We now turn to the second splitting in (\ref{eq:thetwosplittings}), which 
is used to split the boundary trace $f$ into the regularity and Neumann data.
We define the $A$-tangential and normal subspaces of $\hat\mH^1$ to be the 
null spaces of these two operators.
\begin{align*}
  N_A^+\hat\mH^1 &:= \sett{f\in\hat\mH^1}{e_0\cdot Af = 0} \\
  N_A^-\hat\mH^1 &:= \sett{f\in\hat\mH^1}{e_0 \wedg f = 0}
\end{align*}
In contrast with the Hardy subspaces, it is straightforward to show that
we have a topological splitting $\hat\mH^1 = N^+_A \hat\mH^1 \oplus N^-_A \hat\mH^1$,
and therefore that the corresponding pair of projections $N^\pm_A$ are
bounded. 
We can now reformulate (Neu-$A$) and (Reg-$A$) as follows.
The Neumann problem (Neu-$A$) being well posed means
that the restricted projection
$$
  N^-_A : E^+_A \hat\mH^1 \longrightarrow N^-_A\hat\mH^1
$$
is an isomorphism, since $N^-_A\hat\mH^1$ is a complement of the null space of
$e_0\cdot A(\cdot)$.
Similarly the regularity problem (Reg-$A$) being well posed means
that the restricted projection
$$
  N^+_A : E^+_A \hat\mH^1 \longrightarrow N^+_A\hat\mH^1
$$
is an isomorphism, since $N^+_A\hat\mH^1$ is a complement of the null space of
$e_0\wedg(\cdot)$. 
Note that what is important here is which subspace $N_A^\pm$ projects along, not
what subspace they project onto.

We shall also find it convenient to use the operators
$E_A:= E_A^+-E_A^-= \sgn(T_A)$ and $N_A:= N_A^+-N_A^-$. 
These are reflection operators, i.e. $E_A^2=I$ and $N_A^2=I$, and 
we have $E_A^\pm= \tfrac 12(I\pm E_A)$ and $N_A^\pm = \tfrac 12(I\pm N_A)$.
\begin{ex}  \label{ex:r2plus}
Let $n=1$ and $A=I$. 
Then the space $\hat\mH^1$ is simply $L_2(\R;\C^2)$ and the fundamental operator
$T_A$ becomes
$$
  T := T_I=
    \begin{bmatrix}
       0 & \tfrac{d}{dx} \\
       -\tfrac{d}{dx} & 0
     \end{bmatrix}
     \approx
     \begin{bmatrix}
       0 & i\xi \\
       -i\xi & 0
     \end{bmatrix},
$$
if $\approx$ denotes conjugation with Fourier transform.
Furthermore
$$
  E :=E_I=\sgn(T)=
    \begin{bmatrix}
       0 & i H \\
       - iH & 0
     \end{bmatrix},
     \qquad
     N:=N_I= 
     \begin{bmatrix}
       -1 & 0 \\
       0 & 1
     \end{bmatrix},
$$
where $Hf(x) := \tfrac i\pi\text{p.v.}\int\frac{f(y)}{x-y} dy$.
Note that the operator $E$ is contained in the Borel functional calculus of the
self-adjoint operator $T$, and it follows that $\|E\|=1$.
On the other hand the operator $N$ is outside the Borel functional 
calculus of $T$.
Indeed, the operators $b(T)$ in the Borel functional calculus of $T$ all commute,
but we have the anticommutation relation 
\begin{equation}  \label{perfecttransv}
  EN+ NE =0.
\end{equation}
A consequence of this equation is that the boundary value problems 
(Neu-$I$) and (Reg-$I$) are well posed, or equivalently that 
$N^-: E^+L_2\rightarrow N^-L_2$ and $N^+: E^+L_2\rightarrow N^+L_2$
are isomorphisms.
Consider for example (Reg-$I$) and assume we want to
solve $N^+ f=g$ for $f\in E^+ L_2$.
Applying $4E^+$ to the equation gives
$$
  4E^+g= (I+E)(I+N)f= f+Ef+Nf+ENf= f+f+Nf-Nf=2f,
$$
so $f= 2E^+g$ and it follows that $N^+ : E^+ L_2 \rightarrow N^+ L_2$
is an isomorphism.
Having solved for $f\in E^+L_2$, we can find the solution $F(t,x)= C^+_t f(x)$ 
in $\R^2_+$ by using
the Cauchy extension $C_t^+ := e^{-t|T|}E^+$ for $t>0$. 
As a convolution operator, the Fourier multiplier $C^+_t$ has the expression
\begin{multline*}
  C_t^+(u_0e_0+u_1e_1) = 
  \left( \frac 1{2\pi}\int_\R\frac{t u_0(y)-(x-y)u_1(y)}{t^2+(x-y)^2}dy  \right)e_0 \\
   +\left( \frac 1{2\pi}\int_\R\frac{(x-y) u_0(y)+tu_1(y)}{t^2+(x-y)^2}dy  \right)e_1,
\end{multline*}
and in particular
$$
  F(t,x)= C^+_tf(x) = 
  2 C_t^+ (g_1 e_1) = \frac 1{\pi}\int_\R\frac{(-(x-y)e_0+te_1)g_1(y)}{t^2+(x-y)^2}dy. 
$$
\end{ex}
\vspace{5mm}
For a more general $A$, even if $A$ is real symmetric, the operator $T_A$ is not self-adjoint
and proving boundedness of $E_A^\pm$ is a highly non trivial problem.
Also the equation (\ref{perfecttransv}) fails for general $A$, and therefore such explicit formulae 
for the solution $F(t,x)$ as in Example~\ref{ex:r2plus} are not available. 
However, to show well posedness of (Neu-$A$) and (Reg-$A$) it suffices to show
that $I\pm\tfrac 12(E_AN_A+N_AE_A)=\tfrac 12(E_A\pm N_A)^2$ are invertible, as
explained in \cite{Ax1}.
To summarise, in order to solve (Neu-$A$) and (Reg-$A$) we need
\begin{itemize}
\item[{\rm (i)}]
that the Hardy projections $E_A^\pm$ are bounded, so that we have a topological
splitting $\hat\mH^1=E_A^+\hat\mH^1 \oplus E_A^-\hat\mH^1$, and
\item[{\rm (ii)}]
that the restricted projections
$N_A^+: E_A^+\hat\mH^1 \rightarrow N_A^+\hat\mH^1$ and
$N_A^-: E_A^+\hat\mH^1 \rightarrow N_A^-\hat\mH^1$ are isomorphisms.
\end{itemize}
To prove Theorem~\ref{thm:main}, we shall prove the following.
\begin{itemize}
\item[{\rm (i')}]
That $T_A$ satisfies quadratic estimates for all $A$ such that $\|A-A_0\|_\infty<\epsilon$.
From this it will follow that $E_A^\pm$ are bounded for all such $A$ and 
that $A\mapsto E_A^\pm$ is continuous (in fact holomorphic).
\item[{\rm (ii')}]
That (ii) holds for the unperturbed $A_0$. From (i') and since clearly $N_A^\pm$ are bounded and
depend continuously on $A$, it then follows from a continuity argument that (ii) holds for
all $A$ in a neighbourhood of $A_0$.
\end{itemize}
We emphasise that boundedness of the Hardy projections $E_A^\pm$ alone does not show
that (Neu-$A$) and (Reg-$A$) are well posed. 
In our framework, the boundary value problems being well posed means that the Hardy space
$E_A^+\hat\mH^1$, which exists as a closed subspace when the Hardy projections are bounded,
is transversal to the $A$-tangential and normal subspaces $N^\pm_A\hat\mH^1$.
A concrete example showing that (Neu-$A$) and (Reg-$A$) may fail to be well posed, even though
$E_A$ is bounded, is furnished by the matrix
$$
A(x)=
\begin{bmatrix}
  1 & k\, \sgn(x) \\
  -k\, \sgn(x) & 1
\end{bmatrix},
$$
with parameter $k\in \R$.
The corresponding elliptic equation (\ref{eq:divform}) in $\R_+^2$ was studied by Kenig, Koch, Pipher
and Toro in \cite{KKPT}, where they showed that (Dir-$A$) fails to be well posed
for certain values of $k$.
Moreover, that (Neu-$A$) and (Reg-$A$) also fails 
for some $k$, is shown by Kenig and Rule~\cite{KR}.
On the other hand, $\|E_A\|=1$ for all $k\in\R$ since according to (\ref{eq:blockTA})
$$
  T_A=
\begin{bmatrix}
  k(\sgn(x) \,\tfrac{d}{dx}-\tfrac{d}{dx}\, \sgn(x)) & \tfrac{d}{dx} \\
  -\tfrac{d}{dx} & 0
\end{bmatrix}
$$
is self-adjoint, and therefore $E_A^\pm$ are orthogonal projections.
\subsection{Embedding in a Dirac equation}   \label{section:Diracintro}

Unfortunately there is a technical problem in applying harmonic analysis to the operator $T_A$
in order to prove (i'):
the space $\hat\mH^1$ is defined through the non local condition $\curl_x(f_\ta)=0$.
This prevents us from using multiplication operators, for example when localising with a
cutoff $f\mapsto \eta f$, as these does not preserve $\hat\mH^1$.
To avoid this problem we embed $\hat\mH^1\subset\mH := L_2(\R^n; \wedge_\C\R^{n+1})$, where
$$
  \wedge_\C\R^{n+1}= \wedge^0\oplus \wedge^1\oplus \wedge^2\oplus\ldots\oplus\wedge^{n+1}
$$
is the full complex exterior algebra of $\R^{n+1}$, which in particular contains the
vectors $\wedge^1= \C^{n+1}$ and the scalars $\wedge^0=\C$ along with all $k$-vectors $\wedge^k$.
(We identify $k$-vectors with the dual $k$-forms in euclidean space.)
In this way we obtain closure, i.e. all operators, including multiplication operators
preserve $\mH$.
Furthermore we embed the equation (\ref{eq:Laplacein1order}) in a Dirac type equation
\begin{equation}  \label{eq:diractype}
  \big( d_{t,x} + \wt B(x)^{-1} d_{t,x}^* B(x) \big) F(t,x) =0, 
\end{equation}
where $B$, $\wt B\in L_\infty(\R^n; \mL(\wedge_\C \R^{n+1}))$ are accretive and
$B|_{\wedge^1}=A$.
Here $d_{t,x}$ and $d_{t,x}^*$ are the exterior and interior derivative operators,
as defined in section~\ref{section2}. 
In particular, if $F(t,x):\R^{n+1}_+ \rightarrow \wedge^1$ is a vector field 
satisfying (\ref{eq:Laplacein1order}) then
\begin{align*}
  d_{t,x}^* (BF) = -\divv_{t,x} (AF) &= 0, \\
  d_{t,x} F = \curl_{t,x} F &= 0.
\end{align*}
Thus (\ref{eq:diractype}) follows from (\ref{eq:Laplacein1order}) for any choice
of the auxiliary function $\wt B$.
In the same way as in Section~\ref{section1.1} we shall solve for the vertical derivative in
(\ref{eq:diractype}) and obtain an operator $T_B$ acting in $\mH$, such that
$T_A= T_B|_{\hat\mH^1}$.
For applications to the Neumann and regularity problems it suffices to consider 
$B= I\oplus A\oplus I\oplus\ldots\oplus I$.
The stability result for the operator $T_A$ we prove in this paper is the following.
\begin{thm}  \label{thm:mainpert}
Let $A_0$ be a $t$-independent, complex,
accretive coefficient matrix function such that $T_{B_0}$
has quadratic estimates in $\mH$, where
$B_0= I\oplus A_0\oplus I\oplus\ldots\oplus I$.
Then there exists $\epsilon_0>0$ depending only on 
$\|A_0\|_\infty$, $\kappa_{A_0}$ and $n$, such that if 
$A\in L_\infty(\R^n;\mL(\C^{n+1}))$ is $t$-independent and satisfies 
$\|A-A_0\|_{\infty} <\epsilon_0$, then $T_B$ has quadratic estimates
in $\mH$, where $B= I\oplus A\oplus I\oplus\ldots\oplus I$.
In particular $T_A$ has quadratic estimates in $\hat\mH^1$.

Thus $\hat\mH^1$ splits into Hardy subspaces, the spectral subspaces of $T_A$,
i.e. each $f\in \hat\mH^1$ can be 
uniquely written $f= f^++ f^-$, where $f^\pm = F^\pm|_{\R^n}$ and
$F^\pm(t,x)= e^{\mp t|T_A|}E^\pm_A f(x)$ satisfies (\ref{eq:Laplacein1order})
in $\R^{n+1}_\pm$. Moreover, we have equivalence of norms
$$
\|f\|_2\approx \|f^+\|_2+\|f^-\|_2
$$ 
and
$\|f^\pm \|_2\approx \sup_{t>0} \|F^\pm_{\pm t}\|_2
  \approx \tb{t\partial_t F^\pm_{\pm t} }
  \approx \|\widetilde N_* (F^\pm)\|_2$.
\end{thm}
With a more general choice 
$$
  B=B^0\oplus B^1\oplus B^2\oplus\ldots\oplus B^{n+1},
$$
where $B^k \in L_\infty(\R^n;\mL(\wedge^k))$,
we also obtain new perturbation results concerning boundary value problems and
more generally transmission problems for $k$-vector fields.
For $k$-vector fields we consider the function space
$$
  \hat\mH^k_B := \sett{ f\in L_2(\R^n;\mL(\wedge^k)) }
    {d_x(f_\ta) = 0 = d_x^*((B^k f)_\no) }\subset\mH,
$$
where $f_\ta$ and $f_\no$ denote the tangential and normal parts of $f$.
This is the appropriate function space since traces of $k$-vector fields $F(t,x)$ 
in $\R^{n+1}_\pm$ satisfying
\begin{equation} \label{eq:diracwedgek}
\begin{cases}   
  d^*_{t,x} (B^k(x) F(t,x))  =0, \\
  d_{t,x} F(t,x) =0,
\end{cases}
\end{equation}
belong to $\hat\mH^k_B$.
Note that for $1$-vectors, i.e. vectors, the condition $d_x^*((B^1 f)_\no)=0$ is void
and $\hat\mH_B^1=\hat\mH^1$.
For $k$-vector fields, we consider the following.

\vspace{2mm}
\noindent THE TRANSMISSION PROBLEM (Tr-$B^k\alpha^\pm$).

Let $B^k=B^k(x)\in L_\infty(\R^n;\mL(\wedge^k))$ be accretive and let
$\alpha^\pm\in\C$ be given jump parameters.
Given a $k$-vector field $g\in \hat\mH^k_B$,
find $k$-vector fields 
$F_t^+(x)= F^+(t,x)$ in $\R^{n+1}_+$ and $F_t^-(x)= F^-(t,x)$ in $\R^{n+1}_-$
such that $F^\pm_t\in C^1(\R_+; L_2(\R^n;\wedge^k))$ and
$F^\pm$ satisfies (\ref{eq:diracwedgek}) for $\pm t>0$,
and furthermore $\lim_{t\rightarrow \pm\infty}F^\pm_t=0$ and 
$\lim_{t\rightarrow 0^\pm}F^\pm_t=f^\pm$
in $L_2$ norm, 
where the traces $f^\pm$ satisfy the jump conditions
\begin{equation*}  
  \left\{
  \begin{array}{rcl}
    e_0\wedg (\alpha^- f^+ - \alpha^+ f^-) &=& e_0\wedg g \\
    e_0\lctr\big(B^k(\alpha^+ f^+ - \alpha^- f^-) \big) &=& e_0\lctr(B^k g).
  \end{array}
  \right.
\end{equation*}

\vspace{2mm}
The second main result of this paper is the following $L_\infty$ perturbation result 
of the transmission problem (Tr-$B^k\alpha^\pm$).
\begin{thm}  \label{thm:transmain}
Let $B_0^k=B_0^k(x)\in L_\infty(\R^n;\mL(\wedge^k))$ be $t$-independent, accretive and possibly complex, and
assume that $B_0^k$ is a block matrix, i.e. 
$$
  (B_0^k)_{st}= 0,\qquad \text{whenever } 
  0\in (s\setminus t)\cup (t\setminus s)
$$
and $s,t\subset\{0,1,\ldots, n\}$ has lengths 
$|s|=|t|=k$.
Then there exists $\epsilon>0$ and $C<\infty$ depending only on 
$\|B^k_0\|_\infty$, the accretivity constant $\kappa_{B^k_0}$ 
and dimension $n$, such that if 
$B^k=B^k(x)\in L_\infty(\R^n;\mL(\wedge^k))$ is $t$-independent and satisfies
\begin{equation}  \label{eq:specestB}
  \|B^k-B^k_0\|_{L_\infty(\R^n)} < \min\big(\epsilon ,
   C| (\alpha^+/\alpha^-)^2 + 1 |\big),
\end{equation}
then the transmission problem (Tr-$B^k\alpha^\pm$) above
is well posed in the sense that for every boundary function $g\in\hat\mH^k_B$, 
there exists unique $k$-vector fields $F^\pm(t,x)$ with properties as in
(Tr-$B^k\alpha^\pm$).
The solution $F^\pm_t$ depends continuously on $g$
with equivalences of norms $\|g\|_2\approx \|f^++ f^-\|_2 \approx \|f^+\|_2+\| f^-\|_2$
and
$$
  \|f^\pm\|_2 \approx 
  \sup_{t>0} \|F_{\pm t}^\pm\|_2 
  \approx  \tb{t\partial_t F^\pm_{\pm t} }.
$$
Moreover, the solution operators $S_{B^k}(g)= F$ 
depend Lipschitz continuously on $B^k$, i.e. there exists $C<\infty$ such that 
$$
  \|S_{B^k_2}-S_{B^k_1}\|_{L_2(\R^n)\rightarrow \mX}\le C \|B^k_2-B^k_1\|_{L_\infty(\R^n)}
$$
when (\ref{eq:specestB}) holds for $B^k_1$ and $B^k_2$, where $\|F\|_\mX$ 
denotes one of the two norms above.
\end{thm}
This perturbation theorem for transmission problems for $k$-vectors
has two important corollaries. 
On one hand it specialises when $k=1$ to a generalisation of Theorem~\ref{thm:main}(b), giving perturbation results for 
transmission problems across $\R^n$ for the divergence form
equation (\ref{eq:divform}). 
The details of this Neumann--regularity transmission problem is 
stated as (Tr-$A\alpha^\pm$) in section~\ref{section8.3}.

On the other hand it specialises when either $\alpha^+=0$ or
$\alpha^-=0$ to a generalisation of Theorem~\ref{thm:main}(b),
giving perturbation results for boundary value problems for
$k$-vectors.
Our result for these boundary value problems (Nor-$B^k$) and
(Tan-$B^k$) is given as Corollary~\ref{cor:formBVP} in Section~\ref{section8.3}.

\begin{ex}  \label{ex:diraclip}
In the case $\widetilde B=B$, operators of the form $d_{t,x} +B^{-1} d^*_{t,x} B$
appear naturally when pulling back the unperturbed Hodge--Dirac operator $d+d^*$
with a change of variables. As above, consider the region $\Omega$ above a Lipschitz graph
$\Sigma=\sett{(t,x)}{t=g(x)}$. We define the pullback of the field 
$F:\Omega\rightarrow\wedge\R^{n+1}$ to be the field
$$
  (\rho^*F)(t,x):=  \underline \rho^T(x)F(\rho(t,x))
$$
in $\R^{n+1}_+$, where $\rho(t,x)=(t+g(x),x)$ is the parametrisation of $\Omega$, having 
differential
$$
  \underline\rho(x)|_{\wedge^1}=
\begin{bmatrix}
  1 & \nabla g(x) \\
  0 & I
\end{bmatrix}
$$
acting on vector fields and extended naturally to $\wedge\R^{n+1}$, and 
$\underline \rho^T(x)$ denotes the transposed matrix.
From the well known fact that $d_{t,x}$ commutes with $\rho^*$, we get the intertwining relation
\begin{equation}  \label{eq:Gdirac}
  (d_{t,x}+ G d_{t,x}^* G^{-1})(\rho^* F)= \rho^* ((d_{t,x}+d_{t,x}^*)F),
\end{equation}
where $G(x)=(g_{ij}(x))= \underline \rho^T(x)\underline \rho(x)$ is the metric 
for the parametrisation, being real symmetric.
Solving for the vertical derivative in the equation $(d_{t,x}+d_{t,x}^*)F=0$ in $\Omega$,
gives us an operator $\dirac_\Sigma$ in $L_2(\Sigma;\wedge\R^{n+1})$, which 
is similar to the operator $(e_0-\nabla g(x))^{-1}(d_x+d_x^*)$ in $\mH$.
From (\ref{eq:Gdirac}) it follows that
$$
  T_{G^{-1}}(\underline\rho(x)^T f(x)) = \underline\rho(x)^T (\dirac_\Sigma f)(x).
$$
It is known that the operator $\dirac_\Sigma$ satisfies quadratic estimates, and
therefore so does $T_{G^{-1}}$. For references and further discussion of $\dirac_\Sigma$, see
\cite[Consequence 3.6]{AKMc}.
From this we get the bounded Clifford--Cauchy singular integral operator
\begin{multline*}
      E_{G^{-1}}( \underline\rho(x)^T  f(x)) \\
= \underline\rho(x)^T \frac2{\sigma_n}
\text{p.v.}\int_{\R^n}\frac{(g(x)e_0+x)-(g(y)e_0+y)}{(|y-x|^2+(g(y)-g(x))^2)^
{(n+1)/2}} (e_0-\nabla g(y))f(y)dy,
\end{multline*}
where $\sigma_n$ is the area of the unit $n$-sphere in
$\R^{n+1}$ and $E_{G^{-1}}= \sgn(T_{G^{-1}})$.
For this reason, we shall refer to $E_B=\sgn(T_B)$ as generalised Cauchy integral
operators and $E_B^\pm= \chi_\pm(T_B)$ as generalised Hardy projection operators,
also when $B\ne G^{-1}$ and $\wt B\ne B$.
\end{ex}

\subsection{Outline and remarks}

In section~\ref{section2} we explain how we use the exterior algebra
$\wedge_\C \R^{n+1}$ and the exterior and interior 
derivative operators $d$ and $d^*$.
In section~\ref{section2.1} we introduce the Dirac type operator
$T_B$ which extends $T_A$ to the full exterior algebra
as well as projection operators $N_B^\pm$ which extend the
$A$-tangential and normal projections $N_A^\pm$ from above.
Section~\ref{section2.3} is concerned with the spectral properties
of $T_B$, where we prove that $T_B$ is a bisectorial operator
and that the resolvents $(\lambda I- T_B)^{-1}$ has $L_2$ 
off-diagonal estimates.
Section~\ref{section2.4} surveys the theory of functional calculus of
bisectorial operators like $T_B$.
In section~\ref{section2.6}, Lemma~\ref{lem:characthardyfcns} and
\ref{lem:characterisePoisson} characterise the classes of 
solutions $F_t$ and $U_t$ respectively to the boundary value problems, 
and are used in particular to prove uniqueness.

In section~\ref{section7} we prove (i') quadratic estimates and (ii')
invertibility in the unperturbed case $B=B_0$.
The most involved case is when $B_0$ is real symmetric.
In order to prove the quadratic estimates we use the 
results from section~\ref{section2.5} that $T_{B_0}$ leaves
certain subspaces $\hat\mH^k_{B_0}$ and $\hut\mH_{B_0}$ 
invariant and that therefore it suffices to establish 
quadratic estimates in each subspace separately.

Section~\ref{section8} contains the core harmonic analysis
results of the paper, where we establish quadratic estimates
for $T_B$ when $B\approx B_0$ through a perturbation argument
based on equation (\ref{eq:good1})-(\ref{eq:good6}).
Section~\ref{section8.1} treats the case when $B_0$ is a block
matrix and makes use of techniques from the solution of
the Kato square root problem~\cite{AHLMcT}.
In Lemma~\ref{lem:newtestfcns} we construct a new set of test 
functions $f^w_Q$ which also can be used in the Carleson measure 
estimate in \cite{AKMc} to simplify the proof there.
Section~\ref{section8.2} treats matrices of the form
$B= I\oplus A\oplus I \oplus \ldots\oplus I$.
The key technique is Lemma~\ref{lem:wedge1formulae}
where we compare $B_0$ with a corresponding block matrix
$\hB_0$ for which the results from section~\ref{section8.1} applies.
The reason why this approximation $B_0\approx \hB_0$ works 
is that the normal vector component $F^{1,0}$ has additional
regularity by Lemma~\ref{lem:LaplaceF_0^1} when $F$ satisfy
the Dirac type equation (\ref{eq:diractype}).

The paper ends with section~\ref{section8.3}, where we bring the
results together and prove Theorems~\ref{thm:transmain},
\ref{thm:main} and \ref{thm:mainpert}.

\vspace{3mm}
The reader who is interested only in Theorems~\ref{thm:main} and
\ref{thm:mainpert} may ask whether it is necessary to work within
the full exterior algebra as described in section~\ref{section:Diracintro}.
Indeed, instead of proving resolvents bounds for $T_A$ through
$T_B$ as done in Proposition~\ref{prop:spectrest}, a direct proof can be based on
(\ref{eq:blockTA}). To prove a priori estimates $\|u\|\lesssim \|f\|$
when $(I-i\tau T_A)u=f$, we multiply this equation with the 
diagonal matrix 
$
     \begin{bmatrix}
       -a_{\no\no} & 0 \\ 0 & a_{\ta\ta}
     \end{bmatrix}
$
and obtain the identity
\begin{equation}   \label{eq:reminintro}
  \left(  A  \begin{bmatrix}
       0 & 0 \\ i\tau\nabla & I
     \end{bmatrix} +
     \begin{bmatrix}
       -1 & i\tau\divv \\ 0 & 0
     \end{bmatrix} A
  \right)u =
     \begin{bmatrix}
       -a_{\no\no} & 0 \\ 0 & a_{\ta\ta}
     \end{bmatrix} f.
\end{equation}
With an argument as in Lemma~\ref{lem:hodge}(ii), a priori estimates
$\|u\|\lesssim \|f\|$
can be deduced from this by proving that the ranges of the two differential systems 
on the left hand side in (\ref{eq:reminintro}) are complementary subspaces in
$\hat\mH^1$, and similarly for the null spaces.

The problems when working within $\hat \mH^1$ starts when making use of 
multiplication operators, for example when proving off-diagonal
bounds in Proposition~\ref{pseudoloc} and in the perturbation formula
(\ref{eq:good1})-(\ref{eq:good6}), but in particular 
in the proof of Theorem~\ref{thm:pertestsforrealsymm}.
It is not clear to the authors how to avoid using at least part of the full
exterior algebra $\wedge_\C \R^{n+1}$ at these points.
Proving off-diagonal bounds for the resolvents on $\hat\mH^1$ in 
Proposition~\ref{pseudoloc} seems to require having $(I+it T_B)^{-1}$ 
defined as a bounded operator on all vector fields.
Also, in (\ref{eq:good1})-(\ref{eq:good6}) the last factors in each term
preserve $\hat\mH^1$, whereas the middle factor only preserve vector fields,
not the curl freeness. Therefore we need to have the first factors defined 
on all vector fields, and moreover off-diagonal bounds are needed here.
Going back to Proposition~\ref{pseudoloc} this seems to require having 
$(I+itT_B)^{-1}$ defined as a bounded operator also on $\wedge^3$.

More importantly, to derive quadratic estimates for the perturbed operator $T_A$ on $\hat\mH^1$
from (\ref{eq:good1})-(\ref{eq:good6}), we must assume quadratic estimates
for $T_{B_0}$, where $B_0= I\oplus A_0\oplus I \oplus\ldots\oplus I$, on all
$\mH$. The reason is that when proving Theorem~\ref{thm:pertestsforrealsymm},
it does not suffice to consider $f\in\hat\mH^1$ as the term 
$\tb{\widetilde Q_t}_\op$ on the right hand side is the norm on a 
bigger space than $\hat\mH^1$, since $\ep$ does not preserve this space.
This explains the hypothesis in Theorem~\ref{thm:mainpert}, as well as the
necessity of proving quadratic estimates on all $\mH$ in 
section~\ref{section7.3}.

We also point out that for $k$-forms, the situation is worse as the spaces
$\hat \mH^k_B$ depend on $B$ for $k\ge 2$, which makes it impossible to
perform perturbation $B\approx B_0$ within this space.

{\bf Acknowledgements.}
This work was mainly conducted at the University of Orsay
while Axelsson was post-doctoral fellow supported by the
French Ministry of Teaching and Research (Minist\`ere de  
l'Enseignement et de la Recherche) and then by the European
IHP network ''Harmonic Analysis and Related Problems'' 
(contract HPRN-CT-2001-00273-HARP). 
Axelsson gratefully acknowledges the University of Orsay
for providing facilities to work. 
Hofmann also thanks the University of Orsay for partial support of 
a one week visit that jump started this project.

%
%
%
\section{Operator theory and algebra}   \label{section2}

In this section we develop the operator theoretic framework we
use to prove the perturbation theorems stated in the introduction.
In particular we introduce our basic operator $T_B$, along with 
perturbations of the normal and tangential projection operators
$N^-$ and $N^+$. These all act in the Hilbert space 
$\mH := L_2(\R^n;\wedge_\C \R^{n+1})$ on the boundary 
$\R^n = \partial \R^{n+1}_+ = \partial \R^{n+1}_-$.
In $\R^{n+1}$ we write the standard ON basis as
$\{ e_0, e_1,\ldots, e_n \}$, where $e_0$ denote the vertical direction
and $e_1,\ldots,e_n$ span the horizontal hyperplane $\R^n$. 
We write the corresponding coordinates as $x_0,x_1,\ldots,x_n$ and we
also use the notation $t=x_0$.
The corresponding partial derivatives we write as $\partial_i= \tdd{}{x_i}$
and $\partial_t= \partial_0= \tdd{}t$.
Our functions $f\in\mH$ take values in the full complex exterior 
algebra over $\R^{n+1}$
$$
  \wedge=\wedge_\C\R^{n+1}=\wedge^0\oplus \wedge^1\oplus\ldots\oplus \wedge^{n+1}.
$$
This is a $2^{n+1}$ dimensional linear space with $n+2$
pairwise orthogonal subspaces $\wedge^k$ of dimensions 
$\binom{n+1}k$.
With the notation 
$e_s:=e_{s_1}\wedg \ldots \wedg e_{s_k}$ if $s=\{s_1,\ldots,s_k\}\subset\{0,1,\ldots,n\}$
and $s_1<s_2<\ldots <s_k$,
the space $\wedge^k$ of homogeneous $k$-{\em vectors} is the linear span of 
$\sett{e_s}{|s|=k}$. In particular, identifying $e_\emptyset$ with $1$ and the singleton set $\{j\}$ with $j$, we have $\wedge^0=\C$ and $\wedge^1= \C^{n+1}$.
A general element in $\wedge$ is called a {\em multivector} and is a direct sum
of $k$-vectors of different degrees $k$.

\begin{defn}  \label{defn:products}
Introduce the sesqui linear scalar product 
$$
  (f,g)= \left(\sum_s f_s e_s, \sum_t g_t e_t\right) = 
  \sum_s f_s \conj g_s,
$$
on $\wedge$ and the bilinear scalar product 
$f \cdot g = \sum_s f_s g_s$.
Define the counting function 
$
  \sigma(s,t):=\#\sett{(s_i,t_j)}{s_i>t_j},
$
where $s=\{s_i\}$, $t=\{t_j\}\subset\{0,1,\ldots,n\}$.
\begin{itemize}
\item[{\rm (i)}]
  The {\em exterior product} $f\wedg g$ is the complex bilinear product for
  which
$$
  e_s\wedg e_t=(-1)^{\sigma(s,t)}\, e_{s\cup t}\quad
         \text{if } s\cap t=\emptyset \text{ and otherwise zero.}
$$
\item[{\rm (ii)}]
 The (left) {\em interior product} $f\lctr g$ is the complex bilinear product for which  
$
  (e_s\lctr e_t,e_u)=( e_t,e_s\wedg e_u)
$
for all $s,t,u\subset\{0,1,\ldots,n\}$.
Explicitly we have
$$
   e_s\lctr e_t=(-1)^{\sigma(s,t\setminus s)}\, e_{t\setminus s}\quad
         \text{if } s\subset t \text{ and otherwise zero.}
$$
\end{itemize}
\end{defn}
\begin{ex}  \label{ex:mappingforproducts}
The most common situation is when forming a product between a vector
$a= \sum_{j=0}^n a_j e_j\in \wedge^1$ and a $k$-vector 
$$
  f= \sum_{0\le s_1 <\ldots <s_k\le n} f_{s_1,\ldots, s_k} \,
    e_{s_1}\wedg \ldots \wedg e_{s_k}.
$$ 
In this case the interior product $a\lctr f$ is a
$(k-1)$-vector, whereas the exterior product $a\wedg f$ 
yields a $(k+1)$-vector.
Note also that we embed all different spaces $\wedge^k$ of homogeneous
$k$-vectors as pairwise orthogonal subspaces in the $2^{n+1}$ dimensional
linear space $\wedge$.
Thus we may add the two products to obtain (the Clifford product)
$$
  a\lctr f + a\wedg f \in \wedge^{k-1}\oplus \wedge^{k+1}\subset\wedge.
$$
In the special case when $f=b$ is also a vector, i.e. $k=1$, we have
\begin{align*}
  a\lctr b & = \sum_{j=0}^n a_j b_j\in \wedge^0, \\
  a\wedg b & = \sum_{0\le i<j\le n} (a_i b_j-a_j b_i) \, e_i\wedg e_j
  \in\wedge^2,
\end{align*}
where we write $b_{\{j\}}= b_j$.
We see that $a\lctr b$ coincide with the bilinear
scalar product $a\cdot b$. Furthermore, in three dimensions $n+1=3$, 
the exterior product $a\wedg b\in \wedge^2$ can be identified 
with the vector product $a\times b\in\wedge^1$ 
by using the Hodge star identifications
$e_{\{1,2\}}\approx e_0$, $-e_{\{0,2\}} \approx e_1$ and 
$e_{\{0,1\}} \approx e_2$.
\end{ex}

The following anticommutativity, associativity and 
derivation properties of these products summarise the fundamental algebra we shall
need in this paper.

\begin{lem}  \label{lem:algebra}
If $a$, $b\in\wedge^1$ are vectors and $f$, $g$, $h\in\wedge$, then
\begin{align*}
   a\wedg b = -b\wedg a, &\qquad a\wedg a =0, \\
   f\wedg(g\wedg h) =(f\wedg g)\wedg h, &\qquad
   f\lctr( g\lctr h)  = (g\wedg f)\lctr h \\
   a\lctr(b\wedg f) &= (a\cdot b) f- b\wedg(a\lctr f)
\end{align*}
\end{lem}
We shall also frequently use that if $a\in\wedge^1$ is a real vector,
then $(a\lctr f, g) = (f, a\wedg g)$.

\begin{proof}
That $a\wedg b = -b\wedg a$ and $a\wedg a =0$ is readily seen
from Example~\ref{ex:mappingforproducts}.
These and the associativity
$f\wedg(g\wedg h) =(f\wedg g)\wedg h$ are well known 
properties of the exterior product.
To see how $f\lctr( g\lctr h)  = (g\wedg f)\lctr h$
follows, note first that by linearity it suffices to consider the
case when $f$, $g$ and $h$ all are real, i.e. have real coefficients
in the standard basis $\{e_s\}$.
Under this assumption we pair with an arbitrary $w\in\wedge$ and use
that left interior and exterior multiplication are adjoint operations.
We get
\begin{multline*}
  (f\lctr (g\lctr h), w) = (g\lctr h, f\wedg w) = 
  (h, g\wedg(f\wedg w)) \\
  = (h, (g\wedg f)\wedg w)= ((g\wedg f)\lctr h,w).
\end{multline*}
For the derivation identity, by linearity it suffices to prove 
$$
  e_i\lctr (e_j\wedg e_s) + e_j\wedg (e_i\lctr e_s)
  =
  \begin{cases}
    e_s, & i=j, \\
    0, & i\ne j,
  \end{cases}
$$
for all $s\subset \{0,1,\ldots,n\}$.
This is straightforward to verify from
Definition~\ref{defn:products}.
\end{proof}

\begin{defn}  \label{defn:normtang}
For a multivector $f\in\wedge$, we write
$\mu f:= e_0 \wedg f$, $\mu^* f := e_0 \lctr f$ and 
$\m f:= e_0\wedg f + e_0\lctr f= (\mu+\mu^*)f$.
We call 
\begin{align*}
  f_\no= N^- f&:= \mu\mu^* f= \m\mu^* f=\mu \m f, \\
  f_\ta= N^+ f&:= \mu^*\mu f= \m\mu f=\mu^* \m f, 
\end{align*}
the {\em normal} and {\em tangential} parts of $f$ respectively.
\end{defn}
Concretely, if 
$f= \sum_{\{s_1,\ldots, s_k \}} 
f_{s_1,\ldots, s_k} \,  e_{s_1}\wedg \ldots \wedg e_{s_k}$,
then its normal part is
$$
  f_\no = \sum_{ \{s_1,\ldots, s_k \} \ni 0 } 
  f_{s_1,\ldots, s_k} \,  e_{s_1}\wedg \ldots \wedg e_{s_k}
$$
and its tangential part is
$$
  f\ta= \sum_{ \{s_1,\ldots, s_k \} \not\ni 0 } 
  f_{s_1,\ldots, s_k} \,  e_{s_1}\wedg \ldots \wedg e_{s_k}.
$$
In particular $e_s$ is normal if $0\in s$ and tangential if $0\notin s$.
Note that both the subspace of normal multivectors $N^-\wedge$ and
the subspace of tangential multivectors $N^+\wedge$ have dimension
$2^n$, although the subspaces $N^\pm \wedge^k$ have different dimensions
in general.

Throughout this paper, upper case letters denote $\wedge$-valued functions
$F(t,x) = F_t(x)$ in the domain $\R^{n+1}_+$ or $\R^{n+1}_-$ whereas
lower case letters will denote $\wedge$-valued functions $f$ on the 
boundary $\R^n$. 
We use the sesqui linear scalar product 
$(f,g):= \int_{\R^n}(f(x), g(x))\,dx$
on the Hilbert space $\mH$.
\begin{defn}  \label{defn:nablaops}
Using the nabla symbols $\nabla_x= \nabla =\sum_{j=1}^n e_j \pd_j$ in $\R^n$ and
$\nabla_{t,x} =\sum_{j=0}^n e_j \pd_j$ in $\R^{n+1}$,
we define the operators of {\em exterior} and {\em interior} derivation
as
\begin{alignat*}{2}
  d_xf=df &= \nabla\wedg f := \sum_{j=1}^n e_j \wedg \pd_j f, &\qquad
  d_x^*f= d^*f &= -\nabla\lctr f := -\sum_{j=1}^n e_j \lctr \pd_j f, \\
  d_{t,x}F &= \nabla_{t,x}\wedg F := \sum_{j=0}^n e_j \wedg \pd_j F, &\qquad
  d_{t,x}^*F &= -\nabla_{t,x}\lctr F := -\sum_{j=0}^n e_j \lctr \pd_j F.
\end{alignat*}
We shall also find it convenient to use the operators
$\ud:= i\m d$ and $\ud^*= -i d^* \m= i\m d^*$.
\end{defn}
\begin{rem}  \label{rem:mappingud}
Clearly, in the splitting
$\wedge=\wedge^0\oplus \wedge^1\oplus\ldots\oplus \wedge^{n+1}$,
the operators $d$ and $d^*$ map
$d: L_2(\R^n; \wedge^k)\rightarrow L_2(\R^n; \wedge^{k+1})$
and $d^*: L_2(\R^n; \wedge^k)\rightarrow L_2(\R^n; \wedge^{k-1})$
as unbounded operators.
Moreover, if we further decompose the space of homogeneous 
$k$-vectors into its normal and tangential subspaces
as 
$$
  \wedge=\wedge_\ta^0\oplus (\wedge^1_\no \oplus \wedge^1_\ta)\oplus
  (\wedge^2_\no \oplus \wedge^2_\ta)\oplus\ldots\oplus 
  (\wedge^n_\no \oplus \wedge^n_\ta)\oplus
  \wedge_\no^{n+1},
$$
where $\wedge^k_\no:= N^-\wedge^k$ and $\wedge^k_\ta:= N^+\wedge^k$,
then we see that the operators $\ud$ and $\ud^*$ map
\begin{alignat*}{2}
  \ud: L_2(\R^n; \wedge_\no^k)&\longrightarrow L_2(\R^n; \wedge_\ta^k), &
  \qquad
  \ud: L_2(\R^n; \wedge_\ta^k)&\longrightarrow L_2(\R^n; \wedge_\no^{k+2}),
\\
  \ud^*: L_2(\R^n; \wedge_\ta^k)&\longrightarrow L_2(\R^n; \wedge_\no^k), &
  \qquad
  \ud^*: L_2(\R^n; \wedge_\no^k)&\longrightarrow L_2(\R^n; \wedge_\ta^{k-2}).
\end{alignat*}
\end{rem}
\begin{lem}  \label{lem:anticom}
We have, on appropriate domains, $(d_{t,x})^2=(d^*_{t,x})^2=d^2=(d^*)^2=\mu^2=(\mu^*)^2=0$
and the anti commutation relations
\begin{alignat*}{2}
  \{ \m, d_{t,x} \} = \pd_t &=  \{ \m, -d_{t,x}^* \}, &\qquad
  \m^2&=\{ \mu, \mu^* \} = I, \\
  \{ d, d^* \} = -\Delta &= -\sum_1^n \pd_k^2, &\qquad
  \{ d,\mu \} &= \{ d,\mu^* \} = 0,
\end{alignat*}
where $\{A,B\}=AB+BA$ denotes the anticommutator.
\end{lem}
\begin{proof}
The proofs are straightforward, using Lemma~\ref{lem:algebra}.
Let us prove that $\pd_t =  \{ \m,- d_{t,x}^* \}$.
For a function $F(t,x)$ we have
$$
  \nabla_{x,t} \lctr (\m F)= \sum_0^n e_i\lctr\big(\m\pd_i F\big)
  = \sum_0^n e_i\lctr\big(e_0\lctr\pd_i F\big)+
  \sum_0^n e_i\lctr\big(e_0\wedg\pd_i F\big).
$$
Using the anticommutativity and associative properties, the first sum is
$$
  \sum_0^n (e_0\wedg e_i)\lctr\pd_i F= -\sum_0^n (e_i\wedg e_0)
  \lctr\pd_i F =-e_0\lctr \left(\sum_0^n e_i \lctr\pd_i F\right),
$$
whereas using the derivation property, the second sum is
$$
  \sum_0^n e_i\lctr(e_0\wedg\pd_i F)= \pd_0 F-
  e_0\wedg\left( \sum_0^n e_i\lctr\pd_i F\right).
$$
Adding up we obtain 
$\nabla_{x,t} \lctr (\m F)=\pd_t F -\m(\nabla_{x,t} \lctr F)$.
\end{proof}

\subsection{The basic operators}  \label{section2.1}

\begin{defn}  \label{defn:B}
Throughout this paper we denote by $B\in L_\infty(\R^n;\mL(\wedge))$ 
a bounded, accretive and complex matrix function acting on $f\in\mH$
as $f(x)\mapsto B(x)f(x)$, with quantitative bounds $\|B\|_\infty$ and
$\kappa_B>0$, where $\kappa_B$ is the largest constant such that
$$
  \re (B(x)w,w)\ge \kappa_B |w|^2, \qquad \text{for all } w\in \wedge,
  \,\, x\in \R^n.
$$
We shall also assume that $B$ is of the form 
$B= B^0\oplus B^1\oplus B^2\oplus\ldots\oplus B^{n+1}$,
where
$B^k \in L_\infty(\R^n;\mL(\wedge^k))$,
so that $B$ preserve the space of $k$-vectors.
For the matrix part $B^1$ acting on vectors, we use the alternative
notation $A= B^1$.
We also define $\widetilde B:= \m B\m$.
It is not true in general that $\widetilde B$ preserve $k$-vectors.
\end{defn}
\begin{rem}
It would be more optimal to replace the quantitative bounds $\|B\|_\infty$ and
$\kappa_B>0$ with $K_B$ and $k_B$, where $K_B$ and $k_B$ are the optimal
constants such that
$$
  k_B\|f\|^2 \le |(Bf,f)| \le K_B \|f\|^2,  \qquad \text{for all } f\in \mH.
$$
\end{rem}
We consider $F(t,x)$ satisfying the Dirac type equation
\begin{equation}  \label{diraceqn}
  ( \m d_{t,x} + B^{-1} \m d_{t,x}^* B ) F(t,x) =0 
\end{equation}
or equivalently
\begin{equation}  \label{diraceqntilde}
  ( d_{t,x} + \widetilde B^{-1} d_{t,x}^* B ) F(t,x) =0.
\end{equation}

In order to solve for the vertical derivative $\pd_t F$, we note that
$$
  d_{t,x} F = dF + \mu \pd_t F, \qquad
  d^*_{t,x} F = d^* F -\mu^* \pd_t F.
$$
Inserted in (\ref{diraceqn}) this yields
$$
  \pd_t F + \tfrac 1{N^+ - B^{-1}N^-B}(\m d + B^{-1}\m d^* B)F =0.
$$
\begin{defn}  \label{defn:tb}
Write $M_B:= N^+ - B^{-1}N^- B$ and define the unbounded operator
$$
  T_B := M_B^{-1}(\m d + B^{-1}\m d^* B) = 
  -i M_B^{-1}(\ud+B^{-1}\ud^* B)= 
  \tfrac 1{\mu-\widetilde B^{-1}\mu^*B}(d + \widetilde B^{-1} d^* B)
$$
in $\mH$ with domain $\dom(T_B):= \dom(d)\cap B^{-1}\dom(d^*)$.
\end{defn}
A rather surprising fact is that the most obvious choice for $\wt B$, namely 
$\wt B=B$ is not the the best, but rather $\wt B= m B m$.
For example, this is the only choice for which a Rellich type formula,
as in Proposition~\ref{prop:rellich}, holds on all $\mH$, in the case of Hermitean 
coefficients $B$.

Note that $T_B$ is closely related to operators of the form
$\Pi_B= \Gamma+B^{-1}\Gamma^* B$, where $\Gamma$ denotes
a first order, homogeneous partial differential operator 
with constant coefficients such that $\Gamma^2=0$, 
which were studied in \cite{AKMc}.
Unfortunately, the factor $M_B^{-1}$ does not commute with $\Pi_B$
for general $B$.
However, it has other useful commutation properties.

\begin{lem}  \label{lem:magic}
The operator $M_B$ is an isomorphism and 
\begin{alignat*}{2}
  (B^{-1}N^-B)M_B^{-1} &= M_B^{-1} N^-, &\qquad N^- M_B^{-1} &= M_B^{-1}(B^{-1}N^-B), \\
  (B^{-1}N^+B)M_B^{-1} &= M_B^{-1} N^+, &\qquad N^+ M_B^{-1} &= M_B^{-1}(B^{-1}N^+B).
\end{alignat*}
\end{lem}
\begin{proof}
Note that $M_B= B^{-1}(BN^+ - N^- B)$, where the last factor is the diagonal matrix
$(-B_{\no\no})\oplus B_{\ta\ta}$ in the splitting $\mH= N^-\mH \oplus N^+\mH$,
if $B_{\no\no}$ and $B_{\ta\ta}$ denote the diagonal blocks of $B$, 
and thus $BN^+ - N^- B$ commutes with $N^\pm$. 
Furthermore, the diagonal blocks are accretive, so $BN^+ - N^- B$ and thus
$M_B$ is invertible. 
This proves the two equations to the right. 
To obtain the equations to the left, replace $B$ by $B^{-1}$ and note that
$$
  -N^- + B^{-1}N^+ B = N^+ - B^{-1}N^- B.
$$
\end{proof}
Let us comment on the terminology ''Dirac type equation'' for
(\ref{diraceqn}). Normally this denotes a first order differential
operator, like $d_{t,x}+d^*_{t,x}$, whose square acts componentwise as 
the Laplace operator.
In our situation, the following holds.
\begin{lem}   \label{lem:LaplaceF_0^1}
  If $F(t,x)$ satisfies (\ref{diraceqn}), then
$d_{t,x}^* B d_{t,x} (\m F) = 0$.
In particular it follows that 
$$
  \divv_{t,x} A \nabla_{t,x} F^{1,0}=0
$$ 
if
$F= F^0+ (F^{1,0} e_0+F^{1,\ta})+F^2+\ldots+ F^{n+1}$, i.e. $F^{1,0}$ is the
normal component of the vector part of $F$.
\end{lem}
\begin{proof}
We use the anticommutation relations $\m d_{t,x}= \partial_t- d_{t,x}\m$
and $\m d^*_{t,x}= -\partial_t- d^*_{t,x}\m$ from Lemma~\ref{lem:anticom},
which shows that (\ref{diraceqn}) is equivalent with
$$
  ( d_{t,x} \m + B^{-1} d_{t,x}^* \m B ) F(t,x) =0,
$$
since $\partial_t- B^{-1}\partial_t B=0$.
Applying $d^*_{t,x}B$ to this equation shows that 
$d_{t,x}^* B d_{t,x} (\m F) = 0$, and
evaluating the scalar $\wedge^0$ part shows that
$\divv_{t,x} A \nabla_{t,x} F^{1,0}=0$ since $d_{t,x}^* B d_{t,x}$ preserves $k$-vectors.
\end{proof}

The following notions are central in our operator theoretic framework.
\begin{defn}  \label{defn:splittings}
Let $\mH$ be a Hilbert space.
\begin{itemize}
\item[{\rm (i)}]
  A ({\em topological}) {\em splitting} of $\mH$ is a decomposition 
$\mH=\mH_1\oplus\mH_2$
into closed subspaces $\mH_1$ and $\mH_2$. In particular, we have
$\|f_1+f_2\|\approx \|f_1\|+ \|f_2\|$ if $f_i\in \mH_i$.
\item[{\rm (ii)}]
Two bounded operators $R^+$ and $R^-$ in $\mH$ are called 
{\em complementary projections} if
$R^++R^-=I$, $(R^\pm)^2= R^\pm$ and $R^\pm R^\mp=0$.
\item[{\rm (iii)}]
A bounded operator $R$ in $\mH$ is called a {\em reflection operator} if 
$R^2=I$.
\end{itemize}
\end{defn}
We note the following connection between these concepts.
\begin{lem}   \label{lem:splittingrefl}
There is a one-to-one correspondence 
$$
  R= R^+-R^-\longleftrightarrow R^\pm= \tfrac 12(I\pm R)\longleftrightarrow 
  \mH= R^+\mH\oplus R^-\mH
$$
between reflection operators in $\mH$, complementary projections
in $\mH$ and topological splittings of $\mH$.
We write $R^\pm \mH= \ran(R^\pm)$ for the range of the projection $R^\pm$.
\end{lem}
In Definition~\ref{defn:normtang} we introduced the complementary projections
$N^\pm$ associated with the splitting of $\mH$ into the subspaces of tangential 
and normal multivector fields.
The corresponding reflection operator is
$$
  N := N^+-N^- = \mu^*\mu-\mu\mu^* =(\mu^*-\mu)\m = \m(\mu- \mu^*).
$$
We also introduce $B$-perturbed versions of the tangential and normal subspaces
$N^-\mH$ and $N^+\mH$ as
\begin{align*}
  B^{-1} N^+\mH &:= \sett{B^{-1}f}{f\in N^+\mH}, \\
  B^{-1} N^-\mH &:= \sett{B^{-1}f}{f\in N^-\mH}.
\end{align*}
In Definition~\ref{defn:tb} we encountered one of the complementary projections
$B^{-1}N^+B$ and $B^{-1}N^-B$ associated with the splitting
$$
  \mH = B^{-1} N^+\mH \oplus B^{-1} N^-\mH.
$$
However, more important will be the following complementary projections.
\begin{defn}  \label{defn:tangnormops}
Let $\hat N^+_B$ and $\hat N^-_B$ be the complementary projections associated
with the splitting
$$
  \mH = B^{-1} N^+\mH \oplus N^-\mH.
$$ 
We sometimes use the shorter notation $N^+_B := \hat N^+_B$ and $N^-_B :=\hat N^-_B$.
Also let $\hut N^+_B$ and $\hut N^-_B$ be the complementary projections associated
with the splitting
$$
  \mH = N^+\mH \oplus B^{-1} N^-\mH.
$$
Let $N_B=\hat N_B:= \hat N_B^+ - \hat N_B^-$ and 
$\hut N_B:= \hut N_B^+ - \hut N_B^-$ be the associated reflection operators. 
\end{defn}
  With the notation $\mu^*_B := B^{-1}\mu^* B$,
these operators are
\begin{align*}
  \hat N^+_B &= \mu^*_B (\mu + \mu^*_B)^{-1} = (\mu + \mu^*_B)^{-1}\mu, \\
  \hat N^-_B &= \mu (\mu + \mu^*_B)^{-1} = (\mu + \mu^*_B)^{-1}\mu^*_B, \\
  \hat N_B &= (\mu^*_B-\mu) (\mu + \mu^*_B)^{-1} = (\mu + \mu^*_B)^{-1}(\mu-\mu^*_B),
\end{align*}
and similarly for $\hut N^\pm_B$ and $\hut N_B$.
We shall prove in section~\ref{section2.3} that all these 
operators are bounded.
It is mainly the operators $\hat N^+_B$, $\hat N^-_B$ and $\hat N_B$ that we shall use.

\begin{defn}   \label{defn:duality}
  Let
$\dual fg_B := ( (BN^+ - N^-B)f,g )$, for $f,g\in\mH$.
As $BN^+ - N^-B = BM_B$ is invertible, $\dual\cdot\cdot_B$ is a duality,
i.e.
there exists $C<\infty$ such that
\begin{gather*}
  |\dual fg_B| \le C\|f\|\|g\|, \\
  \|f\| \le C\sup_{g\ne 0}|\dual fg_B|/\|g\|,\qquad
  \|g\| \le C\sup_{f\ne 0}|\dual fg_B|/\|f\|.
\end{gather*}
We write $T'$ for the adjoint of an operator $T$ with respect to this duality,
i.e. if 
$\dual {Tf}g_B = \dual f{T'g}_B$ for all $f,g\in\mH$.
\end{defn}

\begin{prop}  \label{prop:duality}
We have adjoint operators
$(\hat N_B)' = \hut N_{B^*}$, $(\hut N_B)' = \hat N_{B^*}$, $N'=N$ and
$(T_B)' = -T_{B^*}$.
\end{prop}
\begin{proof}
To prove that $(\hat N_B)'= \hut N_{B^*}$, recall that 
$\hat N_B$ is the reflection operator for the splitting
$\mH= B^{-1}N^+\mH\oplus N^-\mH$ and 
$\hut N_{B^*}$ is the reflection operator for the splitting
$\mH= N^+\mH\oplus (B^*)^{-1}N^-\mH$.
Thus we need to prove that 
\begin{multline*}
  \dual{B^{-1}N^+f_1- N^-f_2}{N^+g_1 +(B^*)^{-1}N^-g_2}_B \\ =
  \dual{B^{-1}N^+f_1+ N^-f_2}{N^+g_1 -(B^*)^{-1}N^-g_2}_B,
\end{multline*}
for all $f_i$, $g_i$,
i.e. that $\dual{B^{-1}N^+f_1}{(B^*)^{-1}N^-g_2}_B=0=
\dual{N^-f_2}{N^+g_1}_B$.
Use Lemma~\ref{lem:magic} to obtain
\begin{multline*}
  \dual{B^{-1}N^+f_1}{(B^*)^{-1}N^-g_2}_B=
  ( (BN^+ - N^-B) B^{-1}N^+f_1, (B^*)^{-1}N^-g_2) \\
   ( N^-M_B B^{-1}N^+f_1 , g_2)= ( M_B B^{-1}N^-N^+f_1, g_2) =0.
\end{multline*}
A similar calculation shows that $\dual{N^-f_2}{N^+g_1}_B=0$.
The proof of $(\hut N_B)' = \hat N_{B^*}$ is similar.
For the unperturbed reflection operator $N$, the duality $N'=N$ follows
directly from the fact the $BN^+ - N^-B$ is diagonal in the 
splitting $\mH=N^+\mH\oplus N^-\mH$.

To prove $(T_B)' = -T_{B^*}$, we calculate
\begin{multline*}
  \dual{T_Bf}g_B = ((BN^+ - N^-B)M_B^{-1}(\m d+ B^{-1}\m d^* B)f , g) \\
    = ((B \m d + \m d^* B)f , g)= (f , (B^*d\m +d^*\m B^*)g ) \\
    = -(f ,(N^+B^* -B^*N^- ) (N^+B^* -B^*N^-)^{-1} (B^* \m d +\m d^* B^*)g ) \\
    =-\dual f{ (N^+ - (B^*)^{-1}N^- B^*)^{-1} ( \m d + (B^*)^{-1}\m d^* B^*)g}_B 
    =\dual f{-T_{B^*}g}_B,
\end{multline*}
where we have used that $(B^*)^{-1}N^+ B^*- N^- = N^+ - (B^*)^{-1}N^- B^*$.
\end{proof}
\begin{rem}  \label{rem:aprioriduality}
  Our main use of these dualities is for proving surjectivity. Recall 
the following standard technique for proving invertibility of a 
bounded operator $T:\mH_1\rightarrow \mH_2$. Assume we have at our
disposal two pair of dual spaces $\dual{\mH_1}{\mK_1}_1$ and 
$\dual{\mH_2}{\mK_2}_2$ and that the adjoint operator is
$T':\mK_2\rightarrow \mK_1$.
If we can prove {\em a priori estimates} 
$$
  \| Tf \| \gtrsim \|f\|, \qquad\text{for all } f\in \mH_1,
$$
then $T$ is injective and has closed range. 
If furthermore $T'$ is injective, in particular if it satisfies a priori
estimates, then $T$ is surjective and therefore an isomorphism.
\end{rem}
\begin{rem}  \label{rem:restrdualities}
  We shall also need to restrict dualities to subspaces. Let $\dual\mH\mK$
be a duality, i.e. let $\dual\cdot\cdot:\mH\times\mK\rightarrow\C$ 
satisfy the estimates in Definition~\ref{defn:duality}.
If $\mH_1\subset\mH$ is a subspace, then a subspace $\mK_1\subset\mK$
is such that $\dual{\mH_1}{\mK_1}$ satisfies estimates as in 
Definition~\ref{defn:duality} if and only if $\mK_1$ is a 
complementary subspace to the annihilator 
$\sett{g\in \mK}{\dual fg=0,\, \text{ for all }f\in \mH_1}$,
in the sense of Definition~\ref{defn:splittings}(i).

In particular, if $R^\pm$ are complementary projections in $\mH$,
the adjoint operators $(R^\pm)'$ are also complementary projections and
the duality $\dual\mH\mH$ restricts to a duality 
$\dual{R^+\mH}{(R^+)'\mH}$, since the annihilator 
of $R^+\mH$ is $\nul((R^+)')$, which is complementary to $(R^+)'\mH$.
\end{rem}

%
%
%
\subsection{Hodge decompositions and resolvent estimates}  \label{section2.3}

In this section, we estimate the spectrum of the operator $T_B$.
For this we make use of Hodge type decompositions of $\mH$ as explained below.

\begin{defn}
By a {\em nilpotent operator} $\Gamma$ in 
a Hilbert space $\mH$, we mean a closed, densely defined operator
such that $\clos{\ran(\Gamma)}\subset\nul(\Gamma)$. 
In particular $\Gamma^2 f=0$ if $f\in\dom(\Gamma)$. 
We say that a nilpotent operator is {\em exact} if 
$\clos{\ran(\Gamma)}=\nul(\Gamma)$.

If $\wt\Gamma$ is another nilpotent operator, then we say that $\Gamma$ and 
$\wt\Gamma$ are {\em transversal} if there is a constant 
$c=c(\Gamma,\widetilde\Gamma)<1$ such that
$$
  |(f,g)| \le c \|f\|\|g\|, \qquad f\in\clos{\ran(\Gamma)}, 
  g\in\clos{\ran(\widetilde\Gamma)},
$$
or equivalently if $\|f+g\|\approx \|f\|+ \|g\|$ for all $f\in\clos{\ran(\Gamma)}$ and 
$g\in \clos{\ran(\widetilde\Gamma)}$.
Note that any nilpotent operator $\Gamma$ is transversal to its adjoint $\Gamma^*$
with $c=0$, since $\clos{\ran(\Gamma)}\subset \nul(\Gamma)= \clos{\ran(\Gamma^*)}^\perp$.
\end{defn}
Below we collect the properties of Hodge type splittings which we
need in this paper. 
This generalises results from \cite[Proposition 2.2]{AKMc}
and \cite[Proposition 3.11]{AMc}. 
\begin{lem}   \label{lem:hodge}
Let $\Gamma$ and $\widetilde\Gamma$ be two nilpotent operators in a Hilbert space
$\mH$ which are exact and transversal with constant $c(\Gamma,\widetilde\Gamma)$,
and assume also that the adjoints $\Gamma^*$ and $\wt\Gamma^*$ are transversal.
Let $B$ be a bounded, accretive multiplication operator and assume that
$$
  \max( c(\Gamma,\widetilde\Gamma), c(\Gamma^*,\widetilde\Gamma^*) )< \kappa_B/\|B\|_\infty.
$$
Define operators $\Pi:= \Gamma+\widetilde\Gamma$, 
$\wt\Gamma_B := B^{-1}\wt\Gamma B$ and
$\Pi_B := \Gamma+\wt\Gamma_B$
with domains $\dom(\Pi):=\dom(\Gamma)\cap\dom(\wt\Gamma)$, 
$\dom(\wt\Gamma_B):= B^{-1}\dom(\wt\Gamma)$ and
$\dom(\Pi_B):= \dom(\Gamma)\cap\dom(\wt\Gamma_B)$ respectively.

{\rm (i)}
  We have a topological splitting 
$$
\mH = \nul(\Gamma)\oplus \nul(\wt\Gamma_B),
$$
so that
$\|f_1+f_2\| \approx \|f_1\| + \|f_2\|$, $f_1\in\nul(\Gamma)$, 
$f_2\in\nul(\wt\Gamma_B)$.  
The operator $\Pi_B$ is a closed operator with dense domain and range.
Furthermore $\wt\Gamma_B$ is an exact nilpotent operator.

The complementary Hodge type projections associated with the splitting are
\begin{align*}
  \PP^1_B := \Gamma \Pi_B^{-1} = \Pi_B^{-1} \wt\Gamma_B = \Gamma \Pi_B^{-2} \wt\Gamma_B, \\
  \PP^2_B := \wt\Gamma_B \Pi_B^{-1} = \Pi_B^{-1} \Gamma = \wt\Gamma_B \Pi_B^{-2} \Gamma, 
\end{align*}
where we identify a bounded, densely defined operator with
its bounded extension to $\mH$.
In the special case when $B=I$ we write $\PP^1$ and $\PP^2$ for these projections.

{\rm (ii)}
Let $B_1$ and $B_2$ be two bounded, accretive operators in $\mH$, with
$\kappa_{B_i}/\|B_i\|_\infty > \max( c(\Gamma,\widetilde\Gamma), c(\Gamma^*,\widetilde\Gamma^*) )$,
$i=1,2$.
Then there exists $C<\infty$, depending only on $\|B_i\|_\infty$, $\kappa_{B_i}$,
$c(\Gamma,\widetilde\Gamma)$ and $c(\Gamma^*,\widetilde\Gamma^*)$, such that
$$
  \|(\Gamma+B_1^{-1}\wt\Gamma B_2)f\|\ge C^{-1}\lambda \|f\|,
  \qquad\text{for all }f\in \dom(\Gamma) \cap B_2^{-1}\dom(\wt\Gamma),
$$
if $\|\Pi f\|\ge \lambda\|f\|$ for all 
$f\in\dom(\Pi)$.
\end{lem}
\begin{proof}
Since $\wt\Gamma_B$ is conjugated to $\wt\Gamma$, it is an exact 
nilpotent operator.
To prove (i), it suffices to prove the estimate
\begin{equation}  \label{hodgeest}
  \|f\| + \|g\| \lesssim \|f+g\|, \qquad f\in\nul(\Gamma),\,\,
  g\in \nul(\wt\Gamma_B).
\end{equation}
Indeed, this shows that 
$\nul(\Gamma)\oplus \nul(\wt\Gamma_B)\subset\mH$.
Furthermore, replacing $(\Gamma, \wt\Gamma, B)$ with 
$(\wt\Gamma^*, \Gamma^*, B^*)$ shows that 
$\nul(\wt\Gamma^*)\oplus \nul(\Gamma^*_{B^*})\subset\mH$.
Conjugating with $B^*$ then shows that 
$\nul((\wt\Gamma_B)^*)\oplus \nul(\Gamma^*)\subset\mH$.
Therefore a duality argument proves the splitting 
$\mH = \nul(\Gamma)\oplus \nul(\wt\Gamma_B)$.
From this it follows that 
$$
  \dom(\Pi_B)= (\dom(\wt\Gamma_B)\cap\nul(\Gamma)) \oplus 
    ( \dom(\Gamma)\cap\nul(\wt\Gamma_B) )
$$
and $\ran(\Pi_B)= \ran(\Gamma)\oplus\ran(\wt\Gamma_B)$ are dense and 
that $\Pi_B$ is closed, as well as the boundedness of the associated 
projections.
 
To prove (\ref{hodgeest}), we use that $Bg\in\nul(\wt\Gamma)$
and thus $|(f,Bg)|\le c(\Gamma,\widetilde\Gamma)\|f\|\|Bg\|$,
and estimate
\begin{multline*}
  \|f\|^2 \le \kappa_B^{-1} \re(B f, f) 
  \le \kappa_B^{-1} \re\Big( (B (f+g), f) -(Bg, f) \Big) \\
  \le \kappa_B^{-1} \|B\|_\infty \| f+g\|\|f\|+ \kappa_B^{-1} c(\Gamma,\widetilde\Gamma)\|Bg\|\| f\| \\
  \le \kappa_B^{-1}\|B\|_\infty \|f+g\|\|f\|+ \kappa_B^{-1} c(\Gamma,\widetilde\Gamma)\|B\|_\infty
  (\|f+g\|+\|f\|) \|f\| \\
  \le \kappa_B^{-1}\|B\|_\infty (1+c(\Gamma,\widetilde\Gamma))
  \|f+g\|\|f\|
  +  \kappa_B^{-1}\|B\|_\infty c(\Gamma,\widetilde\Gamma) \|f\|^2.
\end{multline*}
Solving for $\|f\|^2$, this shows that $\|f\|\lesssim \|f+g\|$
provided $\kappa_B^{-1}\|B\|_\infty c(\Gamma,\widetilde\Gamma)<1$, which proves (\ref{hodgeest}).

To prove (ii) we factorise
$$
  \Gamma+B_1^{-1}\wt\Gamma B_2= (\PP^1+B_1^{-1}\PP^2)\Pi(\PP^2+\PP^1B_2).
$$
Here $(\PP^1+B_1^{-1}\PP^2)^{-1}=\PP^1_{B_1}+B_1\PP^2_{B_1}$ and
$(\PP^2+\PP^1B_2)^{-1}= \PP^1_{B_2}B_2^{-1}+\PP^2_{B_2}$. 
Indeed, we have
\begin{align*}
      (\PP^1+ & B_1^{-1}\PP^2)(\PP^1_{B_1}+B_1\PP^2_{B_1}) 
    = \PP^1\PP^1_{B_1}+\PP^1B_1\PP^2_{B_1} \\
    &+ B_1^{-1}\PP^2\PP^1_{B_1}+ B_1^{-1}\PP^2B_1\PP^2_{B_1} 
    =\PP^1\PP^1_{B_1}+B_1^{-1}\PP^2B_1\PP^2_{B_1} \\
    &= (\PP^1+\PP^2)\PP^1_{B_1}+B_1^{-1}(\PP^1+\PP^2)B_1\PP^2_{B_1}
    = \PP^1_{B_1}+\PP^2_{B_1}=I, \\
   (\PP^2+ & \PP^1 B_2)(\PP^1_{B_2}B_2^{-1}+\PP^2_{B_2})  
    = \PP^2\PP^1_{B_2}B_2^{-1}+\PP^2\PP^2_{B_2} \\
    &+ \PP^1 B_2\PP^1_{B_2}B_2^{-1}+ \PP^1 B_2\PP^2_{B_2} 
    =\PP^2\PP^2_{B_2}+ \PP^1 B_2\PP^1_{B_2}B_2^{-1} \\
    &=\PP^2(\PP^1_{B_2}+\PP^2_{B_2})+ 
     \PP^1 B_2(\PP^1_{B_2}+\PP^2_{B_2})B_2^{-1} 
    =\PP^2+ \PP^1= I.  
\end{align*}
A similar calculation shows that $\PP^1_{B_1}+B_1\PP^2_{B_1}$
and $\PP^1_{B_2}B_2^{-1}+\PP^2_{B_2}$ also are left inverses.
Thus $\PP^1+B_1^{-1}\PP^2$ and $\PP^2+\PP^1B_2$ are isomorphisms and (ii) follows.
\end{proof}
We can now prove the following perturbed normal and tangential splittings of $\mH$.
\begin{align*}
  \mH &= \hat N^+_B \mH \oplus \hat N^-_B\mH = B^{-1} N^+ \mH \oplus N^- \mH \\
  \mH &= \hut N^+_B \mH \oplus \hut N^-_B \mH =  N^+ \mH \oplus B^{-1} N^- \mH
\end{align*}
To see this, let first $\Gamma=\mu$ and $\wt\Gamma= \Gamma^*=\mu^*$
in Lemma~\ref{lem:hodge}(i). It follows that $\nul(\Gamma)=N^-\mH$,
$\nul(\Gamma^*_B)=B^{-1}N^+\mH$, $\PP^1_B=\hat N^-_B$ and
$\PP^2_B=\hat N^+_B$.
On the other hand, choosing $\Gamma=\mu^*$ and $\wt\Gamma= \Gamma^*=\mu$, 
we see that $\nul(\Gamma)=N^+\mH$,
$\nul(\Gamma^*_B)=B^{-1}N^-\mH$, $\PP^1_B=\hut N^+_B$ and
$\PP^2_B=\hut N^-_B$. 
Lemma~\ref{lem:hodge}(i) thus shows that
all the oblique normal and tangential projections
$\hat N^\pm_B$ and $\hut N^\pm_B$ from Definition~\ref{defn:tangnormops}
are bounded operators on $\mH$, i.e. we have the stated splittings.

Next we apply Lemma~\ref{lem:hodge}(ii) to prove resolvent bounds
for the operator $T_B$ from Definition~\ref{defn:tb}.
Define closed and open sectors and double sectors in the complex plane by
\begin{alignat*}{2}
    S_{\omega+} &:= \sett{z\in\C}{|\arg z|\le\omega}\cup\{0\}, 
    & \qquad
    S_{\omega} &:= S_{\omega+}\cup(-S_{\omega+}), \\    
    S_{\nu+}^o &:= \sett{z\in\C}{ z\ne 0, \, |\arg z|<\nu},  
    & \qquad  
    S_{\nu}^o &:= S_{\nu+}^o\cup(- S_{\nu+}^o).
\end{alignat*}
\begin{prop}  \label{prop:spectrest}
The operator $T_B$ is a closed operator in $\mH$ with dense domain and range, 
for any accretive, complex matrix function $B\in L_\infty(\R^n;\mL(\wedge))$.
Furthermore, $T_B$ is a bisectorial operator with $\sigma(T_B) \subset S_\omega$,
where 
$$
  \omega:= \arccos(\kappa_B/(2\|B\|_\infty))\in[\pi/3,\pi/2).
$$
and if $\omega<\nu<\pi/2$, then there exists $C<\infty$ depending only
on $\nu$, $\kappa_B$ and $\|B\|_\infty$ such that
$$
  \|(\lambda-T_B)^{-1}\| \le C/|\lambda|, \qquad \text{for all }\lambda\notin S_\nu.
$$
\end{prop}
\begin{proof}
  It is a consequence of Lemma~\ref{lem:hodge}(i) that the operator
  $\Pi_B= \ud+B^{-1}\ud^* B$
  is a closed operator with dense domain and range. Since $T_B=-iM_B^{-1}\Pi_B$ with $M_B$
  an isomorphism, it follows that $T_B$ also is closed with dense 
  domain and range.

To prove the resolvent estimate, write $\lambda = 1/(i\tau)$ 
where $\tau\in S^o_{\pi/2-\nu}$.
We first prove that $\|u\|\le C \|f\|$ if 
$(I-i\tau T_B)u = f$,
uniformly for $\tau\in S^o_{\pi/2-\nu}$.
Multiply the equation with
$i(\mu-\widetilde B^{-1}\mu^* B)$ to obtain
$$
  (\Gamma+ \widetilde B^{-1} \wt\Gamma B)u= i(\mu-\widetilde B^{-1}\mu^* B)f,
$$
where $\Gamma:= i\mu+\tau d$ and $\wt\Gamma:= -i\mu^* +\tau d^*$ are
nilpotent by Lemma~\ref{lem:anticom}.
It suffices to prove $\|u\| \lesssim \|(\Gamma+\wt B^{-1}\wt\Gamma B)u\|$.

(i)
By orthogonality we have
\begin{multline*}
  \|(\Gamma+\Gamma^*)u\|^2 = \|\Gamma u\|^2 + \|\Gamma^* u\|^2 
   = \|\mu u\|^2 + |\tau|^2\|du\|^2 +2\re(i\mu u,\tau du) \\
     + \|\mu^* u\|^2 + |\tau|^2\|d^*u\|^2 +2\re(-i\mu^* u,\conj\tau d^*u),
\end{multline*}
where
\begin{multline*}
\re(i\mu u,\tau du) + \re(-i\mu^* u,\conj \tau d^*u) \\ = \re(i\conj\tau d^*\mu u,u) + \re( u,i\conj\tau \mu d^*u) 
 =\re( i\conj\tau \{d^*,\mu\}u,u)=0,
 \end{multline*}
by Lemma~\ref{lem:anticom}. Thus
$\|(\Gamma+\Gamma^*)u\|^2 =  \|(\mu+\mu^*)u\|^2 + |\tau|^2\|(d+d^*)u\|^2 \ge \|\m u\|^2= \|u\|^2$.
In particular $\Gamma$ is exact.

(ii)
Next we prove that $\Gamma$ and $\wt\Gamma$ are transversal, with a 
bound $c<1$ uniformly for all $\tau\in S^o_{\pi/2-\nu}$.
By exactness, it suffices to bound $(f,g)$ for $f=\Gamma u\in\ran(\Gamma)$
and $\wt\Gamma g=0$. Furthermore, using the orthogonal Hodge splitting
$\mH= \nul(\Gamma)\oplus\nul(\Gamma^*)$, we may assume that $\Gamma^* u=0$.
We get
\begin{multline*}
  (f,g) = (\Gamma u,g)= (u, \Gamma^* g) = (u, (-i\mu^* + \conj \tau d^*)g) \\
  = (u, -i\mu^* g +i(\conj \tau/\tau)\mu^* g)
  = 2(\tau-\conj\tau)/(2i\conj\tau) (u,\mu^*g),
\end{multline*}
and thus $|(f,g)|\le 2|\sin(\arg\tau)| \|u\|\|g\|\le c\|f\| \|g\|$ 
by (i), where $c<\kappa_B/\|B\|_\infty$ since
$\pi/2-\nu<\arcsin(\kappa_B/(2\|B\|_\infty))$.
A similar argument shows that $\Gamma^*$ and $\wt\Gamma^*$ are transversal 
with the same constant $c<\kappa_B/\|B\|_\infty$ uniformly for 
$\tau\in S^o_{\pi/2-\nu}$.

(iii) 
To apply Lemma~\ref{lem:hodge}(ii) it now suffices to prove that
$\|(\Gamma+\wt\Gamma)u\| \ge C^{-1}\|u\|$ uniformly for 
all $\tau\in S^o_{\pi/2-\nu}$.
From Lemma~\ref{lem:hodge}(i) with $B=I$ we have
$\|(\Gamma+\wt\Gamma)u\| \approx \|\Gamma u\|+\|\wt\Gamma u\|$.
Using the Hodge splitting $\mH= \nul(\Gamma)\oplus \nul(\wt\Gamma)$
it suffices to prove 
$\|\Gamma u\|\gtrsim \|u\|$ for $u\in\nul(\wt\Gamma)$, and
$\|\wt\Gamma v\|\gtrsim \|v\|$ for $v\in\nul(\Gamma)$.
To prove for example the first estimate, write 
$\nul(\wt\Gamma)\ni u= u_1+u_2\in\nul(\Gamma)\oplus \nul(\Gamma^*)$.
Then
\begin{multline*}
  \|\Gamma u\|^2= \|\Gamma u_2\|^2\ge \|u_2\|^2=\|u-u_1\|^2 \\
  \ge \|u\|^2+\|u_1\|^2-2c\|u\|\|u_1\| \ge (1-c^2)\|u\|^2,
\end{multline*}
where we have used (i) in the second step and (ii) in the
fourth step.

This proves that $\|u\|\le C \|f\|$ if $(I-i\tau T_B)u = f$.
Since $(I-i\tau T_B)'= I-i\conj\tau T_{B^*}$ by Proposition~\ref{prop:duality},
a duality argument shows that $I-i\tau T_B$ is onto, 
and the proof is complete.
\end{proof}
From the uniform boundedness of the resolvents $R_t^B:=(I+itT_B)^{-1}$ for $t\in \R$ and
the boundedness of the Hodge projections 
$\PP_B^1 := \ud \Pi_B^{-1}$ and $\PP_B^2 := \ud^*_B \Pi_B^{-1}$, where 
$\Pi_B= \ud + \ud^*_B$, $\ud= i\m d$ and $\ud^*_B = B^{-1}\ud^* B$, 
we can now deduce boundedness of operators related to
$P_t^B = \tfrac 12(R_{-t}^B+R_t^B)$ and
$Q_t^B = \tfrac 1{2i}(R_{-t}^B-R_t^B)$.
\begin{cor}  \label{cor:opfamilies}
The following families of operators are all uniformly bounded for
$t>0$.
\begin{alignat*}{2}
  R_t^B & :=(I+itT_B)^{-1}, &\qquad
  t\ud P_t^B & = i\PP_B^1 M_B Q_t^B, \\
  P_t^B & := (I+ t^2 T_B^2)^{-1}, &\qquad
  t\ud^*_B P_t^B & = i\PP_B^2 M_B Q_t^B, \\
  Q_t^B & := tT_B(I+ t^2 T_B^2)^{-1}, &\qquad
  P_t^B M_B^{-1}t\ud & = i Q_t^B \PP_B^2, \\
  t^2 T_B^2 P_t^B & = tT_B Q_t^B = I-P_t^B, &\qquad
  P_t^B M_B^{-1}t\ud^*_B  &= i Q_t^B \PP_B^1, \\
  t\ud Q_t^B & = i\PP_B^1 M_B (I-P_t^B), &\qquad
  t\ud P_t^B M_B^{-1}t\ud & = \PP^1_B M_B (P_t^B-I) \PP_B^2, \\
  t\ud^*_B Q_t^B  &= i\PP_B^2 M_B (I-P_t^B), &\qquad
  t\ud P_t^B M_B^{-1}t\ud^*_B & = \PP^1_B M_B (P_t^B-I) \PP_B^1, \\
  Q_t^B M_B^{-1}t\ud & = i (I-P_t^B) \PP_B^2, &\qquad
  t\ud^*_B P_t^B M_B^{-1}t\ud & = \PP^2_B M_B (P_t^B-I) \PP_B^2, \\
  Q_t^B M_B^{-1}t\ud^*_B  &= i (I-P_t^B) \PP_B^1, &\qquad
  t\ud^*_B P_t^B M_B^{-1}t\ud^*_B & = \PP^2_B M_B (P_t^B-I) \PP_B^1.
\end{alignat*}
\end{cor}
These families of operators are not only bounded, but have 
$L_2$ off-diagonal bounds in the following sense.
\begin{defn}  \label{defn:offdiag}
Let $(U_t)_{t>0}$ be a family of operators on $\mH$,
and let $M\ge 0$.
We say that $(U_t)_{t>0}$ has $L_2$ {\em off-diagonal bounds}
(with exponent $M$) if there exists $C_M<\infty$ such that
$$
  \|U_t f\|_{L_2(E)} \le C_M \brac{\dist (E,F)/t}^{-M}\|f\|
$$
whenever $E,F \subset \R^n$ and $\supp f\subset F$.
Here $\brac x:=1+|x|$, and
$\dist(E,F) :=\inf\{|x-y|:x\in E,y\in F\}$.
We write $\|U_t\|_{\off,M}$ for the smallest constant $C_M$.
The exact value of $M$ is normally not important and we write
$\|U_t\|_{\off}$, where it is understood that $M$ is chosen 
sufficiently large but fixed.
\end{defn}
\begin{prop}  \label{pseudoloc}
All the operator families from
Corollary~\ref{cor:opfamilies}
has $L_2$ off-diagonal bounds for all exponents $M\ge 0$.
\end{prop}
\begin{proof}
First consider the resolvents  $R^B_t = (I+it T_B)^{-1}$.
As we already have proved uniform bounds for $R^B_t$, it suffices to prove
$$
       \|(I+itT_B)^{-1} f\|_{L_2(E)} \le C_M
       (|t|/\dist(E,F))^M\|f\|
$$
for $|t|\le \dist(E,F)$.
We prove this by induction on $M$ as in \cite[Proposition 5.2]{AKMc}.
Let $\eta:\R^n\rightarrow[0,1]$ be a bump function such that 
$\eta|_E=1$, $\supp\eta\subset\widetilde E:=\sett{x\in\R^n}
{\dist(x,E)\le\dist(x,F)}$ and 
$\|\nabla\eta\|_\infty\lesssim 1/\dist(E,F)\approx 1/\dist(\widetilde E,F)$.
Since the commutator is
$$
  [\eta I, R_t^B]= t R_t^B M_B^{-1}([\eta I, \ud]+B^{-1}[\eta I,\ud^*]B)R_t^B,
$$
where $\|[\eta I, \ud]\|, \|[\eta I,\ud^*]\| \lesssim \|\nabla \eta\|_\infty$,
we get
$$
  \|R_t^B f\|_{L_2(E)}\le \|\eta R_t^B f\|= \|[\eta I, R_t^B ]f\|
  \lesssim |t| \|\nabla \eta\|_\infty \|R_t^B f\|_{L_2(\widetilde E)},
$$
where we used that $\eta f=0$.
By induction, this proves the off-diagonal bounds for $R_t^B$.
From this, off-diagonal bounds for $P_t^B$, $Q_t^B$ and $I-P_t^B$ also follows immediately.

Next we consider $t\ud P_t^B$ and use Lemma~\ref{lem:hodge}(i) 
to obtain
\begin{multline}  \label{hodgeoffdiag}
  \|t\ud P_t^B f\|_{L_2(E)}\le \|\eta t\ud P_t^B f\|\le 
   \|[\eta I,t\ud] P_t^B f\| + \|t\ud \eta P_t^B f\|  \\
   \lesssim
   |t|\|\nabla\eta\|_\infty \|P_t^B f\|_{L_2(\widetilde E)} + \|t T_B \eta P_t^B f\| \\
   \lesssim  |t|\|\nabla\eta\|_\infty \|P_t^B f\|_{L_2(\widetilde E)} + 
      \|[\eta,t T_B]  P_t^B f\| +  \|\eta t T_B P_t^B f\| \\
   \lesssim  |t|\|\nabla\eta\|_\infty \|P_t^B f\|_{L_2(\widetilde E)} + 
     \|Q_t^B f\|_{L_2(\widetilde E)}.
\end{multline}
This and the corresponding calculation with $\ud$ replaced by $\ud^*_B$ proves 
the off-diagonal bounds for $t\ud P_t^B$ and $t\ud^*_B P_t^B$.
From this the result for $P_t^B M_B^{-1}t\ud$ and $P_t^B M_B^{-1}t\ud^*_B$ follows
immediately with a duality argument. Indeed, $(M_B^{-1}\ud)'= M_{B^*}^{-1}\ud^*_{B^*}$
and $(M_B^{-1}\ud^*_B)'= M_{B^*}^{-1}\ud$ is proved similarly to $T_B' = -T_{B^*}$ 
in Proposition~\ref{prop:duality}.

The proof for $t\ud Q_t^B$, $t\ud^*_B Q_t^B$, $Q_t^B M_B^{-1}t\ud$ and 
$Q_t^B M_B^{-1}t\ud^*_B$ is similar, replacing $P_t$ and $Q_t$ with 
$Q_t$ and $I-P_t$.
Finally, the last four estimates follows from a computation like (\ref{hodgeoffdiag}),
for example replacing $P_t$ and $Q_t$ with $P_t^B M_B^{-1}t\ud$ and
$Q_t^B M_B^{-1}t\ud$ proves the estimate for $t\ud P_t^B M_B^{-1}t\ud$.
\end{proof}

We finish this section with a lemma to be used in Section~\ref{section8}.
This lemma is proved with an argument similar to that in 
\cite[Lemma 2.3]{HM}.
For completeness, we include a short proof.
\begin{lem}  \label{lem:offdiagcomposition}
Assume that $(U_t)_{t>0}$ and $(V_t)_{t>0}$ both have $L_2$ 
off-diagonal bounds with exponent $M$.
Then $(U_tV_t)_{t>0}$ has $L_2$ off-diagonal bounds with exponent $M$
and
$$
  \|U_tV_t\|_{\off,M} \le 2^{M+1} \|U_t\|_{\off,M} \|V_t\|_{\off,M}.
$$
\end{lem}
\begin{proof}
By Definition~\ref{defn:offdiag} we need to prove that
$$
\|U_tV_t f\|_{L_2(E)} \lesssim \brac{\dist (E,F)/t}^{-M}\|f\|
$$
whenever $\supp f\subset F$.
To this end let $G:= \sett{x\in \R^n}{\dist(x, E)\le \rho/2}$,
where $\rho:= \dist(E,F)$.
We get
\begin{multline*}
 \|U_tV_t f\|_{L_2(E)} \le \|U_t(\chi_G V_tf)\|_{L_2(E)}
 +\|U_t(\chi_{\R^n\setminus G} V_tf)\|_{L_2(E)} \\
 \le C_0^{U_t} \|V_tf\|_{L_2(G)}
 +  C_M^{U_t} \brac{\rho/2t}^{-M} \| V_tf\|_{L_2(\R^n\setminus G)} \\
 \le C_0^{U_t} C_M^{V_t} \brac{\rho/2t}^{-M} \|f\|
 +  C_M^{U_t} \brac{\rho/2t}^{-M} C_0^{V_t} \|f\| 
 \le 2^{M+1} C_M^{U_t} C_M^{V_t}\brac{\rho/t}^{-M} \|f\|.
\end{multline*}
\end{proof}

%
%
%
\subsection{Quadratic estimates: generalities}  \label{section2.4}

In Proposition~\ref{prop:spectrest} we proved the spectral estimate
$\sigma(T_B)\subset S_\omega$ for some angle $\omega<\pi/2$ with bounds
on the resolvent outside $S_\omega$.
In this section we survey some general facts about the functional
calculus of the operator $T_B$. For a further background and discussion 
of these matters we refer to \cite{AKMc, ADMc}.
\begin{defn}
For $\omega<\nu<\pi/2$, we define the following classes of holomorphic 
functions $f\in H(S_\nu^o)$ on the open double sector $S_\nu^o$.
\begin{align*}
  \Psi(S^o_\nu)&:= \sett{\psi\in H(S_\nu^o)}
  {| \psi(z)| \le C \min(|z|^s,|z|^{-s}),
  z \in S^o_\nu, \\
   &\hspace{6cm} \text{for some } s>0, C<\infty}, \\
  H_\infty(S^o_\nu)&:= \sett{b\in H(S_\nu^o)}
  {| b(z)| \le C, z \in S^o_\nu,\text{ for some } C<\infty}, \\
  F(S^o_\nu)&:= \sett{w\in H(S_\nu^o)}
  {| w(z)| \le C \max(|z|^s,|z|^{-s}),
  z \in S^o_\nu,\\
   &\hspace{6cm} \text{for some } s<\infty, C<\infty}.
\end{align*}
Thus $\Psi(S^o_\nu)\subset H_\infty(S^o_\nu)\subset F(S^o_\nu)\subset H(S_\nu^o)$.
\end{defn}
For $\psi\in \Psi(S^o_\nu)$, we define a bounded operator $\psi(T_B)$
through the Dunford functional calculus
\begin{equation}  \label{dunfordschwarz}
     \psi(T_B):=\frac 1{2\pi i}\int_\gamma
\psi(\lambda)(\lambda I-T_B)^{-1} d\lambda
\end{equation}
where $\gamma$ is the unbounded contour $\sett{\pm r
e^{\pm i\theta}}{r> 0}$, $\omega <\theta<\nu$, parametrised
counterclockwise around $S_{\omega}$. 
The decay estimate on $\psi$ and the resolvent bounds of Proposition~\ref{prop:spectrest} 
guarantee that $\|\psi(T_B)\| <\infty$.

For general $w\in F(S^o_\nu)$ we define 
$$
  w(T_B) := (Q^B)^{-k} (q^kw)(T_B), 
$$
where $k$ is an integer larger than $s$ if $|w(z)| \le C \max(|z|^s,|z|^{-s})$,
and $q(z):= z(1+z^2)^{-1}$ and $Q^B:= q(T_B)$.
This yields a closed, densely defined operator $w(T_B)$ in $\mH$.
Furthermore, we have 
\begin{align}
  \clos{\lambda_1 w_1(T_B) + \lambda_2 w_2(T_B)} &= (\lambda_1 w_1+\lambda_2w_2)(T_B), \nonumber \\
  \clos{w_1(T_B) w_2(T_B)} &= (w_1w_2)(T_B),  \label{eq:homom}
\end{align}
for all $w_1$ and $w_2\in F(S^o_\nu)$.
Here $\conj T=S$ means that the graph $\graf(T)$ is dense in the graph $\graf(S)$.

The functional calculus $w\mapsto w(T_B)$ have the following
convergence properties as proved in \cite{ADMc}.
\begin{lem}  \label{lem:convergence}
  If $b_k\in H_\infty(S_\nu^o)$ is a sequence uniformly bounded on $S_\nu^o$
which converges to $b$ uniformly on compact subsets, and if $b_k(T_B)$
are uniformly bounded operators, then 
$$
  b_k(T_B)f\longrightarrow b(T_B)f,\qquad\text{for all } f\in\mH,
$$
and $\|b(T_B)\|\le \limsup_k\|b_k(T_B)\|$.
\end{lem}
\begin{defn}   \label{defn:fcalcops}
The following operators in the functional calculus are of special importance
to us.
\begin{enumerate}
\item
$q_t(z)=q(tz):=tz(1+t^2z^2)^{-1}\in \Psi(S^o_\nu)$, which give the operator
$Q_t^B$.
\item
$|z|^s:= (z^2)^{s/2}\in F(S^o_\nu)$, which give the operator $|T_B|^s$.
Note that $|z|$ does not denote absolute value here, but $z\mapsto|z|$ 
is holomorphic on $S^o_\nu$.
\item
$e^{-t|z|}\in H_\infty(S^o_\nu)$, which give the operator $e^{-t|T_B|}$.
\item
The characteristic functions
$$
\chi^\pm(z)= \left\{ \begin{array}{ll}
1    & \quad \textrm{if $\pm \re z>0$}\\
0    & \quad \textrm{if $\pm \re z< 0$} \\
\end{array} \right.
$$
which give the generalised {\em Hardy projections} $E^\pm_B := \chi^\pm(T_B)$.
\item
The signum function
$$
  \sgn(z)= \chi^+(z) -\chi^-(z)
$$
which give the generalised {\em Cauchy integral} $E_B := \sgn(T_B)$.
\end{enumerate}
\end{defn}
The main work in this paper is to prove the boundedness the projections
$E^\pm_B$. 
As in Lemma~\ref{lem:splittingrefl}, if these are bounded then they correspond to 
a splitting
$$
  \mH = E^+_B\mH \oplus E^-_B \mH
$$
of $\mH$ into the {\em Hardy subspaces} $E^\pm_B \mH$ associated 
with the equation (\ref{diraceqn}). 
That the projections are bounded is also equivalent with having a 
bounded reflection operator $E_B$.
\begin{defn}
  For a function $F(t,x)$ defined in $\R^{n+1}_\pm$ we write
$$
  \tb{ F }_\pm:=
  \left(\int_0^\infty\| F(\pm t,x) \|^2 \, \frac{dt}t\right)^{1/2}
$$
and for short $\tb{F}_+ =: \tb F$.
When $F(t,x)= (\Theta_t f)(x)$ for some family of operators $(\Theta_t)_{t>0}$,
we use the notation
$$
  \tb{\Theta_t}_\op := \sup_{\|f\|=1} \tb{\Theta_tf}.
$$
\end{defn}
Our main goal in this paper will be to prove quadratic estimates
of the form
\begin{equation}  \label{eq:quadrests}
  \int_0^\infty\| Q_t^B f \|^2 \, \frac{dt}t \approx \|f\|^2,
\end{equation}
for certain coefficients $B$.
We recall the following two basic results concerning quadratic 
estimates which are proved by Schur estimates. 
For details we refer to \cite{ADMc}.
\begin{prop}   \label{prop:differentpsi}
Let $\psi\in\Psi(S^o_\nu)$ be non vanishing on both $S^o_{\nu+}$
and $S^o_{\nu-}$,
and define $\psi_t(z):= \psi(tz)$.
Then there exists $0<C<\infty$ such that
$$
  C^{-1}\tb{ Q_t^B f } \le \tb{ \psi_t(T_B)f } \le C\tb{ Q_t^B f }.
$$ 
\end{prop}

\begin{prop}  \label{prop:QtoFCalc}
If $T_B$ satisfies quadratic estimates (\ref{eq:quadrests}), 
then $T_B$ has bounded $H_\infty(S^o_\nu)$ functional calculus,
i.e.
$$
  \|b(T_B)\| \lesssim \|b\|_\infty, \qquad\text{for all } b\in H_\infty(S^o_\nu).
$$
Thus $H_\infty(S^o_\nu) \ni b\longmapsto b(T_B) \in \mL(\mH)$
is a continuous homomorphism.
\end{prop}

%
%
%
Before proving quadratic estimates for $T_B$ for certain $B$
in Sections~\ref{section7} and \ref{section8}, 
we introduce a dense subspace on which the operator $b(T_B)$ is defined
for any $b\in H_\infty(S^o_\nu)$.
\begin{defn}   \label{defn:denseset}
Let $\V_B$ be the dense linear subspace
$$
  \V_B := \bigcup_{s>0} (\dom(|T_B|^s)\cap\ran(|T_B|^s))\subset\mH.
$$  
\end{defn}
We see that $\dom(|T_B|^s)\cap\ran(|T_B|^s)$ increases when $s$ decreases.
The density of $\dom(|T_B|)\cap\ran(|T_B|)$, and therefore of $\V_B$, follows from the fact that
\begin{equation*} 
2\int_\alpha^\beta (Q^B_t)^2f\,\frac{dt}t=
(P^B_\alpha-P^B_\beta)f \longrightarrow f,
\end{equation*}
as $(\alpha,\beta)\rightarrow (0,\infty)$, for all $f\in\mH$.

Moreover, if $b\in H_\infty(S^o_\nu)$ and $f\in\V_B$ then
$b(T_B)f\in\V_B\subset\mH$.
To see this, write
$$
  b(T_B) f = (b\psi)(T_B)(\psi(T_B)^{-1}f),
$$
where $\psi(z)^{-1}:=(1+|z|^{s})/|z|^{s/2}$ 
if $f\in\dom(|T_B|^s)\cap\ran(|T_B|^s)$.
Then $\psi(T_B)^{-1}f\in \mH$ and $(b\psi)(T_B)$ is bounded since $b\psi\in \Psi(S^o_\nu)$.
Furthermore, if $s'<s/2$ then 
$|T_B|^{s'}(\psi(T_B)^{-1}f)\in\mH$ and 
$|T_B|^{-s'}(\psi(T_B)^{-1}f)\in\mH$, 
so $b(T_B)f\in\dom(|T_B|^{s'})\cap\ran(|T_B|^{s'})$.

\begin{lem}  \label{lem:algsplitting}
  We have an algebraic splitting
$$
  \V_B = E_B^+\V_B + E_B^-\V_B,
$$
$E_B^+\V_B\cap E_B^-\V_B =\{0\}$,
and $f\in E_B^\pm\V_B$ is in one-to-one correspondence with
$F(t,x)=F_t(x):= (e^{\mp t|T_B|}f)(x)$ in $\R_\pm^{n+1}$, and
\begin{gather*}
  \lim_{t\rightarrow 0^\pm} F_t = f, \qquad
  \lim_{t\rightarrow \pm\infty} F_t = 0, \\
  F_t\in \dom(T_B)= \dom(d)\cap\dom(d^*B),\qquad 
  \pd_t F_t=-T_BF_t\in L_2(\R^n),
  \qquad \pm t>0.
\end{gather*}
\end{lem}
\begin{proof}
That each $f\in \V_B$ can be uniquely written $f= f^++ f^-$, where $f^\pm\in  E^\pm_B\V_B$, follows from (\ref{eq:homom}).
To verify the properties of $F_t$, it suffices to consider the case $f\in E_B^+\V_B$ 
as the case $f\in E_B^-\V_B$ is similar.
Since $z e^{-t|z|}\in \Psi(S_\nu^0)$ it follows that 
$F_t\in\dom(T_B)$.
Moreover, since 
$\tfrac 1h(e^{-(t+h)|z|}-e^{-t|z|})\rightarrow -|z|e^{-t|z|}$
uniformly on $S_\nu^0$, it follows from
Lemma~\ref{lem:convergence} that 
$\pd_t F_t =-|T_B| F_t$ and since $F_t\in E_B^+\V_B$ we have 
$|T_B| F_t= T_B F_t$.
 
To prove the limits, assume that $f\in \dom(|T_B|^s)\cap\ran(|T_B|^s)$
for some $0<s<1$.
Writing $f= |T_B|^{-s}u$, we see that
$$
  F_t-f= t^s \psi(t T_B)u, \qquad\text{where } 
  \psi(z)= (e^{-|z|}-1)/|z|^s.
$$
Similarly with $f= |T_B|^s v$, we get
$$
  F_t = t^{-s} \psi(t T_B)v, \qquad\text{where }
  \psi(z)= |z|^s e^{-|z|}.
$$
Since in both cases $\psi(t T_B)$ are uniformly bounded in $t$, 
using a direct norm estimate in (\ref{dunfordschwarz}),
it follows that
both limits are $0$ as $t\rightarrow 0$ and $t\rightarrow\infty$
respectively.
In particular, $f=\lim_{t\rightarrow 0} F_t$ is uniquely determined by 
$F$.
\end{proof}
We now further discuss the quadratic estimates (\ref{eq:quadrests}).
First note the following consequence of the duality 
$T_B' =-T_{B^*}$ from Proposition~\ref{prop:duality}.
Again we refer to \cite{ADMc} for further details.
\begin{lem}   \label{lem:dualityforquadratic}
If $\tb{Q_t^B f}\lesssim \|f\|$ for all $f\in\mH$, then
$\tb{Q_t^{B^*} f}\gtrsim \|f\|$ for all $f\in\mH$.
\end{lem}
In Section~\ref{section7.3} we shall use the following Hardy space reduction
of the quadratic estimate. 
This is a technique due to Coifman, Jones and Semmes~\cite{CJS}, and adapted to 
the setting of functional calculus by \Mcc Intosh and Qian~\cite[Theorem 5.2]{McQ}.
\begin{prop}  \label{prop:Hardyreduction}
  Assume that we have reverse quadratic estimates in $E_{B^*}^+\V_{B^*}$
  and $E_{B^*}^-\V_{B^*}$, i.e.
$$
  \|g\| \lesssim \tb{ t\pd_t G_t }_\pm,
  \qquad g\in E_{B^*}^\pm\V_{B^*},
$$
where $G_t= e^{\mp t|T_{B^*}|}g$. 
Then 
$$
  \tb{Q_t^B f}\lesssim \|f\|,
  \qquad f\in\mH.
$$
In particular, if we have reverse quadratic estimates in both Hardy 
spaces for both operators $T_B$ and $T_{B^*}$, then  
$\tb{Q_t^B f}\approx \|f\|\approx \tb{Q_t^{B^*} f}$, $f\in\mH$, 
and thus $T_B$ and $T_{B^*}$ have bounded $H_\infty(S^o_\nu)$ functional
calculus for all $\omega<\nu<\pi/2$.
\end{prop}
\begin{proof}
  Assume first that 
$$
  \|g\| \lesssim \tb{ t\pd_t G_t },
  \qquad g\in E_{B^*}^+\V_{B^*},
$$
and write $\psi_t(z) := tz e^{-t|z|}$ so that $t\pd_t G_t=-\psi_t(T_{B^*})g$
by Lemma~\ref{lem:algsplitting}.
Let $f\in\mH$ and define $q_t^\mp(z):= \chi^\mp q_t(z)$ so that
$q_t^\mp(T_B)= E^\mp_B Q_t^B$.
To prove that 
$$
  \tb{q_t^-(T_B) f}\lesssim \|f\|, \qquad f\in\mH,
$$
it suffices to bound $\int_\alpha^\beta \|q_t^-(T_B) f\|^2\, \tfrac{dt}t$
uniformly for $\alpha>0$ and $\beta<\infty$.
To this end, define the auxiliary functions
\begin{align*}
  h_t &:= \left(\int_\alpha^\beta \|q_t^-(T_B) f\|^2\, \tfrac{dt}t\right)^{-1/2}
  (N^+B^*-B^*N^-)^{-1}q_t^-(T_B) f, \\
  g &:= -\int_\alpha^\beta q_t^+(T_{B^*})h_t \, \tfrac{dt}t,
\end{align*}
so that $\int_\alpha^\beta \|h_t\|^2 \frac{dt}t\le C$ and 
$g\in E_{B^*}^+\V_{B^*}$, and calculate
\begin{multline*}
  \left(\int_\alpha^\beta \|q_t^-(T_B) f\|^2\, \tfrac{dt}t\right)^{1/2} =
  \int_\alpha^\beta \dual{q_t^-(T_B) f}{h_t}_B \, \tfrac{dt}t
  =\dual fg_B \\
  \lesssim \|f\|\,\|g\|\lesssim 
  \|f\|\, \tb{\psi_s(T_{B^*})g } \\
  \lesssim
  \|f\| \left( \int_0^\infty \left(\int_\alpha^\beta \| (\psi_s q_t^+)(T_{B^*})\|\,
  \|h_t\|\,\frac{dt}t \right)^2 \frac{ds}s \right)^{1/2} \\
  \lesssim\|f\| \left( \int_\alpha^\beta \|h_t\|^2 \frac{dt}t \right)^{1/2} \lesssim \|f\|.
\end{multline*}
In the second equality we have used that $q_t^-(T_B)'= q_t^-(-T_{B^*})= -q_t^+(T_{B^*})$,
in the second estimate we used the hypothesis and the second last estimate
is a Schur estimate.
We here use that $\| (\psi_s q_t^+)(T_{B^*})\|\lesssim \eta(t/s)$,
where $\eta(x):=\min(x^s,x^{-s})$ for some $s>0$.

With a similar argument $\tb{q_t^+(T_B) f}\lesssim \|f\|$, $f\in\mH$, follows
from the reverse quadratic estimate for $g\in E_{B^*}^-\V_{B^*}$.
If both reverse estimates holds for $B^*$, then
$$
  \tb{Q_t^Bf}\le \tb{q_t^-(T_B) f}+\tb{q_t^+(T_B) f}\lesssim \|f\|,\qquad f\in\mH,
$$
and if the same holds for $B$ and $B^*$ interchanged, then
Lemma~\ref{lem:dualityforquadratic} proves that 
$\tb{Q_t^B f}\approx \|f\|$, $f\in\mH$, and by 
Proposition~\ref{prop:QtoFCalc} this proves that $T_B$ has 
bounded $H_\infty(S^o_\nu)$ functional calculus, and similarly for $T_{B^*}$.
\end{proof}
%
%
%
We end this section with a discussion of the holomorphic perturbation
theory for the functional calculus of $T_B$.
\begin{defn}
  Let $z\mapsto U_z\in\mL(\mX,\mY)$ be an operator valued function defined
  on an open subset $D\subset\C$ of the complex plane.
  We say that $U_z$ is {\em holomorphic} if for all $z\in D$ there exists
  an operator $U_z'\in \mL(\mX,\mY)$ such that
$$
  \| \tfrac 1w(U_{z+w}-U_z) -U_z' \|_{\mX\rightarrow\mY}
  \longrightarrow 0,\qquad w\longrightarrow 0.
$$
\end{defn}

\begin{lem}  \label{lem:strongweakanal}
  Let $z\mapsto U_z\in\mL(\mX,\mY)$ be an operator valued function defined
  on an open subset $D\subset\C$ of the complex plane.
  Then the following are equivalent.
\begin{itemize}
\item[{\rm (i)}]
  $z\mapsto U_z$ is holomorphic.
\item[{\rm (ii)}]
  The scalar function $h(z)=(U_z f,g)$ is a holomorphic function
  for all $f\in\wt\mX$ and $g\in\wt\mY^*$, where $\wt\mX\subset\mX$ and 
  $\wt\mY^*\subset\mY^*$ are dense,
  and $\|U_z\|$ is locally bounded.
\end{itemize}
In particular, if $z\mapsto U^k_z$ are holomorphic on $D$ for $k=1,2,\ldots$ and
$$
(U^k_z f,g)=h_k(z)\longrightarrow h(z)=(U_z f,g),\qquad \text{for all } z\in D,
f\in\wt\mX, g\in \wt\mY^*,
$$
and $\sup_{z\in K,k\ge 1}\|U^k_z\|<\infty$ for each compact subset $K\subset D$, 
then $z\mapsto U_z$ is holomorphic on $D$.
\end{lem}
\begin{proof}
For the equivalence between (i) and (ii), see Kato~\cite[Theorem III 3.12]{Kato}.
To prove the convergence result, it suffices to show that $h(z)$ is holomorphic on
$D$. That this is true follows from an application of the dominated convergence theorem 
in the Cauchy integral formula for $h_k(z)$. 
\end{proof}
  Below, we shall assume that $z\mapsto B_z$ is a given holomorphic matrix valued 
function defined on an open subset $D\subset\C$ such that $B_z$ is a multiplication operator 
  as in Definition~\ref{defn:B} for each $z\in D$, and that
  $\omega_D:=\sup_{z\in D}\arccos(\kappa_{B_z}/(2\|B_z\|_\infty) )<\pi/2$.
  Let $\omega_D<\nu<\pi/2$.
\begin{lem}   \label{lem:analres}
  For $\tau\in S^0_{\pi/2-\nu}$, the operator valued function
$D\ni z\longmapsto (I - i\tau T_{B_z})^{-1}$
  is holomorphic.
\end{lem}
\begin{proof}
Similar to the proof of Proposition~\ref{prop:spectrest} we have
$$
  (I - i\tau T_{B_z})^{-1} = (\widetilde B_z\Gamma+ \wt\Gamma B_z)^{-1} 
  i (\widetilde B_z\mu- \mu^* B_z),
$$
where $\Gamma= i\mu+\tau d$ and $\wt\Gamma= -i\mu+\tau d^*$.
It is clear that the multiplication operator $\widetilde B_z\mu- \mu^* B_z$
depends holomorphically on $z$, so it suffices to show that 
$z\mapsto(\widetilde B_z\Gamma+ \wt\Gamma B_z)^{-1}$ is holomorphic.
Let $z\in D$, write $B:=B_z$ and $B_w:= B_{z+w}$ and calculate
\begin{multline}  \label{eq:analyticres}
  (\widetilde B_w\Gamma+ \wt\Gamma B_w)^{-1}-
  (\widetilde B\Gamma+ \wt\Gamma B)^{-1} \\=
  -(\widetilde B_w\Gamma+ \wt\Gamma B_w)^{-1}
  ( \widetilde B_w -\widetilde B) 
  \big(\Gamma(\widetilde B\Gamma+ \wt\Gamma B)^{-1}\big) \\
  -\big((\widetilde B_w\Gamma+ \wt\Gamma B_w)^{-1}\wt\Gamma\big)
  ( B_w -B) 
  (\widetilde B\Gamma+ \wt\Gamma B)^{-1}.
\end{multline}
The first and last factors in both terms on the right hand side are
uniformly bounded by Lemma~\ref{lem:hodge}(i) so we deduce
continuity of $w\mapsto(\widetilde B_w\Gamma+ \wt\Gamma B_w)^{-1}$.
Furthermore, multiplying equation (\ref{eq:analyticres}) from the 
right with $\wt\Gamma$, the first term on the right vanishes
as
$$
  \Gamma(\widetilde B\Gamma+ \wt\Gamma B)^{-1}\wt\Gamma=
  (\widetilde B\Gamma+ \wt\Gamma B)^{-1}\wt\Gamma^2=0,
$$
and we deduce 
continuity of $w\mapsto(\widetilde B_w\Gamma+ \wt\Gamma B_w)^{-1}\wt\Gamma$.
Thus, dividing equation (\ref{eq:analyticres}) by $w$ and letting
$w\rightarrow 0$ we see that the limit exists and equals
$$  
  -(\widetilde B\Gamma+ \wt\Gamma B)^{-1} \widetilde B'
    \big(\Gamma(\widetilde B\Gamma+ \wt\Gamma B)^{-1}\big)-
    \big((\widetilde B\Gamma+ \wt\Gamma B)^{-1} \wt\Gamma\big)
     B'(\widetilde B\Gamma+ \wt\Gamma B)^{-1}.
$$
\end{proof}

\begin{lem}  \label{lem:analpsi}
If $\psi\in \Psi(S^o_\nu)$, then
$D\ni z\longmapsto \psi(T_{B_z})$
is holomorphic.
\end{lem}
\begin{proof}
Let $\gamma$ be the unbounded contour $\sett{\pm r
e^{\pm i\theta}}{r> 0}$, $\omega_D <\theta<\nu$, parametrised
counterclockwise around $S_{\omega_D}$. 
By inspection of the proof of Lemma~\ref{lem:analres} we have
$$
  \sup_{\lambda\in\gamma}|\lambda|\left\|\tfrac 1w\big((\lambda- T_{B_{z+w}})^{-1}-(\lambda- T_{B_z})^{-1}\big)
  - \pd_z (\lambda- T_{B_z})^{-1}\right\|\rightarrow 0,
$$
as $w\rightarrow 0$.
Thus
$$
 \tfrac 1w\big( \psi(T_{B_{z+w}})-\psi(T_{B_z}) \big)
 \longrightarrow \frac 1{2\pi i}\int_\gamma \psi(\lambda)
   \pd_z (\lambda- T_{B_z})^{-1} \, d\lambda,
$$
since $\int_\gamma |\psi(\lambda)|\,\left|\tfrac {d\lambda}\lambda\right|<\infty$.
\end{proof}

\begin{lem} \label{lem:qgivesanal}
Assume that $T_{B_z}$ satisfy quadratic estimates
$\tb{ Q_t^{B_z} f} \approx \|f\|$, $f\in\mH$,
locally uniformly for $z\in D$.
If $b\in H_\infty(S^o_\nu)$, $\psi\in \Psi(S^o_\nu)$ and $0<\alpha<\beta<\infty$, then
the following operators depend holomorphically on $z\in D$.
\begin{itemize}
\item[{\rm (i)}]
$u(x)\longmapsto (b(T_{B_z})u)(x) : \mH\longrightarrow \mH$
\item[{\rm (ii)}]
$u(x)\longmapsto v(t,x)=(b(tT_{B_z})u)(x) : \mH\longrightarrow L_2(\R^n\times(\alpha,\beta);\wedge)$
\item[{\rm (iii)}]
$u(x)\longmapsto v(t,x)=(\psi_t(T_{B_z})u)(x) : \mH\longrightarrow L_2(\R^{n+1}_+, \tfrac{dtdx}t;\wedge)$
\end{itemize}
\end{lem}
\begin{proof}
(i)
Take a uniformly bounded sequence $\psi_k(z)\in \Psi(S^o_\nu)$ which
converges to $b(z)\in H_\infty(S^o_\nu)$ uniformly on compact subsets of $S^o_\nu$.
Lemma~\ref{lem:strongweakanal} then applies with $U_z^k= \psi_k(T_{B_z})$ and
$U_z= b(T_{B_z})$, using Lemma~\ref{lem:analpsi}
and Lemma~\ref{lem:convergence}.

(ii)
It suffices by Lemma~\ref{lem:strongweakanal} to show that 
$h(z) = \int_\alpha^\beta h_t(z) dt$ is holomorphic, where 
$h_t(z)= (b(tT_{B_z})f, G_t)$, for all $f(x)\in\mH$ and 
$G(t,x)\in C^\infty_0(\R^n\times(\alpha,\beta);\wedge)$.
That $h_t(z)$ is holomorphic for each $t$ is clear from (i), and
for $h(z)$ this follows from an application of the Fubini theorem 
to the Cauchy integral formula for $h_t(z)$.

(iii)
Consider the truncations 
$U_z^k : u(x) \mapsto v(t,x)=\chi_k(t)(\psi_t(T_{B_z})u)(x)$, where
$\chi_k$ denotes the characteristic function of the interval $(1/k,k)$.
It is clear from (ii) that $U_z^k$ is holomorphic, and letting $k\rightarrow\infty$ 
we deduce from Lemma~\ref{lem:strongweakanal} that $\psi_t(T_{B_z})$ is holomorphic.
\end{proof}
\begin{prop}  \label{prop:lipcont}
  Assume that $T_B$ satisfies quadratic estimates
$\tb{ Q_t^{B} f} \approx \|f\|$, $f\in\mH$, 
locally uniformly for all $B$ such that $\|B-B_0\|_\infty<\epsilon$,
and let $\psi\in \Psi(S^o_\nu)$.
Then we have the Lipschitz estimates
\begin{align*}
  \| b(T_{B_2}) - b(T_{B_1}) \| &\le C \|b\|_\infty \| B_2-B_1\|_\infty, 
    \qquad b\in H_\infty(S^o_\nu), \\
  \tb{ \psi_t(T_{B_2})- \psi_t(T_{B_1}) } &\le C \| B_2-B_1\|_\infty,
\end{align*}
when $\|B_i- B_0\|<\epsilon/2$, $i=1,2$.
\end{prop}
\begin{proof}
Let $B(z):= B_1+ z (B_2-B_1)/\|B_2-B_1\|_\infty$, so that $B(z)$ is holomorphic in a
neighbourhood of the interval $[0,\| B_2-B_1\|_\infty]$.
In this neighbourhood, we have bounds $\|b(T_{B(z)})\|\lesssim \|b\|_\infty$ by
Proposition~\ref{prop:QtoFCalc} and holomorphic dependence on $z$ by 
Lemma~\ref{lem:qgivesanal}(i).
Schwarz' lemma now applies and proves that 
$\|\tfrac{d}{dz} b(T_{B(z)})\|\lesssim \|b\|_\infty$ for all $z\in [0,\| B_2-B_1\|_\infty]$.
This shows that
$$
  \| b(T_{B_2}) - b(T_{B_1}) \| \le \int_0^{\|B_2-B_1\|} \| \tfrac{d}{dt} b(T_{B(t)}) \| dt 
  \le C \|b\|_\infty \|B_2-B_1\|_\infty.
$$
The proof of the Lipschitz estimate for $\psi_t(T_{B})$ similarly follows
from Lemma~\ref{lem:qgivesanal}(iii).
\end{proof}

%
%
%
\subsection{Decoupling of the Dirac equation}  \label{section2.5}

In section~\ref{section2.1} we introduced the Dirac type equation
\begin{equation}  \label{eq:full1}
  ( \m d_{t,x} + B^{-1} \m d_{t,x}^* B ) F =0 
\end{equation}
satisfied by functions $F(t,x): \R^{n+1}_\pm \rightarrow \wedge$.
Of particular interest is when both terms vanish, i.e. when
\begin{equation}  \label{eq:sep1}
\begin{cases}
  d_{t,x} F =0 \\
  d_{t,x}^*(BF) =0
\end{cases}
\text{  or equivalently when   }
\begin{cases}
  d F = -\mu \pd_t F \\
  d^*(BF) = \mu^* B\pd_t F 
\end{cases}.
\end{equation}
Consider a solution $F(t,x) = e^{\mp t|T_B|}f$ to (\ref{eq:full1}) in $\R^{n+1}_\pm$
as in Lemma~\ref{lem:algsplitting},
where $f, T_Bf\in E_B^\pm\V_B$. 
Using $\pd_t F = -T_B F$ and Lemma~\ref{lem:magic}, we get
\begin{multline*}
  (\m d_{t,x}F)|_{t=0} = \m (d- \mu T_B)f 
  = ( \m d - N^+ M_B^{-1}(\m d + B^{-1}\m d^* B) )f \\
  = M_B^{-1}( (M_B - B^{-1}N^+B)\m d - B^{-1}N^+ \m d^* B )f \\
  = M_B^{-1}( -N^-\m d - B^{-1}N^+\m d^* B )f 
  = M_B^{-1}( d\mu + B^{-1}d^*\mu^* B )f.
\end{multline*}
Thus we see that $d_{t,x}F= 0= d_{t,x}^*(BF)$ at $t=0$ if and only if
$d\mu f=0 = d^*\mu^* Bf$.

We may also rewrite equation (\ref{eq:full1}) as
\begin{equation}  \label{eq:full2}
  ( d_{t,x}\m  + B^{-1} d_{t,x}^* \m B) F(t,x) =0, 
\end{equation}
by using $\m d_{t,x}= \pd_t- d_{t,x}\m$ and 
$\m d^*_{t,x}= -\pd_t- d^*_{t,x}\m$ from Lemma~\ref{lem:anticom}.
Consider the case when both terms vanish, i.e. when
\begin{equation}  \label{eq:sep2}
\begin{cases}
  d_{t,x}(\m F) =0 \\
  d_{t,x}^*(\m BF) =0
\end{cases}
\text{  or equivalently when   }
\begin{cases}
  d F = \mu^* \pd_t F \\
  d^*(BF) = -\mu B\pd_t F 
\end{cases}.
\end{equation}
For $F(t,x) = e^{\mp t|T_B|}f$ solving (\ref{eq:full1}) we also have
\begin{multline*}
  d_{t,x}(\m F)|_{t=0} = \pd_t F|_{t=0} -\m d_{t,x}F|_{t=0} 
  = M_B^{-1}(d\m +B^{-1}d^*\m B)f \\ -M_B^{-1}( d\mu+B^{-1}d^*\mu^* B)f 
  = M_B^{-1}( d\mu^*+B^{-1} d^*\mu B)f.
\end{multline*}
Thus we see that $d_{t,x}(\m F)= 0= d_{t,x}^*(\m BF)$ at $t=0$ if and only if
$d\mu^* f=0 = d^*\mu Bf$.

\begin{defn}  \label{defn:Upsilon}
Introduce the closed, densely defined operators
$$
  \hat T_B := -M_B^{-1}(d\mu^* + B^{-1}d^*\mu B), \qquad
  \hut T_B := -M_B^{-1}(d\mu + B^{-1}d^*\mu^* B),
$$
with domains
\begin{align*}
  \dom(\hat T_B) &:= \sett{f\in\mH}{\mu^* f\in\dom(d),\, \mu B f\in\dom(d^*)}, \\
  \dom(\hut T_B) &:= \sett{f\in\mH}{\mu f\in\dom(d),\, \mu^* B f\in\dom(d^*)}, 
\end{align*}
respectively, and define the closed subspaces 
\begin{align*}
  \hat\mH_B &:= \sett{f\in\mH}{d(\mu f)=0= d^*(\mu^*B f)}, \\
  \hut\mH_B &:= \sett{f\in\mH}{d(\mu^* f)=0= d^*(\mu Bf)}.
\end{align*}
\end{defn}

\begin{prop}  \label{prop:huthatsplitting}
We have a topological splitting of $\mH$ into closed subspaces
$$
  \mH = \hat\mH_B \oplus \hut\mH_B.
$$
Furthermore we have $T_B = \hat T_B+ \hut T_B$ with 
$\dom(T_B)= \dom(\hat T_B)\cap \dom(\hut T_B)$,
and
$$
  \clos{\ran(\hat T_B)} =\hat\mH_B = \nul(\hut T_B), \qquad
  \clos{\ran(\hut T_B)} =\hut\mH_B = \nul(\hat T_B).
$$
Thus, if we identify $\hat T_B$ with its restriction to $\hat\mH_B$
and $\hut T_B$ with its restriction to $\hut\mH_B$, then
these are bisectorial operators with spectral and resolvent estimates
as in Proposition~\ref{prop:spectrest}.
\end{prop}
\begin{proof}
Clearly 
$\dom(\hat T_B)\cap \dom(\hut T_B)=\dom(d)\cap B^{-1}\dom(d^*)=\dom(T_B)$
and $\hat T_B+ \hut T_B=T_B$.
Furthermore Lemma~\ref{lem:hodge}(i) shows that 
$\nul(\hut T_B)=\hat\mH_B$ and $\nul(\hat T_B)=\hut\mH_B$.

(1)
To show $\ran(\hat T_B) \subset\hat\mH_B$, let $f\in\dom(\hat T_B)$ and use 
Lemma~\ref{lem:magic} to get
\begin{multline*}
  d\mu(\hat T_Bf) = -d\mu M_B^{-1}(d\mu^* + B^{-1} d^*\mu B)f 
  = -d \m N^+ M_B^{-1}(d\mu^* + B^{-1} d^*\mu B)f \\
  = -d \m M_B^{-1} B^{-1}N^+B(d\mu^* + B^{-1} d^*\mu B)f 
  = -d \m M_B^{-1} (B^{-1}N^+B -N^-)d\mu^* f \\
  = -d \m d\mu^* f=  \m d^2 \mu^* f=0,
\end{multline*}
and similarly
\begin{multline*}
  d^*\mu^*B(\hat T_Bf) = -d^*\m B(B^{-1}N^-B) M_B^{-1}(d\mu^* + B^{-1} d^*\mu B)f \\
  = -d^* \m B  M_B^{-1} N^-(d\mu^* + B^{-1} d^*\mu B)f 
  = -d^* \m B M_B^{-1} (N^- - B^{-1}N^+B) B^{-1}d^*\mu B f \\
  = d^* \m d^*\mu B f= - \m (d^*)^2 \mu B f=0.
\end{multline*}
A similar calculation shows that $\ran(\hut T_B) \subset\hut\mH_B$.

(2)
Next we show that for all $f\in\dom(T_B)$, we have
$$
  \|T_B f\| \approx \|\hat T_B f\| + \| \hut T_B f\|.
$$
Using that $M_B$ and $B$ are isomorphisms, Lemma~\ref{lem:hodge}(i) 
with $\Gamma= d$ and $\wt\Gamma= d^*$
and orthogonality $\|\mu g\|^2+ \|\mu^* g\|^2= \|g\|^2$, we obtain
\begin{multline*}
  \|\hat T_B f\| + \| \hut T_B f\|
  \approx ( \|\mu^*d f \|+ \|\mu d^* Bf\| ) +
    ( \|\mu df\|+ \|\mu^* d^* Bf\| ) \\
    \approx \|df\|+ \|d^* Bf\| \approx \|T_B f\|.
\end{multline*}

(3)
Clearly $\hat\mH_B\cap\hut\mH_B=\{0\}$ and $\ran(T_B)\subset\ran(\hat T_B)+\ran(\hut T_B)$.
Taking closures, using that $T_B$ has dense range in $\mH$ and
using (2) yields
$$
  \mH = \clos{\ran(\hat T_B)}\oplus
    \clos{\ran(\hut T_B)}.
$$    
Thus from (1) it follows that 
$\clos{\ran(\hat T_B)}=\hat\mH_B$ and $\clos{\ran(\hut T_B)}=\hut\mH_B$
and that
$\mH = \hat\mH_B \oplus \hut\mH_B$.
\end{proof}
It follows that $T_B$ is diagonal in the splitting $\mH=\hat\mH_B\oplus\hut\mH_B$.
We shall now further decompose the subspace $\hat\mH_B$ into the subspaces
of homogeneous $k$-vector field $\wedge^k$, which also are preserved by $T_B$.
The same decomposition can be made for $\hut\mH_B$, but this is not useful
since $T_B$ does not preserve these subspaces.
\begin{defn}
Let $\mH^k := L_2(\R^n;\mL(\wedge^k))$ and 
$\hat\mH_B^k := \mH^k \cap \hat\mH_B$.
\end{defn}
\begin{lem}  \label{lem:preservehats}
  The operator $T_B$ preserve all subspaces $\hat\mH_B^0$, $\hat\mH_B^1$,
  $\hat\mH_B^2$, $\ldots$, $\hat\mH_B^{n+1}$ and $\hut\mH_B$
  (but not the subspaces $\mH^k$) and we have a splitting 
$$
  \hat\mH_B= \hat\mH_B^0\oplus\hat\mH_B^1\oplus\hat\mH_B^2\oplus
  \ldots\oplus\hat\mH_B^{n+1}.
$$ 
  Furthermore, for the operators $\hat N_B$, $\hut N_B$ and $N$, we have 
mapping properties
$$
  \hat N_B : \hat\mH_B^k \longrightarrow \hat\mH_B^k,\quad
  \hut N_B:\hut\mH_B\longrightarrow\hut\mH_B\quad \text{and}\quad
  N: \hat\mH^1\longrightarrow \hat\mH^1.
$$
\end{lem}
\begin{proof}
As $d\mu^*$, $d^*\mu$, $B$ and $N^\pm$ all preserve $\mH^k$ it follows
that $T_B(\hat\mH_B^k)=\hat T_B(\hat\mH_B^k)\subset\hat\mH_B^k$.
To show the splitting of $\hat\mH_B$ it suffices to prove
$$
  \hat\mH_B\subset \hat\mH_B^0\oplus\hat\mH_B^1\oplus\hat\mH_B^2\oplus
  \ldots\oplus\hat\mH_B^{n+1}.
$$ 
To this end, let $f\in\hat\mH_B$ and write $f_k$ for the $\wedge^k$
part of $f$.
Since $d\mu f=0=d^*\mu^*Bf$ we get
\begin{align*}
  0 &= (d\mu f)_{k+2}= d((\mu f)_{k+1}) = d\mu(f_k), \\
  0 &= (d^*\mu^*B f)_{k-2}= d^*((\mu^* Bf)_{k-1}) = d^*\mu^*((Bf)_k)
  = d^*\mu^*B(f_k),
\end{align*}
for all $k$. Thus $f_k\in \hat\mH_B^k$ and $f\in\bigoplus_k \hat\mH_B^k$.

To prove the mapping properties for $\hat N_B$, note that if 
$f=f_1+f_2$ in the splitting $\mH= B^{-1}N^+\mH\oplus N^-\mH$, then
$f\in\hat\mH_B$ if and only if $\mu f_2\in\dom(d)$ and 
$\mu Bf_1\in\dom(d^*)$, according to Definition~\ref{defn:Upsilon}.
Clearly $f_1-f_2=\hat N_B(f)\in \hat\mH_B$ if $f\in\hat\mH_B$.
Since $\hat N_B$ preserves $\mH^k$, the desired mapping property follows.
The proofs for $\hut N_B$ and $N$ are similar.
\end{proof}
\begin{lem}  \label{lem:huthatdualities}
  With $\dual\cdot\cdot_B$ denoting the duality from 
  Definition~\ref{defn:duality}, we have dual operators
$$
  (\hat T_B)' = -\hat T_{B^*}, \qquad
  (\hut T_B)' = -\hut T_{B^*},
$$
and restricted dualities
$\dual{\hat\mH_B}{\hat\mH_{B^*}}_B$, $\dual{\hut\mH_B}{\hut\mH_{B^*}}_B$
and $\dual{\hat\mH_B^k}{\hat\mH_{B^*}^k}_B$ for all $k$.
In the case $k=1$, we shall write the duality as $\dual{\hat\mH^1}{\hat\mH^1}_A$.
\end{lem}
\begin{proof}
The proofs of $\hat T_B' = -\hat T_{B^*}$ and
$\hut T_B' = -\hut T_{B^*}$ are similar to that of
$T_B'=-T_{B^*}$ in Proposition~\ref{prop:duality}.
From this we get that the annihilator of 
$\hat\mH_B= \clos{\ran(\hat T_B)}$ is
$\nul(\hat T_{B^*})= \hut\mH_{B^*}$ which is a complement
of $\hat\mH_{B^*}$. Thus we see from Remark~\ref{rem:restrdualities} that 
$\dual{\hat\mH_B}{\hat\mH_{B^*}}_B$ is a duality. The proof of the duality 
$\dual{\hut\mH_B}{\hut\mH_{B^*}}_B$ is similar.
Also, since $BN^+-N^-B$ in the definition of $\dual\cdot\cdot_B$
preserves $\mH^k$ we also have dualities
$\dual{\hat\mH_B^k}{\hat\mH_{B^*}^k}_B$ for all $k$.
\end{proof}
\begin{rem}   \label{rem:huthatV_B}
\begin{itemize}
\item
The subspace of vector fields with a curl free tangential part 
$$
  \hat\mH_B^1 = \sett{f\in L_2(\R^n;\wedge^1)}{d(e_0\wedg f)=0}
$$
is independent of $B$ and coincides with the space $\hat\mH^1$ from section~\ref{section1.1}. 
Furthermore, the operator $T_A$ there coincides with
$T_B|_{\hat\mH^1}= \hat T_B|_{\hat\mH^1}$.
\item
The dense subspace $\V_B\subset \mH$ splits
$\V_B= \hat\V_B\oplus \hut\V_B$ with 
Proposition~\ref{prop:huthatsplitting}, where
$\hat\V_B:= \V_B\cap\hat\mH_B\subset\hat\mH_B$ and
$\hut\V_B:= \V_B\cap\hut\mH_B\subset\hut\mH_B$ are dense subspaces,
and we can further decompose $\hat\V_B$ into homogeneous
$k$-vector fields
$$
\hat\V_B= \hat\V_B^0\oplus\hat\V_B^1\oplus\ldots\oplus\hat\V_B^{n+1},
$$
where $\hat\V_B^k := \V_B\cap \hat\mH^k_B$.
\item
Furthermore, all these dense subspaces 
$\hat\V_B^k$, $\hat\V_B$, $\hut\V_B$ of 
$\hat\mH_B^k$, $\hat\mH_B$, $\hut\mH_B$
splits algebraically
into Hardy spaces similar to Lemma~\ref{lem:algsplitting}, e.g.
$\hat\V_B^k= E_B^+\hat\V_B^k + E_B^-\hat\V_B^k$.
Here $f\in E_B^\pm\hat\V_B$
if and only if $F(t,x)$ satisfies (\ref{eq:sep1}), and $f\in E_B^\pm\hut\V_B$
if and only if $F(t,x)$ satisfies (\ref{eq:sep2}).
\end{itemize}
\end{rem}

%
%
%
\subsection{Operator equations and estimates for solutions}    \label{section2.6}

Our objective in this section is to set up our boundary operator method 
for solving the boundary value problems (Neu-$A$), (Reg-$A$), (Neu$^\perp$-$A$) 
and (Dir-$A$) as well as the transmission problem (Tr-$B^k\alpha^\pm$).
Under the assumption that $T_B$ has quadratic estimates, which is made
throughout this section, we first show that 
solutions $F(t,x)$ are determined by their traces $f$
through the reproducing Cauchy type formula $F_t = e^{-t|T_B|}f$.
Recall from Lemma~\ref{lem:preservehats} that both operators $E_B$ and $N_B$
preserve all subspaces $\hat\mH^k_B$.
We shall write $E_{B^k}= E_B|_{\hat\mH^k_B}= \sgn(T_B|_{\hat\mH^k_B})$
and $N_{B^k}= N_B|_{\hat\mH^k_B}$ for the restrictions.
In particular, we write $A= B^1$ when $k=1$.
\begin{lem}  \label{lem:characthardyfcns}
Assume that $T_{B}$ satisfies quadratic estimates, 
let $f\in\mH^k$ and let
$(0,\infty)\ni t\mapsto F_t(x)=F(t,x)\in\mH^k$ be a family of functions.
Then the following are equivalent.
\begin{itemize}
\item[{\rm (i)}]
  $f\in E^+_{B}\hat\mH^k_B$ and $F_t= e^{-t|T_{B}|}f$.
\item[{\rm (ii)}]
  $F_t\in C^1(\R_+; \mH^k)$ and satisfies the equations
\begin{equation*} 
\begin{cases}   
  d^*_{t,x} (B(x) F(t,x))  =0, \\
  d_{t,x} F(t,x) =0,
\end{cases}
\end{equation*}
and have $L_2$ limits
$\lim_{t\rightarrow 0^+} F_t= f$ and
$\lim_{t\rightarrow \infty} F_t= 0$.
\end{itemize}
In fact, such $F_t$ belong to $C^j(\R_+; \mH^k)$ for all $j\ge 1$ and 
are in one-to-one correspondence with the trace 
$f\in E^+_{B}\hat\mH^k_{B}$, and we have equivalences of norms
$$
  \|f\|\approx \sup_{t>0} \|F_t\|\approx \tb{t\pd_t F_t}.
$$
The corresponding reproducing formula $F_t= e^{t|T_{B}|}f$, $f= F|_{\R^n}\in E^-_{B}\hat\mH^k_{B}$ 
is also valid for $F$ solving the equations in $\R^{n+1}_-$, and the corresponding estimates holds.
\end{lem}
\begin{proof}
(i) implies (ii):
As in the proof of Lemma~\ref{lem:algsplitting}, from Lemma~\ref{lem:convergence} it follows that 
$\lim_{t\rightarrow 0}F_t =f$ and $\lim_{t\rightarrow\infty}F_t=0$,
and also that $\pd_t^j F_t= (-|T_A|)^j e^{-t|T_A|}f$. Therefore $F_t\in C^j(\R_+;\mH^k)$
for all $j$.
For $j=1$, we get
$0= (\partial_t+ |T_{B}|)F_t = (\partial_t+ T_{B})F_t$,
since $F_t\in E^+_{B}\hat\mH^k_B$.
As explained in Section~\ref{section2.5} this is equivalent
with the two equations $d_{t,x} F =0$ and
$d^*_{t,x} (B F)  =0$.

(ii) implies (i):
The two equations can be written
$(-\mu^*\partial_t+d^*)BF=0$ and $(\mu\partial_t+d)F=0$.
Applying $\mu^*$ and $\mu$ respectively to these equations and
using nilpotence, we obtain
$\mu^*d^* BF=0$ and $\mu dF=0$.
Therefore $F_t\in\hat\mH^k_{B}$, and since $\hat\mH^k_{B}$ is closed
we also have $f\in\hat\mH^k_{B}$.

Next we write $F_t= F_t^+ + F_t^-$, where 
$F_t^\pm:= E^\pm_{B} F_t \in E_{B}^\pm\hat\mH_B^k$ and similarly $f= f^+ +f^-$.
We rewrite the equations satisfied by $F_t$ as
$\partial_t F_t+ T_{B} F_t=0$. Applying the projections $E^\pm_{B}$ 
to this equation yields
\begin{align*}
  \partial_t F_t^++ |T_{B}| F_t^+ & = 0, \\
  \partial_t F_t^-- |T_{B}| F_t^- & = 0, 
\end{align*}
since $T_{B} F_t^\pm= \pm |T_{B}| F_t^\pm$.
Fix $t>0$. Then it follows that
$e^{(t-s)|T_{B}|}F_s^-$ is constant for $s\in (t,\infty)$ and that
$e^{(s-t)|T_{B}|}F_s^+$ is constant for $s\in (0,t)$.
Taking limits, this shows that 
$F_t^-=\lim_{s\rightarrow t^+}e^{(t-s)|T_{B}|}F_s^-
=\lim_{s\rightarrow \infty}e^{(t-s)|T_{B}|}F_s^-= 0$ and 
$F_t^+= \lim_{s\rightarrow t^-} e^{(s-t)|T_{B}|}F_s^+
= \lim_{s\rightarrow 0} e^{(s-t)|T_{B}|}F_s^+ = e^{-t|T_{B}|}f^+$.
Therefore $f=f^+\in E^+_{B}\hat\mH^k_B$ and $F_t= e^{-t|T_{B}|}f$.

To prove the norm estimates we note that $\|f\|=\lim_{t\rightarrow 0}\|F_t\|$.
By Proposition~\ref{prop:QtoFCalc}, $e^{- t|T_{B}|}$ are
uniformly bounded and thus $\sup_{t>0}\|F_t\| \lesssim \|f\|$.
Furthermore, using Proposition~\ref{prop:differentpsi} with
$\psi(z)= z e^{-|z|}$ shows that 
$\tb{t\partial_t F_t}\approx \|f\|$.
\end{proof}
\begin{rem}
In proving that (ii) implies (i), it suffices to assume that 
$\|F_t\|$ grows at most polynomially when $t\rightarrow\infty$.
Indeed, from the equation $F_t^-= e^{-(s-t)|T_{B}|}F_s^-$ for $s>t$, it
then follows that 
$$
  \|T_{B}^kF_t^-\|= (s-t)^{-k}\|((s-t)T_{B})^k e^{-(s-t)|T_{B}|}F_s^-\|
  \le C  (s-t)^{-k}\|F_s\| \longrightarrow 0,
$$
when $s\rightarrow \infty$.
Since $T_{B}$ is injective, this shows that $F_t^-=0$ and therefore that
$F_t= e^{-t|T_{B}|}f \in E_{B}^+\hat\mH_{B}^k$ as before.
\end{rem}
We now proceed by showing how, given data $g$ in (Tr-$B^k\alpha^\pm$),
we can solve for the trace $f= F|_{\R^n}$ by using the boundary operators
$E_B$ and $N_B$.
\begin{lem}  \label{lem:optransmain}
Assume that $T_{B}$ satisfies quadratic estimates, so that the Hardy projections
$E^\pm_{B}$ are bounded by Proposition~\ref{prop:QtoFCalc}, 
let $\alpha^\pm\in\C$ be given jump parameters and define the associated
spectral point $\lambda := (\alpha^+ + \alpha^-)/(\alpha^+ - \alpha^-)$.
Then 
$$
  \lambda- E_{B^k}N_{B^k}: \hat\mH^k_B \longrightarrow \hat\mH^k_B
$$
is an isomorphism if and only if the transmission problem 
(Tr-$B^k\alpha^\pm$) is well posed.
\end{lem}
\begin{proof}
If we identify the $k$-vector fields $F^\pm(t,x)$ in the transmission problem
(Tr-$B^k\alpha^\pm$) with the boundary traces $f^\pm$ and
write $f= f^+ + f^-$ using Lemma~\ref{lem:characthardyfcns}, then we see that 
the transmission problem is equivalent with the system of equations
\begin{equation*}
\begin{cases}   
  N^+_{B^k}( \alpha^- E_{B^k}^+ - \alpha^+ E_{B^k}^-) f = N_{B^k}^+ g, \\
  N^-_{B^k}( \alpha^+ E_{B^k}^+ - \alpha^- E_{B^k}^-) f = N_{B^k}^- g.
\end{cases}
\end{equation*}
Using $E_{B^k}^\pm= \tfrac 12(I\pm E_{B^k})$ and adding up the equations, we 
see that the system is equivalent with the equation
$$
  (\lambda- E_{B^k}N_{B^k})f = \frac 2{\alpha^+-\alpha^-} E_{B^k}g.
$$
This proves the lemma.
\end{proof}
Next we consider $k=1$ and the boundary value problems (Neu-$A$), (Reg-$A$) and 
(Neu$^\perp$-$A$). 
By Lemma~\ref{lem:characthardyfcns}, we have the following.
\begin{enumerate}
\item (Neu-$A$) is well posed if and only if the restricted projection
      $N^-_A: E_A^+\hat\mH^1\rightarrow  N_A^-\hat\mH^1$ is an isomorphism.
\item (Reg-$A$) is well posed if and only if the restricted projection
      $N^+_A: E_A^+\hat\mH^1\rightarrow  N_A^+\hat\mH^1$ is an isomorphism,
      or equivalently if and only if 
      $N^+: E_A^+\hat\mH^1\rightarrow  N^+\hat\mH^1$ is an isomorphism.
      Note that both $N_A^+$ and $N^+$ project along $N^-\mH^1$.
\item (Neu$^\perp$-$A$) is well posed if and only if the restricted projection
      $N^-: E_A^+\hat\mH^1\rightarrow  N^-\hat\mH^1$ is an isomorphism.
\end{enumerate}
\begin{prop}  \label{prop:regneueq}
  Assume that $T_A$ satisfies quadratic estimates.
Then (Reg-$A$) is well posed if and only if (Neu$^\perp$-$A^*$) is well posed.
\end{prop}
\begin{proof}
We need to show that if $N^+: E_A^+\hat\mH^1\rightarrow  N^+\hat\mH^1$ is an isomorphism,
then so is $N^-: E_{A^*}^+\hat\mH^1\rightarrow  N^-\hat\mH^1$.
The proof uses two facts. First that we have adjoint operators $(E_A)'= -E_{A^*}$
and $N'=N$ according to Proposition~\ref{prop:duality} and Lemma~\ref{lem:huthatdualities}.
As in Remark~\ref{rem:restrdualities}, this shows that we have dual spaces
$\dual{E_A^+\hat\mH^1}{E_{A^*}^-\hat\mH^1}_A$ and $\dual{N^+\hat\mH^1}{N^+\hat\mH^1}_A$,
and we see that
$$
  \dual{N^+f}{g}_A=\dual{f}{g}_A= \dual{f}{E_{A^*}^-g}_A,
$$
for all $f\in E_A^+\hat\mH^1$ and $g\in N^+\hat\mH^1$.
Therefore, the restricted projections 
$N^+: E_A^+\hat\mH^1\rightarrow  N^+\hat\mH^1$ and
$E_{A^*}^- : N^+\hat\mH^1\rightarrow E_{A^*}^-\hat\mH^1$ are adjoint.

Secondly, if $R_1^\pm$ and $R_2^\pm$ are two pairs of complementary projections
in a Hilbert space $\mH$, as in Definition~\ref{defn:splittings}, then 
$R_1^-: R_2^+\mH\rightarrow R_1^-\mH$ has a priori estimates, as in Remark~\ref{rem:aprioriduality}, 
if and only
if $R_2^-: R_1^+\mH\rightarrow R_2^-\mH$ has a priori estimates.
Indeed, both statements are seen to be equivalent with that 
the subspaces $R_1^+\mH$ and $R_2^+\mH$ are transversal, i.e. that the 
estimate $\|f_1+f_2\|\approx \|f_1\|+\|f_2\|$ holds for all $f_1\in R_1^+\mH$
and $f_2\in R_2^+\mH$.
To see this, assume that $\|R_1^-f_2\|\gtrsim \|f_2\|$ holds for all $f_2\in R_2^+\mH$.
Then $\|f_2\|\lesssim \|R_1^-(f_1+f_2)\|\lesssim\|f_1+f_2\|$ for all $f_i\in R^+_i\mH$,
which proves transversality.
Conversely, assume that $R_1^+\mH$ and $R^+_2\mH$ are transversal. 
Then $f_2-R_1^-f_2= R^+_1 f_2 =:f_1\in R^+_1\mH$ for all $f_2\in R_2^+\mH$.
Therefore $\|R^-_1f_2\|=\|f_2-f_1\|\approx \|f_2\|+\|f_1\|\ge\|f_2\|$.
The same argument can be used to show that transversality also holds if and only if
$\|R_2^-f_1\|\gtrsim \|f_1\|$ holds for all $f_1\in R_1^+\mH$.

To prove the proposition, assume that $N^+: E_A^+\hat\mH^1\rightarrow  N^+\hat\mH^1$ is an isomorphism.
It follows that the adjoint operator 
$E_{A^*}^- : N^+\hat\mH^1\rightarrow E_{A^*}^-\hat\mH^1$ also is an
isomorphism.
Using the second fact above twice, shows that
$$
  E_{A}^- : N^-\hat\mH^1\longrightarrow E_{A}^-\hat\mH^1 \qquad\text{and}\qquad
  N^-: E_{A^*}^+\hat\mH^1\longrightarrow  N^-\hat\mH^1
$$
have a priori estimates.
As these are adjoint operators as well, both must in fact be isomorphisms.
In particular we have shown that (Neu$^\perp$-$A^*$) is well posed.
The proof of the converse implication is similar.
\end{proof}
When perturbing $A$, it is preferable not to have operators defined on spaces
like $E^+_A\hat\mH$, which varies with $A$. 
We have the following.
\begin{lem}  \label{lem:opmain}
Assume that $T_A$ satisfies quadratic estimates.
\begin{enumerate}
\item
If $I-E_AN_A : \hat\mH^1\rightarrow \hat\mH^1$ is an isomorphism, then the Neumann
problem (Neu-$A$) is well posed.
\item
If $I+E_AN_A : \hat\mH^1\rightarrow \hat\mH^1$ is an isomorphism, or if
$I+E_AN : \hat\mH^1\rightarrow \hat\mH^1$ is an isomorphism, then the regularity
problem (Reg-$A$) is well posed.
\item
If $I-E_AN : \hat\mH^1\rightarrow \hat\mH^1$ is an isomorphism, then the Neumann
problem (Neu$^\perp$-$A$) is well posed.
\end{enumerate}
\end{lem}
\begin{proof}
Assume for example that $I+E_AN_A$ is an isomorphism.
We need to prove that
$N^+_A : E^+_A \hat\mH^1 \rightarrow N^+_A \hat\mH^1$
is an isomorphism.
Note that if $N^+_A f=g$ where $f\in E^+_A\hat\mH^1$, then
$$
  g=N_A^+f= \tfrac 12(I+N_A)f= \tfrac 12(E_A+N_A)f= \tfrac 12 E_A(I+E_AN_A)f.
$$
If $g\in N_A^+\hat\mH^1$, let $f:= 2(E_A+N_A)^{-1}g$. Then it follows
that $0= N^-_A g= \tfrac 12(E_A+N_A)(E_A^-f)$, since 
$N_A^-(E_A+N_A)=\tfrac 12(E_A+N_A-I-N_AE_A)=(E_A+N_A)E_A^-$, so 
$E_A^- f=0$ and therefore $f\in E_A^+\hat\mH^1$.

A similar calculation proves well posedness of the other boundary value problems.
\end{proof}
\begin{rem}  \label{rem:rotasprojs}
More generally, letting $k=1$ and $(\alpha^+, \alpha^-)=(1,0)$, i.e.
$\lambda=1$, in Lemma~\ref{lem:optransmain}, we see that 
$I-E_AN_A$ is an isomorphism if and only if
the restricted projections
\begin{align*}
  N_A^+ &: E_A^-\hat\mH^1 \rightarrow N_A^+\hat\mH^1\quad \text{and} \\
  N_A^- &: E_A^+\hat\mH^1 \rightarrow N_A^-\hat\mH^1
\end{align*}
are isomorphisms.
Similarly, if $(\alpha^+, \alpha^-)=(0,1)$, i.e.
$\lambda=-1$ in Lemma~\ref{lem:optransmain}, we see that 
$I+E_AN_A$ is an isomorphism if and only if
the restricted projections
$N_A^+ : E_A^+\hat\mH^1 \rightarrow N_A^+\hat\mH^1$ and
$N_A^- : E_A^-\hat\mH^1 \rightarrow N_A^-\hat\mH^1$
are isomorphisms.
\end{rem}

We next turn to the Dirichlet problem (Dir-$A$), where we aim to prove an analogue
of Lemma~\ref{lem:characthardyfcns} which characterises the solution $U_t$ as a 
Poisson integral of the boundary trace $u$. 
As discussed in the introduction, we shall use (Neu$^\perp$-$A$) to construct 
the solution $U_t$.
Given Dirichlet data $u\in L_2(\R^n;\C)$, we form $ue_0\in N^-\hat\mH^1$.
It then follows from Lemma~\ref{lem:LaplaceF_0^1} that the vector field $F_t$ 
solving (Neu$^\perp$-$A$) with data $ue_0$, has a normal component $U:= F_0$
which satisfies the second order equation (\ref{eq:divform}).
We now define the {\em Poisson integral} of $u$ to be
$$
  \mP_t(u) := (F_t,e_0) ,\qquad \text{when } F_t= e^{-t|T_A|}f \, \text{ and }\,
  (f,e_0)=u.
$$
\begin{lem}  \label{lem:characterisePoisson}
Assume that $T_A$ satisfies quadratic estimates and that the 
Neumann problem (Neu$^\perp$-$A$) is well posed.
Let $u\in L_2(\R^n;\C)$ and let
$(0,\infty)\ni t\mapsto U_t(x)=U(t,x)\in L_2(\R^n;\C)$ be a family of functions.
Then the following are equivalent.
\begin{itemize}
\item[{\rm (i)}]
  $U_t= \mP_t u$ for all $t>0$.
\item[{\rm (ii)}]
  $U_t\in C^2(\R_+; L_2(\R^n;\C))$, $\nabla_{t,x} U_t\in C^1(\R_+; L_2(\R^n;\C^{n+1}))$ and
  $U$ satisfies the equation 
$$
  \divv_{t,x} A(x) \nabla_{t,x} U(t,x) =0,
$$
and we have $L_2$ limits
$\lim_{t\rightarrow 0^+} U_t= u$,
$\lim_{t\rightarrow \infty} U_t= 0$ and
$\lim_{t\rightarrow \infty}\nabla_{t,x} U_t=0$.
\end{itemize}
If this holds, then $U_t, \nabla_{t,x}U_t \in C^j(\R_+; L_2(\R^n))$ for all $j\ge 1$.
Furthermore, $U_t$ is in one-to-one correspondence with the trace 
$u$, and we have equivalences of norms
$$
  \|u\|\approx \sup_{t>0} \|U_t\|
  \approx \tb{t\pd_t U_t}.
$$
\end{lem}
\begin{proof}
(i) implies (ii):
Assume that $U_t=( F_t, e_0)$, where $F_t= e^{-t|T_A|}f$ and $(f,e_0)=u$.
As in the proof of Lemma~\ref{lem:algsplitting}, from Lemma~\ref{lem:convergence} it follows that 
$F_t\rightarrow f$, and therefore $U_t\rightarrow u$,
and also that $\pd_t^j F_t= (-|T_A|)^j e^{-t|T_A|}f$. Therefore $F_t\in C^j(\R_+; L_2(\R^n))$
for all $j$, and so does $U_t$.
For $j=1$, we see that $F$ satisfies the Dirac type equation since $|T_A|F_t=T_AF_t$, and
Lemma~\ref{lem:LaplaceF_0^1} thus shows that $U_t$ satisfies the second order equation.
Furthermore, we note from the expression (\ref{eq:blockTA}) for $T_A$, that 
$\nabla_x U_t= (\partial_t F_t)_\ta$.
Thus $\nabla_x U_t\in C^j(\R_+; L_2(\R^n;\C))$ for all $j$, and yet another application of
Lemma~\ref{lem:convergence} shows that $U_t=o(1)$ and $\nabla_{t,x} U_t= o(1/t)$ when $t\rightarrow\infty$.

(ii) implies (i):
Assume $U_t$ has the stated properties and boundary trace $u$.
Consider the family of vector fields $G_t:= \nabla_{t,x} U_t$.
Since these satisfy Lemma~\ref{lem:characthardyfcns}(ii) for
$t\ge s>0$ with boundary trace $G_s$, we obtain that
$G_{s+t}= e^{-t|T_A|}G_s$ for all $s,t>0$.
For the normal components, this means that 
$$
  \pd_0 U_{s+t}= \mP_t(\pd_0 U_s),
$$
or equivalently that $\pd_s(U_{s+t}-\mP_t(U_s))=0$.
Since $\lim_{s\rightarrow\infty} U_s=0$, we must have $U_{s+t}= \mP_t(U_s)$ for all $s,t>0$.
Letting $s\rightarrow 0$, we conclude that $U_t=\mP_t(u)$.

The equivalence of norms $\|u\|\approx \sup_{t>0} \|U_t\|$ follows from the
uniform boundedness of the operators $\mP_t$.
For the equivalence $\|u\|\approx \tb{t\pd_t U_t}$ we use that (Neu$^\perp$-$A$) is
well posed and the corresponding square function estimate for $F_t$ from 
Lemma~\ref{lem:characthardyfcns} to get
$\|u\|\approx \|f\| \approx \tb{t\pd_t F_t} \approx \tb{t\pd_t U_t}$,
since for all $t>0$ we have $\|\pd_t U_t\| = \|N^-(\pd_t F_t)\| \approx \|\pd_t F_t\|$.
\end{proof}

We end this section with the proof of the non-tangential estimate 
$\|\widetilde N_*(F)\|\approx\|f\|$ in Theorem~\ref{thm:main}.
\begin{prop}   \label{prop:modnontang}
Assume that $T_A$ satisfies quadratic estimates.
Let $F_t= e^{-t |T_A|}f$, where $f\in E^+_A \hat\mH^1$.
Then $\|f\|\approx \|\widetilde N_*(F)\|$,
where the non-tangential maximal function is 
$$
  \widetilde N_*(F)(x):= \sup_{t>0} \left(\barint_{}
  \barint_{\hspace{-6pt} D(t,x)} |F(s,y)|^2 ds\,dy \right)^{1/2},
$$
and $D(t,x):= \sett{(s,y)\in \R^{n+1}_+}{|s-t|<c_0t,\, |y-x|<c_1t}$,
for given constants $c_0\in(0,1)$ and $c_1>0$.
\end{prop}

The proof uses the following lemma.

\begin{lem}  \label{lem:H_t}
Let $f\in \hat\mH^1$ and define $H_t= (1+ it T_B)^{-1} f\in \hat\mH^1$.
Write $H_t= H_t^{1,0}e_0 + H_t^{1,\ta}$
and $f= f^{1,0}e_0 + f^{1,\ta}$.
\begin{itemize}
\item[{\rm (i)}]
The normal component $H_t^{1,0}$ satisfies the second order
divergence form equation
$$
     \begin{bmatrix}
       1 & it \divv
     \end{bmatrix}
     \begin{bmatrix}
       a_{\no\no} & a_{\no\ta} \\ a_{\ta\no} & a_{\ta\ta}
     \end{bmatrix}
     \begin{bmatrix}
       1 \\ it \nabla
     \end{bmatrix}
H^{1,0}_t =
     \begin{bmatrix}
       1 & it \divv
     \end{bmatrix}
     \begin{bmatrix}
       a_{\no\no} f^{1,0} \\ -a_{\ta\ta} f^{1,\ta}
     \end{bmatrix},
$$
where we identify normal vectors $ue_0$ with scalars $u$,
and the tangential component $H_t^{1,\ta}$ satisfies
$$
  H_t^{1,\ta} = f^{1,\ta}+ it \nabla H_t^{1,0}.
$$
\item[{\rm (ii)}]
There exists $p<2$ and $q>2$ such that for any fixed $r_0<\infty$ we have
$$
  \left( \barint_{\hspace{-6pt} B(x,r_0 t)} |H_t^{1,0}|^q \right)^{1/q} +
  \left( \barint_{\hspace{-6pt} B(x,r_0 t)} |H_t^{1,\ta}|^p \right)^{1/p} \lesssim M(|f|^p)^{1/p}(x), 
$$
for all $x\in \R^n$ and $t>0$.
Here $M(f)(x):= \sup_{r>0} \barint_{\hspace{-3pt} B(x,r)} |f(y)| dy$ 
denotes the Hardy--Littlewood maximal function.
\end{itemize}
\end{lem}
\begin{proof}
(i)
Multiplying the equation $(1+ it T_B)H_t=f$ with $M_B$
we get
$$
  \Big( \m(\mu +it d) + B^{-1} ( \m(-\mu^*+it d^*) ) B \Big)
  H_t = M_B f.
$$
Similar to the proof of Lemma~\ref{lem:LaplaceF_0^1} we can now use 
the anticommutation relations 
from Lemma~\ref{lem:anticom} to rewrite this equation as
$$
  \Big( -(\mu +it d)\m + B^{-1} ( -(-\mu^*+it d^*)\m ) B \Big)
  H_t = M_B f.
$$
since $\{\m,\mu+itd\}=I$, $\{\m,-\mu^*+itd^*\}=-I$
and $I-B^{-1}IB=0$.
Then apply $(-\mu^*+it d^*)B$ to obtain
$$
  (\mu^*-it d^*)B(\mu +it d)(\m H_t)= (-\mu^*+it d^*)(B N^+-N^-B)f,
$$
since $-\mu^*+it d^*$ is nilpotent.
Evaluating the scalar part of this equation, we get the
desired identity.

To find the identity for $H^{1,\ta}_t$, we use the expression for 
$T_A= \hat T_B|_{\hat\mH^1}$ from Definition~\ref{defn:Upsilon}.
Multiplying the equation $(I+it \hat T_A)H_t =f$ by $AN^+-N^-A$ yields
$$
  a_{\ta\ta}H^{1,\ta}_t- a_{00} H^{1,0}_t e_0-it (A\nabla H^{1,0}_t+ d^*\mu A H_t)
= a_{\ta\ta}f^{1,\ta}- a_{00} f^{1,0}_t e_0.
$$
Evaluating the tangential part of this equation gives the desired identity.

(ii)
By rescaling, we see from (i) that it suffices to show that
\begin{equation}  \label{eq:claimH_t}
  \left( \int_{B(x, r_0)} |u(y)|^q dy \right)^{1/q} +
  \left( \int_{B(x,r_0)} |\nabla u(y)|^p dy \right)^{1/p} \lesssim M(|g|^p)^{1/p}(x), 
\end{equation}
for all $x\in \R^n$ and all $u$ and $g$ satisfying an equation
$$
     \begin{bmatrix}
       1 & i \divv
     \end{bmatrix}
     \begin{bmatrix}
       a'_{\no\no} & a'_{\no\ta} \\ a'_{\ta\no} & a'_{\ta\ta}
     \end{bmatrix}    
     \begin{bmatrix}
       1 \\ i \nabla
     \end{bmatrix}
      u =
     \begin{bmatrix}
       1 & i \divv
     \end{bmatrix}
     \begin{bmatrix}
       a'_{\no\no} g^{1,0} \\ -a'_{\ta\ta} g^{1,\ta}
     \end{bmatrix},
$$
where $A'$ is a matrix with same norm and accretivity constant
as for $A$.
Indeed, by rescaling we see that $u(x)= H_t^{1,0}(tx)$,
$g(x)= f(tx)$ and $A'(x)= A(tx)$ satisfies this hypothesis.
To prove (\ref{eq:claimH_t}), we use that the maps $g\mapsto u$ and $g\mapsto \nabla u$
have $L_p(\R^n)\rightarrow L_q(\R^n)$ and $L_p(\R^n)\rightarrow L_p(\R^n)$ 
off-diagonal bounds respectively, with exponent $M$
for any $M>0$, i.e. there exists $C_M<\infty$ such that
\begin{equation}  \label{eq:LpLqoffdiag}
  \|u\|_{L_q(E)}+ \|\nabla u\|_{L_p(E)} \le C_M \brac{\dist (E,F)}^{-M}\|g\|_p
\end{equation}
whenever $E,F \subset \R^n$ and $\supp g\subset F$.
To see this, let 
$
  L:= \begin{bmatrix}
       1 & i \divv
     \end{bmatrix}
  A'
     \begin{bmatrix}
       1 & i \nabla
     \end{bmatrix}^t
$.
Note that $L: W^1_2(\R^n)\rightarrow W^{-1}_2(\R^n)$ is an isomorphism
and that $L: W^1_p(\R^n)\rightarrow W^{-1}_p(\R^n)$
is bounded. 
Then by the stability result of {\v{S}}ne{\u\i}berg~\cite{sneiberg},
it follows that there exists $\epsilon>0$ such that 
$L: W^1_p(\R^n)\rightarrow W^{-1}_p(\R^n)$ is an isomorphism when
$|p-2|<\epsilon$.
We then fix $p_0\in (2-\epsilon, 2)$ and use Sobolev's embedding theorem
to see that
$$
  \|u\|_{q_0} +\|\nabla u\|_{p_0}\lesssim \| u \|_{W^1_{p_0}} \lesssim \| Lu \|_{W^{-1}_{p_0}}
  \lesssim  \| g\|_{p_0}.
$$
By choosing $p_0$ close to $2$, we may assume that $q_0>2$.
Thus we have bounded maps 
$g\mapsto u: L_{p_0}(F)\rightarrow L_{q_0}(E)$ and
$g\mapsto \nabla u: L_{p_0}(F)\rightarrow L_{p_0}(E)$,
with norms $\le C$.
Also, by Proposition~\ref{pseudoloc}, the norms
of $g\mapsto u: L_2(F)\rightarrow L_2(E)$ and 
$g\mapsto \nabla u: L_2(F)\rightarrow L_2(E)$
are bounded by $\brac{\dist (E,F)}^{-M_0}$.
Interpolation now proves (\ref{eq:LpLqoffdiag}) for some
$p_0<p<2$, $2<q<q_0$ and $M= M_0(1- p_0/p)/(1-p_0/2)$.

Finally we show how (\ref{eq:LpLqoffdiag}) implies (\ref{eq:claimH_t}).
Let $E= F_0:= B(x, r_0)$ and for $k\ge 1$ let 
$F_k:= B(x, 2^k r_0)\setminus B(x, 2^{k-1}r_0)$.
This gives
\begin{multline*}
  \left( \int_{B(x,r_0)} |u(y)|^q dy \right)^{1/q}+
  \left( \int_{B(x,r_0)} |\nabla u(y)|^p dy \right)^{1/p}
  \lesssim \sum_{k=0}^\infty 2^{-Mk} 
  \left( \int_{ F_k} |g(y)|^p dy \right)^{1/p} \\
  \lesssim \sum_{k=0}^\infty 2^{(n/p-M)k} 
  \left( \barint_{\hspace{-6pt} B(x, 2^k r_0)} |g(y)|^p dy \right)^{1/p}
  \lesssim ( M(|g|^p)(x) )^{1/p},
\end{multline*}
provided we chose $M> n/p$.
\end{proof}

\begin{proof}[Proof of Proposition~\ref{prop:modnontang}]
  To prove that $\|\widetilde N_*(F)\|\gtrsim \|f\|$, we calculate
\begin{multline*}
  \|\widetilde N_*(F)\|^2\gtrsim \sup_{t>0} \int_{\R^n}
\barint_{\hspace{-6pt} |y-x|<c_1t}\barint_{\hspace{-6pt} |s-t|<c_0t} |F(s,y)|^2 \, dsdydx \\
= \sup_{t>0} \barint_{\hspace{-6pt} |s-t|<c_0t} \|F_s\|^2 ds \gtrsim 
\sup_{t>0} \|F_{(1+c_0)t}\|^2\approx \|f\|^2.
\end{multline*}
We have here used that $F_{(1+c_0)t}= e^{-((1+c_0)t-s)|T_A|}F_s$, which by
Proposition~\ref{prop:QtoFCalc} shows that $\|F_{(1+c_0)t}\|\lesssim \|F_s\|$
when $|s-t|<c_0t$.

  To prove $\|\widetilde N_*(F)\|\lesssim \|f\|$, we note that since $\curl_{t,x}F=0$, 
we can write $F= \nabla_{t,x} U$ for some scalar potential $U$, and we see that $U$ solves 
the second order equation (\ref{eq:divform}).
With notation $\wt D(t,x)=\sett{(s,y)}{|s-t|<\tilde c_0 t, \, |y-x|<\tilde c_1 t}$,
for constants $c_0<\tilde c_0<1$ and $c_1<\tilde c_1<\infty$, and
$\clos U:= \barint_{}\barint_{\hspace{-3pt} \wt D(t,x)} U(s,y) ds\, dy$,
we have
\begin{multline*}
  \left( \barint_{}\barint_{\hspace{-6pt} D(t,x)} |t\nabla U(s,y)|^2 ds\,dy \right)^{1/2}
\lesssim 
  \left( \barint_{}\barint_{\hspace{-6pt} \wt D(t,x)} |U(s,y)-\clos U|^2 ds\,dy \right)^{1/2}\\
\lesssim  
\left( \barint_{}\barint_{\hspace{-6pt} 
  \wt D(t,x)} |t\nabla U(s,y)|^p ds\,dy \right)^{1/p},
\end{multline*}
with $2(n+1)/(n+3)<p<2$.
The first estimate uses Caccioppoli's inequality and the second estimate uses Poincar\'e's
inequality.
Thus it suffices to bound the $L_2$ norm of 
$\sup_{t>0}\big( \barint_{}\barint_{\hspace{-3pt} 
  \wt D(t,x)} |F(s,y)|^p\, ds\,dy \big)^{1/p}$.
To this end, we write $F_s= H_s + \psi_s(T_A)f$, where $H_s:= (I+is T_A)^{-1}f$ and
$\psi(z) := e^{-|z|}-(1+iz)^{-1}$.
Using the quadratic estimates, the second term has estimates
\begin{multline*}
  \int_{\R^n} \sup_{t>0} \left( \barint_{}\barint_{\hspace{-6pt} 
  \wt D(t,x)} |\psi_s(T_A) f(y)|^p \,ds\,dy \right)^{2/p} dx \\
\lesssim   \int_{\R^n} \sup_{t>0}\left( \int_{|s-t|<\tilde c_0t} \int_{|y-x|<\tilde c_1t} 
   |\psi_s(T_A) f(y)|^2 \,\frac {dyds}{t^{n+1}}\right)dx \\
\lesssim \int_0^\infty  \int_{\R^n} \int_{|y-x|<\tilde c_1s/(1-\tilde c_0)} 
   |\psi_s(T_A) f(y)|^2 s^{-(n+1)} dydx \, ds \\
\lesssim \int_0^\infty \|\psi_s(T_A) f\|^2\frac {ds}s\lesssim \|f\|^2.
\end{multline*}
For the first term, we use Lemma~\ref{lem:H_t}(ii) and obtain
\begin{multline*}
\int_{\R^n} \sup_{t>0} \left( \barint_{}\barint_{\hspace{-6pt} 
  \wt D(t,x)} |H_s(y)|^p \,ds\,dy \right)^{2/p} dx \\
\lesssim 
\int_{\R^n} \sup_{t>0} \sup_{|s-t|<\tilde c_0t} 
\left( \barint_{\hspace{-6pt} B(x,\tilde c_1s/(1-\tilde c_0))} |H_s(y)|^p \,dy \right)^{2/p} dx \\
\lesssim 
  \| M(|f|^p) \|_{2/p}^{2/p} \lesssim \| |f|^p \|_{2/p}^{2/p}
  = \| f\|_2^2,
\end{multline*}
using the boundedness of the Hardy--Littlewood maximal function
on $L_{2/p}(\R^n)$.
This completes the proof.
\end{proof}

%
%
%
\section{Invertibility of unperturbed operators}  \label{section7}

In this section we prove Theorem~\ref{thm:transmain} and Theorem~\ref{thm:main} 
for the unperturbed problem, i.e. for $B_0^k$ and $A_0$ respectively.
We do this by verifying the hypothesis in Lemma~\ref{lem:optransmain} and
Lemma~\ref{lem:opmain},
i.e. we prove that $T_{B^k_0}$ satisfies quadratic estimates and 
$\lambda- E_{B^k_0}N_{B^k_0}$ is an isomorphism,
and that $T_{A_0}$ satisfies quadratic estimates and 
$I\pm E_{A_0}N_{A_0}$ are isomorphisms respectively.

\subsection{Block coefficients}  \label{section7.1}

In this section we assume that $B= B_0\in L_\infty(\R^n;\mL(\wedge))$ 
have properties as in Definition~\ref{defn:B}
with the extra property that it is a block matrix, i.e. 
$$
  B=
     \begin{bmatrix}
       B_{\no\no} &0 \\ 0 & B_{\ta\ta}
     \end{bmatrix}
$$
in the splitting $\mH= N^-\mH \oplus N^+\mH$.
Note that $B$ being of this form is equivalent with the commutation
relations $N^\pm B=B N^\pm$.
\begin{lem}  \label{lem:blockasswapping}
  Let $B$ be a block matrix as above.
Then 
$$
  T_B= \Gamma + B^{-1}\Gamma^*B,
$$ 
where $\Gamma= N\m d= -i N\ud$ is a nilpotent first order, homogeneous
partial differential operator with constant coefficients.
\end{lem}
\begin{proof}
Since $N^\pm B=B N^\pm$, it follows that 
$M_B=N^+-B^{-1}N^-B= N^+-N^-=N$ and
$$
  T_B= N(\m d +B^{-1}\m d^*B)= \Gamma + B^{-1}\Gamma^*B,
$$
since $N^2=I$ and $NB^{-1}= B^{-1}N$. 
The operator $\Gamma$ is nilpotent since
$\Gamma^2= N\m dN\m d= N\m Nd\m d=-N\m N\m d^2=0$.
\end{proof}
\begin{rem}  \label{rem:gammaintertwining}
Note that if $\Pi_B= \Gamma+ \Gamma^*_B$, where $\Gamma^*_B=B^{-1}\Gamma^*B$
and $\Gamma$ is nilpotent, is an operator of the form considered
in \cite{AKMc}, then $\Pi_B$
intertwines $\Gamma$ and
$\Gamma^*_B$ in the sense that $\Pi_B\Gamma u=\Gamma^*_B\Pi_B u$
for all $u\in \dom(\Gamma^*_B\Pi_B)$ and $\Pi_B\Gamma^*_B
u=\Gamma\Pi_B u$ for all $u\in \dom(\Gamma\Pi_B)$. Thus
$\Pi^2_B$ commutes with both $\Gamma$ and $\Gamma^*_B$ on
appropriate domains. 
In particular, if $P_t^B= (1+t^2\Pi_B^2)^{-1}$ and 
$Q_t^B= t\Pi_B(1+t^2\Pi_B^2)^{-1}$, then
we find that $\Gamma P^B_tu= P^B_t\Gamma u$, 
$\Gamma^*_B Q^B_tu= Q^B_t\Gamma u$ for all $u \in\dom(\Gamma)$ 
and  
$\Gamma^*_B P^B_tu= P^B_t\Gamma^*_B u$, $\Gamma Q^B_tu= Q^B_t\Gamma^*_B u$
for all $u \in\dom(\Gamma^*_B)$.
\end{rem}

\begin{thm}  \label{thm:unpertblockmain}
  Let $B$ be a block matrix as above. Then
\begin{itemize}
\item[{\rm (i)}]
$T_B$ satisfies quadratic estimates, and
\item[{\rm (ii)}]
$E_B N_B+ N_B E_B=0$ and $N_B= N$.
In particular $\lambda-E_BN_B$ is an isomorphism whenever
$\lambda\notin \{i,-i\}$ with
$$
  (\lambda-E_BN_B)^{-1}= \frac 1{\lambda^2+1}(\lambda- N_BE_B).
$$
\end{itemize}
\end{thm}
\begin{proof}
For operators of the form $\Gamma+B^{-1}\Gamma^*B$, quadratic estimates
were proved in \cite{AKMc}, with essentially the same methods as we use here
in section~\ref{section8.1}.

To prove (ii), note that
since $B$ is a block matrix, it follows that $B^{-1} N^\pm\mH= N^\pm\mH$.
Thus the projections $N_B^\pm$ associated with the 
splitting $\mH= N^-\mH\oplus B^{-1}N^+\mH$ are $N^\pm$ and
the associated reflection operator is $N_B= N$.

To prove invertibility of $\lambda-E_BN_B$, we note that
$$
  T_BN_B= (\Gamma + B^{-1}\Gamma^*B)N= - N(\Gamma + B^{-1}\Gamma^*B)= -N_BT_B,
$$
since $N$ commutes with $d$, $d^*$ and $B$, and anticommutes with $\m$.
Thus
$$
  E_BN_B =\sgn(T_B)N_B= N_B\sgn(N_BT_BN_B)= N_B\sgn(-T_B)= -N_BE_B,
$$
since $\sgn(z)$ is odd.
Using this anticommutation formula we obtain
$$
  (\lambda-N_BE_B)(\lambda-E_BN_B)= \lambda^2+1-\lambda(E_BN_B+N_BE_B)= \lambda^2+1,
$$
and similarly 
$(\lambda-E_BN_B)(\lambda-N_BE_B)= \lambda^2+1$,
from which the stated formula for the inverse follows.
\end{proof}

\subsection{Constant coefficients}

We here collect results in the case when $B(x)=B\in\mL(\wedge)$ is
a constant accretive matrix.
In this case we make use of the Fourier transform
$$
  \mF u(\xi) = \hat u(\xi):= \frac 1{2\pi i}\int_{\R^n} 
  u(x) e^{-i (x, \xi)} dx,
$$
acting componentwise.
If we let 
\begin{align*}
  \mu_\xi f(\xi) & := \xi \wedg f(\xi), \\
  \mu^*_\xi f(\xi) & := \xi \lctr f(\xi),
\end{align*}
then $T_B$, conjugated with $\mF$, is the multiplication
operator
$$
  M_\xi f(\xi) := M_B^{-1}(i\m\mu_\xi -iB^{-1}\m\mu_\xi^*B)f(\xi),
  \qquad \xi\in\R^n. 
$$
\begin{lem}
For all $t\in\R$ and $\xi\in\R^n$ we have
$$
  |(it+M_\xi)^{-1}| \approx (t^2+|\xi|^2)^{-1/2}.
$$
\end{lem}
\begin{proof}
Let $u= (it+M_\xi) f$. 
It suffices to prove that $\|u\|^2\gtrsim (t^2+|\xi|^2)\|f\|^2$.
With $\wt B:= \m B \m$, it follows from the definition of $M_\xi$
that
$$
  \m M_B u= (\Gamma + \wt B^{-1} \Gamma^* B) f,
$$
where $\Gamma= i(t\mu+\mu_\xi)$.
Using Lemma~\ref{lem:anticom} we get
$$
 \|(\Gamma+\Gamma^*)g\|^2= \|\Gamma g\|^2+\|\Gamma^* g\|^2=
 ((\Gamma^*\Gamma+\Gamma^*\Gamma)g,g)= (t^2+|\xi|^2)\|g\|^2.
$$
Therefore our estimate follows from 
Lemma~\ref{lem:hodge}(ii).
\end{proof}
\begin{prop}  \label{prop:constantquadraticests}
If $B(x)=B\in\mL(\wedge)$ is a constant, accretive matrix,
then $T_B$ satisfies quadratic estimates.
\end{prop}
\begin{proof}
Using the lemma, we obtain the estimate
$$
  \left| \frac{tM_\xi}{1+t^2M_\xi^2} \right|
  \le t|M_\xi|\, |(i-M_{t\xi})^{-1}|\, |(i+M_{t\xi})^{-1}|
  \lesssim\frac{t|\xi|}{1+t^2|\xi|^2}.
$$
Thus using Plancherel's formula we obtain
\begin{multline*}
\int_0^\infty  \left\|\frac{tT_B}{1+t^2T_B^2}u\right\|^2 \frac{dt}t
\approx \int_0^\infty  \left\|\frac{tM_\xi}{1+t^2M_\xi^2} \hat u\right\|^2 
\frac{dt}t \\
\lesssim \int_{\R^n}\left( \int_0^\infty \left( \frac{t|\xi|}{1+t^2|\xi|^2} \right)^2 
\frac{dt}t \right) |\hat u(\xi)|^2 d\xi
\approx \|u\|^2,
\end{multline*}
where the last step follows from a change of variables $s=t|\xi|$.
\end{proof}
Next we prove that the Neumann and regularity problems are well posed
in the case of a complex, constant, accretive matrix
$$
  A = \begin{bmatrix}
       a_{00} & a_{0\ta} \\ a_{\ta 0} & a_{\ta\ta}
     \end{bmatrix} .
$$
For this it suffices to consider 
$B= I\oplus A \oplus I\oplus I\oplus \ldots\oplus I$
and the action of $T_A= T_B|_{\hat\mH^1}$ on the invariant subspace
$\hat\mH^1$.
Recall that $f\in\hat\mH^1$ means that $f$ is a vector field 
$f:\R^n\rightarrow \wedge^1= \C^{n+1}$ such that $d f_\ta=0$.
On the Fourier transform side $f\in\hat\mH^1$ 
is seen to correspond to a vector field 
$\hat f: \R^n\rightarrow\wedge^1$ such that
$\xi\wedg \hat f_\ta=0$, i.e.
$\hat f$ is such that its tangential part $f_\ta$ is a radial vector field.
Thus the space
$$
  \mF(\hat\mH^1) = \sett{\hat f\in L_2(\R^n;\C^{n+1})}
  {\xi\wedg \hat f_\ta=0}
$$
can be identified with $L_2(\R^n;\C^2)$ if we use 
$\{e_0, \xi/|\xi|\}$ as basis for the two dimensional
space $\hat\mH^1_\xi$ to which $\hat f(\xi)$ belongs. 
Furthermore, the operator $T_A$ 
is conjugated to the multiplication operator
$$
  M_\xi= M_A^{-1}(i\mu^*\mu_\xi -i A^{-1}\mu\mu^*_\xi A)
  :\mF(\hat\mH^1)\longrightarrow\mF(\hat\mH^1),
$$
under the Fourier transform, where $M_A:= N^+-A^{-1}N^-A$.
We see from the expression (\ref{eq:blockTA}) for $T_A$ that,
at a fixed point $\xi\in S^{n-1}$, the matrix for $M_\xi$ 
in the basis $\{e_0,\xi\}$ is
$$  
  \begin{bmatrix}
    \tfrac i{a_{00}} (a_{0\ta}+ a_{\ta 0}, \xi)
    & \tfrac i{a_{00}} (a_{\ta\ta}\xi, \xi) \\
    -i & 0
  \end{bmatrix}.
$$
\begin{thm}  \label{thm:unpertconstrellich}
  Let $A(x)=A\in\mL(\wedge)$ be a constant, complex, accretive matrix.
Then 
$$
  I\pm E_A N_A :\hat\mH^1 \longrightarrow \hat\mH^1
$$
are isomorphisms. 
In particular, the Neumann and regularity problems 
(Neu-$A$) and (Reg-$A$) are well posed.
\end{thm}
\begin{proof}
From Remark~\ref{rem:rotasprojs} we see that it suffices to 
prove that all four restricted projections
$$
  N^\pm_A : E_A^\pm\hat\mH^1 \longrightarrow N_A^\pm\hat\mH^1
$$
are isomorphisms,
or equivalently that the constant multiplication operators
$N^\pm_A$ are isomorphisms on $\mF(\hat\mH^1)$.
For the projections $\chi_\pm(M_\xi)$ conjugated to the
Hardy projection operators $E_A^\pm$,
we observe that $M_{t\xi} =tM_\xi$ and hence 
$\chi_\pm(tM_\xi)= \chi_\pm(M_\xi)$ for all $t>0$.
Thus it suffices to verify that
$$
 N^\pm_A : \chi_\pm(M_\xi) \hat\mH^1_\xi \longrightarrow N^\pm_A\hat\mH^1_\xi
$$
are isomorphisms for each $\xi\in S^{n-1}$.
For such fixed $\xi$, using the basis $\{e_0,\xi\}$ from above,
$N_A^-\hat\mH^1_\xi=\sett{\hat f\in\hat\mH^1_\xi}{e_0\wedg \hat f=0}$ is spanned
by $\begin{bmatrix} 1 & 0\end{bmatrix}^t$
and $N_A^+\hat\mH^1_\xi=\sett{\hat f\in\hat\mH^1_\xi}{( A \hat f, e_0)=0}$ is spanned
by $\begin{bmatrix} (a_{0\ta},\xi) & -a_{00} \end{bmatrix}^t$.
Indeed $(A(ze_0+w\xi),e_0)= za_{00}+w(a_{0\ta},\xi)$.

If we call $e_\xi^+$ and $e_\xi^-$ the two eigenvectors
of $M_\xi$, it follows that $\chi_\pm(M_\xi) \hat\mH^1_\xi$ are spanned by these, as these subspaces are one dimensional.
It suffices to show that $\begin{bmatrix} 1 & 0\end{bmatrix}^t$ and
$\begin{bmatrix} (a_{0\ta},\xi) & -a_{00} \end{bmatrix}^t$
are not eigenvectors of $M_\xi$.
We have
\begin{align*}
  M_\xi
  \begin{bmatrix} 1 \\ 0\end{bmatrix} 
  &=i \begin{bmatrix} \tfrac 1{a_{00}} (a_{0\ta}+ a_{\ta 0}, \xi) \\
        -1   \end{bmatrix}, \\
  M_\xi
  \begin{bmatrix} (a_{0\ta}, \xi) \\ -a_{00} \end{bmatrix} 
  &=i
  \begin{bmatrix}
   \tfrac 1{a_{00}} (a_{0\ta}+ a_{\ta 0}, \xi)(a_{0\ta}, \xi)
   -(a_{\ta\ta}\xi,\xi) \\
   -(a_{0\ta},\xi)
   \end{bmatrix}.
\end{align*}
Clearly the normal vector is not an eigenvector.
To prove that the second is not an eigenvector, note that the cross product
is
$$
  -(a_{0\ta}, \xi)^2 +(a_{0\ta}+ a_{\ta 0},\xi)(a_{0\ta}, \xi)
  -a_{00} (a_{\ta\ta}\xi,\xi) 
   = (a_{\ta 0}, \xi)(a_{0\ta}, \xi)- a_{00}(a_{\ta\ta}\xi,\xi).
$$
The right hand side is non-zero since 
$$
\left( 
\begin{bmatrix}
  a_{00} & a_{0\ta} \\
  a_{\ta 0} & a_{\ta\ta}
\end{bmatrix}
\begin{bmatrix}
   z \\ w\xi
\end{bmatrix}
,
\begin{bmatrix}
   z \\ w\xi
\end{bmatrix}
\right)
=
\left( 
\begin{bmatrix}
  a_{00} & (a_{0\ta},\xi) \\
  (a_{\ta 0},\xi) & (a_{\ta\ta}\xi,\xi)
\end{bmatrix}
\begin{bmatrix}
   z \\ w
\end{bmatrix}
,
\begin{bmatrix}
   z \\ w
\end{bmatrix}
\right)
$$
is a non-degenerate quadratic form as $A$ is accretive.
\end{proof}
\begin{rem}  \label{rem:Nforconst}
  We note that the method above also can be used to show that
$I\pm E_A N :\hat\mH^1 \rightarrow \hat\mH^1$, with the unperturbed 
operator $N$, are isomorphisms when $A$ is constant.
Here we also need to observe that the tangential vector 
$\begin{bmatrix} 0 & 1\end{bmatrix}^t$ is not an eigenvector to 
$M_\xi$.
\end{rem}
\subsection{Real symmetric coefficients}  \label{section7.3}

In this section, we assume that $B^*=B$.
We first prove a Rellich type estimate.
\begin{prop}  \label{prop:rellich}
Assume that $B^*=B$ and that $f\in E_B^+\V_B$ or $f\in E_B^-\V_B$.
Then
$$
  (Bf, f)= 2\re\big(e_0\lctr (Bf),e_0\lctr f\big)
         = 2\re\big(e_0\wedg (Bf),e_0\wedg f\big).
$$
In particular
$\|f\| \approx \| \hat N^-_B f \| \approx \| \hut N^-_B f \| \approx
\| \hut N^+_B f \| \approx \| \hat N^+_B f \|$.
\end{prop}
\begin{proof}
It suffices to consider $f\in E_B^+\V_B$ as the case $f\in E_B^-\V_B$ is treated
similarly.
We use Lemma~\ref{lem:algsplitting} and write $F_t:= e^{-t|T_B|}f$.
Hence $(0,\infty)\ni t\mapsto F_t\in \mH$ is differentiable,
$\lim_{t\rightarrow 0}F_t=f$ and $\lim_{t\rightarrow\infty}F_t=0$.
Furthermore $F_t\in \dom(d_x)$ and $B F_t\in\dom(d_x^*)$ for all 
$t\in (0,\infty)$.
We note the formulae 
\begin{align*}
  d^*_{t,x}\m G_t + \m d^*_{t,x} G_t &= -\pd_t G_t, \\ 
  B \m d_{t,x} F_t & = -\m d^*_{t,x}BF_t.
\end{align*}
The first identity, which we apply with $G_t= BF_t$, follows from 
Lemma~\ref{lem:anticom}, whereas the second is equivalent to 
$\pd_t F_t+ T_B F_t=0$, and follows from Lemma~\ref{lem:algsplitting}.
We get
\begin{multline*}
  (Bf,f) = -\int_0^\infty (\pd_t BF_t,F_t)+ (BF_t,\pd_t F_t) 
  = -2\re\int_0^\infty (\pd_t BF_t,F_t) \\
  = 2\re\int_0^\infty (d^*_{t,x} \m BF_t + 
        \m d^*_{t,x} BF_t, F_t) 
  = 2\re\int_0^\infty (d^*_{t,x}\m BF_t -
        B\m d_{t,x} F_t, F_t) \\
  = 2\re\int_0^\infty (d^*_{t,x}\m BF_t, F_t) - 
        (\m BF_t, d_{t,x} F_t) \\
  = 2\re\int_0^\infty ((d^*-\mu^*\pd_t)\m BF_t, F_t) - 
        (\m BF_t, (d+\mu\pd_t) F_t) \\
  = -2\re\int_0^\infty (\pd_t \m BF_t, \mu F_t) +
        (\m BF_t, \pd_t\mu F_t) \\
  = 2\re (\m B f, \mu f)= 2\re(e_0 \wedg (B f), e_0\wedg f).
\end{multline*}
Note that all integrals are convergent since 
$\|F_t\|\lesssim\min(1,t^{-s})$
and since $\pd_t F_t= -T_B F_t$ with 
$\|T_B F_t\|\lesssim\min(t^{s-1}, t^{-1})$ if 
$f\in\dom(|T_B|^s)\cap\ran(|T_B|^{-s})$.
This follows as in the proof of Lemma~\ref{lem:algsplitting}.
Furthermore, we note that
\begin{multline*}
  (e_0 \wedg B f, e_0\wedg f) =(B f, e_0\lctr(e_0\wedg f)) \\
  =(Bf, f- e_0\wedg(e_0\lctr f))= (Bf,f)-(e_0\lctr (Bf),e_0\lctr f).
\end{multline*}
Together with the calculation above this proves that 
$(Bf, f)= 2\re\big(e_0\lctr (Bf),e_0\lctr f\big)$.

To prove that $\|f\| \approx \| \hat N^+_B f \| \approx \| \hat N^-_B f \|$,
it suffices to show that $\|f\|\lesssim\| \hat N^\pm_B f \|$ since
$\hat N_B^\pm$ are bounded.
From the Rellich type identities above we have
$\|f\|^2\lesssim\|e_0\lctr (Bf)\|\|e_0\lctr f\|\lesssim \|\hat N_B^-f\|\|f\|$
which proves $\|f\|\lesssim\| \hat N^-_B f \|$, and using the other identity
we obtain
$\|f\|^2\lesssim\|e_0 \wedg (B f)\|\| e_0\wedg f\|\lesssim 
\|f\| \|\hat N_B^+f\|$.
The proof of $\|f\| \approx \| \hut N^+_B f \| \approx \| \hut N^-_B f \|$
is similar.
\end{proof}
We note that Proposition~\ref{prop:rellich} also proves that the 
norms of the two components of $f\in E_B^\pm\V_B$ in the unperturbed splitting
$\mH= N^-\mH\oplus N^+\mH$ are comparable.
\begin{cor}   \label{cor:rellich}
  Assume that $B^*=B$ and that $f\in E_B^+\V_B$ or $f\in E_B^-\V_B$,
  and decompose $B$ in the splitting $\mH= N^-\mH\oplus N^+\mH$ as
$$
  B=
     \begin{bmatrix}
       B_{\no\no} & B_{\no_\ta} \\ B_{\ta\no} & B_{\ta\ta}
     \end{bmatrix}.
$$
If $f_\ta= N^+f$ and $f_\no= N^-f$, then
$$
  (B_{\no\no} f_\no, f_\no)= (B_{\ta\ta} f_\ta, f_\ta).
$$
In particular $\|f\| \approx \|N^+f\| \approx \|N^-f\|$.
\end{cor}
\begin{proof}
We obtain from Proposition~\ref{prop:rellich} that
$\re(B_{\no\no} f_\no+ B_{\no\ta}f_\ta,f_\no)
= \re(B_{\ta\no} f_\no+ B_{\ta\ta}f_\ta,f_\ta)$.
The corollary now follows 
since $B_{\no\no}$, $B_{ \ta\ta}\ge \kappa_B$ and
$B_{\no\ta}^*= B_{\ta\no}$.
\end{proof}

Next we turn to quadratic estimates for the operator $T_B$.
As before we assume that $B$ is as in Definition~\ref{defn:B}
and that $B=B^*$.
For the rest of this section we shall also assume that 
\begin{equation}  \label{assumption:diagonalonforms}
  B_{\ta\no}^j = B_{\no\ta}^j= 0, \qquad\text{for } j\ge 2.
\end{equation}
Note that in particular this is true if
$B= I\oplus A\oplus I\oplus\ldots\oplus I$,
where $A\in L_\infty(\R^n;\mL(\wedge^1))$.

To prove quadratic estimates, we shall use 
Proposition~\ref{prop:Hardyreduction} where we verify the
hypothesis separately on the subspaces $\hat\mH^k_B$ and 
$\hut\mH_B$, which is possible by 
Lemma~\ref{lem:preservehats}.
\begin{lem}  \label{lem:realsymhut}
  Assume that $B^*=B$ and (\ref{assumption:diagonalonforms})
  holds.
Then
$$
  \|f\| \lesssim \tb{t\pd_t F_t}_\pm,\qquad \text{for all }
  f\in E_B^\pm\hut\V_B,
$$
where $F_t= e^{\mp t|T_B|} f$.
\end{lem}
\begin{proof}
As the two estimates are similar, we only show the estimate for
$f\in E_B^+\hut\V_B$.
Note that $F(t,x)$ satisfies (\ref{eq:sep2}), i.e.
$$
\begin{cases}
  d F = e_0\lctr \pd_t F, \\
  d^*(BF) = -e_0\wedg( B\pd_t F ).
\end{cases}
$$
Splitting both sides of both equations into normal and tangential parts,
with notation as in Corollary~\ref{cor:rellich}, we obtain
\begin{align}
  dF_\ta &= \m \pd_t F_\no, \label{huteq1} \\
  dF_\no &= 0, \label{huteq2} \\
  d^*(B_{\no\no}F_\no + B_{\no\ta} F_\ta) &= 
    -\m (B_{\ta\no} \pd_t F_\no +B_{\ta\ta}\pd_t F_\ta ), \label{huteq3} \\
  d^*(B_{\ta\no}F_\no + B_{\ta\ta} F_\ta) &= 0. \label{huteq4}
\end{align}
We have here used that $e_0\lctr f = \m f$ if $f$ is normal, and
$e_0\wedg f = \m f$ if $f$ is tangential.
A key observation is that the first term on the right hand side in
(\ref{huteq3}) vanishes since $B_{\ta\no} F_\no=0$. To see this, note that
$$
  B_{\ta\no} F_\no=
  B_{\ta\no}^1 F_\no^1+ B_{\ta\no}^2 F_\no^2+\ldots+
  B_{\ta\no}^{n+1} F_\no^{n+1}.
$$
By hypothesis (\ref{assumption:diagonalonforms}), 
$B_{\ta\no}^2=\ldots= B_{\ta\no}^{n+1}=0$.
Furthermore, writing $F_\no^1= F_0 e_0$ with the function $F_0$ being
scalar, we get from (\ref{huteq2}) that $0= d(F_0e_0)= (\nabla F_0)\wedg e_0$,
so $F_0$ is constant and therefore vanishes, and thus so does $F_\no^1$.
Equation (\ref{huteq3}) reduces to
\begin{equation}  \label{huteq5} 
  B_{\ta\ta}\pd_t F_\ta= -\m d^*(B_{\no\no}F_\no + B_{\no\ta} F_\ta).  
\end{equation}
We calculate
\begin{multline*}
  \|f\|^2 \lesssim (B_{\ta\ta} f_\ta,f_\ta) 
  = -2\re\int_0^\infty (B_{\ta\ta}\pd_t F_\ta, F_\ta) dt \\
  = 2\re\int_0^\infty \big(d^*(B_{\no\no}F_\no + B_{\no\ta} F_\ta) ,\m F_\ta \big) dt \\
  = -2\re\int_0^\infty \Big(
     \big(d^*(B_{\no\no}\pd_t F_\no + B_{\no\ta} \pd_t F_\ta ) ,\m F_\ta\big)
      +\big(d^*(B_{\no\no}F_\no + B_{\no\ta} F_\ta) ,\m \pd_t F_\ta \big) 
      \Big) t\,dt      \\
  = 2\re\int_0^\infty \Big(
     \big( B_{\no\no}\pd_t F_\no + B_{\no\ta} \pd_t F_\ta  , \pd_t F_\no \big)
      +\big( \m B_{\ta\ta}\pd_t F_\ta , \m \pd_t F_\ta \big) 
      \Big) t\,dt      \\
  = 2\re\int_0^\infty 
     ( B \pd_t F , \pd_t F) t\,dt \approx \tb{t\pd_t F}^2.
\end{multline*}
Here in the third step we use (\ref{huteq5}), in the fourth step 
we integrate by parts, in the fifth step we use duality and that
$d\m F_\ta= -\m dF_\ta= -\pd_t F_\no$ by (\ref{huteq1}) for the first term
and again (\ref{huteq5}) for the second term.
Finally in step $6$ we use that $\m$ is an isometry and that 
$B_{\ta\no} F_\no=0$.
\end{proof}
Next we turn to the subspace $\hat\mH^1$. 
\begin{lem}  \label{lem:realsymhat1}
  Let $A\in L_\infty(\R^n;\mL(\wedge^1))$ be real symmetric and
  $A\ge\kappa>0$, and let
  $B= I\oplus A\oplus I\oplus\ldots\oplus I$.
Then
$$
  \|f\| \lesssim \tb{t\pd_t F}_\pm,\qquad \text{for all }
  f\in E_B^\pm\hat\V_B^1,
$$
where $F= e^{\mp t|T_B|} f$.
\end{lem}
\begin{proof}
Recall that if $f\in E_B^\pm\hat \V_B^1$ and 
$F= F_0e_0+F_\ta= e^{\mp t|T_B|}f$, then 
Lemma~\ref{lem:LaplaceF_0^1} shows that $F_0$ satisfies
the equation 
$$
  \divv_{t,x} A(x) \nabla_{t,x} F_0(t,x)=0.
$$
Therefore, by the square function estimate of Dahlberg, Jerison and Kenig \cite{DJK}
and the estimates of harmonic measure of Jerison and Kenig~\cite{JK1},
we have estimates
$\|f_0\|\lesssim \tb{t\nabla_{t,x} F_0}_\pm$.
Hence applying the Rellich estimates in Proposition~\ref{prop:rellich},
we obtain
$$
  \|f\| \approx \|f_0\| \lesssim \tb{t\nabla_{t,x} F_0}_\pm
  \lesssim \tb{t\pd_t F}_\pm,
$$
since $\pd_i F_0= \pd_0 F_i$, as $d_{t,x} F=0$.
\end{proof}

We are now in position to prove the main result of this section.
\begin{thm}  \label{thm:mainforrealsymm}
  Let $A\in L_\infty(\R^n;\mL(\wedge^1))$ be real symmetric and
  $A\ge\kappa>0$, and let
  $B= I\oplus A\oplus I\oplus\ldots\oplus I$.
  Then $T_B$ has quadratic estimates, so that in particular 
  $E_B= \sgn(T_B):\mH\rightarrow\mH$ is bounded.
  Furthermore we have isomorphisms
$$
  I\pm E_BN_B: \mH \longrightarrow \mH.
$$
\end{thm}
\begin{proof}
To prove that $T_B$ has quadratic estimates it suffices
by Proposition~\ref{prop:Hardyreduction} to prove
$$
  \|f\| \lesssim \tb{t\pd_t F}_\pm,\qquad \text{for all }
  f\in E_B^\pm\V_B.
$$
To this end, we split 
$$
  f= f_0 + f_1+ \ldots+f_{n+1}+ \hut f
$$
where $f_j\in \hat\V_B^j$ and $\hut f\in\hut\V_B$.
Lemma~\ref{lem:preservehats} shows that $e^{\mp t|T_B|}$
preserves these subspaces so that similarly
$$
  F= F_0 + F_1+ \ldots+F_{n+1} + \hut F,
$$
where $(F_j)_t\in\hat\V_B^j$ and $\hut F_t\in\hut\V_B$.
It thus suffices to prove that 
$\|f_j\| \lesssim \tb{t\pd_t F_j}_\pm$, 
$j=0,1,\ldots, n+1$, and $\|\hut f\| \lesssim \tb{t\pd_t\hut F}_\pm$.
Lemma~\ref{lem:realsymhut} shows that 
$\|\hut f\| \lesssim \tb{t\pd_t\hut F}_\pm$
and Lemma~\ref{lem:realsymhat1} shows that 
$\|f_1\| \lesssim \tb{t\pd_t F_1}_\pm$.
Furthermore, for $j\ne 1$ we observe that 
$T_B= \hat T_B= T_I$ on the subspace $\hat\V_B^j$,
where $I$ denotes the identity matrix.
Thus $F_j= e^{\mp t|T_I|} f_j$ and it follows from 
Proposition~\ref{prop:constantquadraticests}
that $\|f_j\| \lesssim \tb{t\pd_t F_j}_\pm$.

Applying Proposition~\ref{prop:Hardyreduction} and Proposition~\ref{prop:QtoFCalc} now shows that $E_B= \sgn(T_B):\mH\rightarrow\mH$ is a bounded operator.
To show that $I\pm E_B\hat N_B: \mH \rightarrow \mH$
is invertible, note that the adjoint with respect to the duality $\dual\cdot\cdot_B$ from Definition~\ref{defn:duality} is 
$$
  I\mp \hut N_B E_B= \hut N_B(I\mp E_B\hut N_B)\hut N_B, 
$$
according to Proposition~\ref{prop:duality}.
Having established the boundedness of $E_B$, the Rellich estimates 
clearly extends to a priori estimates for all eight restricted projections
$$
  \hat N_B^\pm : E_B^\pm\mH \longrightarrow \hat N^\pm_B\mH, \qquad
  \hut N_B^\pm : E_B^\pm\mH \longrightarrow \hut N^\pm_B\mH.
$$
As in Remark~\ref{rem:rotasprojs} this translates to a priori 
estimates for $I\pm E_B\hat N_B$ and their adjoints, which
proves that they are isomorphisms as in Remark~\ref{rem:aprioriduality}.
\end{proof}
\begin{rem}  \label{rem:Nforrealsym}
Using instead Corollary~\ref{cor:rellich}, we also prove as in
Theorem~\ref{thm:mainforrealsymm} that 
$I\pm E_BN: \mH \longrightarrow \mH$, where $N$ is the unperturbed operator, are isomorphisms.
\end{rem}

%
%
%
\section{Quadratic estimates for perturbed operators} \label{section8}

In Section~\ref{section7} we proved that $T_{B_0}$ satisfies quadratic
estimates
\begin{equation}  \label{eq:unpertubedquadratic}
  \tb{ Q_t^{B_0} f} \approx \|f\|,\qquad f\in \mH,
\end{equation}
for certain unperturbed coefficients $B_0$:
\begin{itemize}
\item[{\rm (b)}]
Block coefficients 
$$
  B_0=
     \begin{bmatrix}
       (B_0)_{\no\no} &0 \\ 0 & (B_0)_{\ta\ta}
     \end{bmatrix}.
$$
\item[{\rm (c)}]
Constant coefficients $B_0(x)= B_0$, $x\in\R^n$,
of the form
$B_0= I\oplus A_0\oplus I \oplus\ldots\oplus I$,
i.e. $B_0$ only acts nontrivially on the vector part.
\item[{\rm (s)}]
Real symmetric coefficients of the form
$B_0= I\oplus A_0\oplus I \oplus\ldots\oplus I$.
\end{itemize}
Note that for (c), we did prove quadratic estimates for 
general constant coefficients $B_0$ in 
Proposition~\ref{prop:constantquadraticests}.
However, since we only prove invertibility $I\pm E_{B_0}N_{B_0}$
on the subspace $\hat\mH^1$ in Theorem~\ref{thm:unpertconstrellich},
we shall only prove perturbation results for constant coefficients
of the form $B_0= I\oplus A_0\oplus I \oplus\ldots\oplus I$.

In this section we let $B_0\in L_\infty(\R^n;\mL(\wedge))$
be a fixed accretive coefficient matrix with properties (b), (c) or (s).
Constants $C$ in estimates or implicit in notation $\lesssim$ and $\approx$ 
in this section will be allowed to depend only
on $\|B_0\|_\infty$, $\kappa_{B_0}$ and dimension $n$.
Note that this is indeed the case for the constants implicit in
(\ref{eq:unpertubedquadratic}).

We now consider a small perturbation $B\in L_\infty(\R^n;\mL(\wedge))$ of $B_0$:
$$
  \|B-B_0\|_{L_\infty(\R^n;\mL(\wedge))} \le\epsilon_0.
$$
Throughout this section, we assume in particular that $\epsilon_0$
is chosen small enough so
that $B$ has properties as in Definition~\ref{defn:B} with
$\|B\|_\infty \le 2\|B_0\|_\infty$ and $\kappa_B\ge \tfrac 12 \kappa_{B_0}$.
The goal is to show that
\begin{equation}  \label{eq:pertestgoal}
  \tb{Q_t^B f}\le C\|f\|, \qquad\text{whenever } \|B-B_0\|_\infty\le\epsilon,
  \,\, f\in\mH,
\end{equation}
where $C= C(\|B_0\|_\infty, \kappa_{B_0}, n)$ and 
$\epsilon= \epsilon(\|B_0\|_\infty, \kappa_{B_0}, n) \le\epsilon_0$.
Since the properties of $B_0$ are stable under taking adjoints, 
we see that the quadratic estimate (\ref{eq:quadrests}) follows from (\ref{eq:pertestgoal}) and
Lemma~\ref{lem:dualityforquadratic}.
In order to prove (\ref{eq:pertestgoal}) we make use of the following
identity, where we note that each of the the six terms consists of three factors,
the first being one of the operators from Corollary~\ref{cor:opfamilies}, the second being a multiplication operator 
$\ep$ with norm $\|\ep\|_\infty \lesssim \epsilon_0$ and
the last factor being one of the operators from Corollary~\ref{cor:opfamilies}
but with $B$ replaced by $B_0$.
\begin{align}
  Q_t^B - Q_t^{B_0} 
  &= P_t^B \Big( (M_B^{-1}-M_{B_0}^{-1})M_{B_0} \Big) Q_t^{B_0} 
  \label{eq:good1}  \\
  &-iP_t^B \Big( M_B^{-1}(B^{-1}-B_0^{-1})B_0  \Big)\Big( t \ud^*_{B_0} P_t^{B_0} \Big) 
  \label{eq:good2}  \\
  &-i\Big(P_t^B M_B^{-1}t\ud^*_B\Big) \Big(B^{-1}(B-B_0)\Big) P_t^{B_0} 
  \label{eq:bad3}   \\
  &- Q_t^B \Big( (M_B^{-1}-M_{B_0}^{-1})M_{B_0} \Big) \Big(tT_{B_0}Q_t^{B_0}\Big) 
  \label{eq:bad4}  \\
  &+iQ_t^B \Big( M_B^{-1}(B^{-1}-B_0^{-1})B_0  \Big)\Big( t \ud^*_{B_0} Q_t^{B_0} \Big) 
  \label{eq:bad5}  \\
  &+i\Big(Q_t^B M_B^{-1}t\ud^*_B\Big) \Big( B^{-1}(B-B_0) \Big) Q_t^{B_0} 
  \label{eq:good6}
\end{align}
Recall that
$$
  T_B= -iM_B^{-1}(\ud+B^{-1}\ud^*B), \qquad\text{where } \ud= i\m d,\quad
  M_B=N^+-B^{-1}N^-B,
$$
$P^B_t= (1+t^2 T_B^2)^{-1}$ and $Q^B_t= tT_B(1+t^2 T_B^2)^{-1}$ and
similarly for $B_0$.
The identity (\ref{eq:good1}-\ref{eq:good6}) is established by using
$$
  Y(1+Y^2)^{-1}- X(1+X^2)^{-1} = (1+Y^2)^{-1}( (Y-X)-Y(Y-X)X )(1+X^2)^{-1}
$$
with $Y= tT_B$ and $X= tT_{B_0}$, and then inserting
\begin{multline*}
  Y-X= T_B-T_{B_0}= (M_B^{-1}-M_{B_0}^{-1})M_{B_0}tT_{B_0} \\
  -i (M_B^{-1}(B^{-1}-B_0^{-1})B_0) (t\ud^*_{B_0})
    -i (M_B^{-1}t\ud^*_B)( B^{-1}(B-B_0)).
\end{multline*}
\begin{lem}   \label{lem:pertests}
Assume that for all $\ep\in L_\infty(\R^n;\mL(\wedge))$ and all
operator families $(\wt Q_t)_{t>0}$ with $L_2$ off-diagonal bounds $\|\wt Q_t\|_\off \le C$ as in Definition~\ref{defn:offdiag},
we have the two estimates
\begin{align}
  \tb{ \wt Q_t\ep (t\ud Q_t^{B_0}) }_\op &\le C( \tb{\wt Q_t}_\op+ 1 )\|\ep\|_\infty, \label{eq:perttildeest1}  \\
  \tb{ \wt Q_t\ep (t\ud^*_{B_0} Q_t^{B_0}) }_\op &\le C( \tb{\wt Q_t}_\op+ 1 )\|\ep\|_\infty.   \label{eq:perttildeest2}
\end{align}
Then there exists $\epsilon>0$ such that
(\ref{eq:pertestgoal}) holds.
\end{lem}
\begin{proof}
First consider the terms (\ref{eq:good1}), (\ref{eq:good2}) and
(\ref{eq:good6}).
Using the uniform boundedness of $P_t^B$ and 
$Q_t^B M_B^{-1}t\ud^*_B$ from Corollary~\ref{cor:opfamilies},
we deduce that
\begin{align*}
  \tb{(\ref{eq:good1})}_\op &\lesssim \|B-B_0\|_\infty\tb{Q_t^{B_0}}_\op, \\
  \tb{(\ref{eq:good2})}_\op &\lesssim \|B-B_0\|_\infty\tb{ t\ud^*_{B_0} P_t^{B_0}}_\op, \\
  \tb{(\ref{eq:good6})}_\op &\lesssim \|B-B_0\|_\infty\tb{Q_t^{B_0}}_\op.
\end{align*}
Observe that $t\ud^*_{B_0} P_t^{B_0}  = i\PP_{B_0}^2 M_{B_0} Q_t^{B_0}$ as in Corollary~\ref{cor:opfamilies}, where the Hodge projection $\PP_{B_0}^2$
and $M_{B_0}$ are bounded. 
Thus we see from (\ref{eq:unpertubedquadratic}) that
$$
  \tb{(\ref{eq:good1})}_\op+\tb{(\ref{eq:good2})}_\op+
  \tb{(\ref{eq:good6})}_\op \lesssim \|B-B_0\|_\infty.
$$
To handle the terms (\ref{eq:bad3}-\ref{eq:bad5})
we introduce the truncated operator families
$\wt Q_t^1:= \chi(t) Q_t^B$ and 
$\wt Q_t^2:= \chi(t) P_t^B M_B^{-1}t\ud^*_B$,
where $\chi(t)$ denotes the characteristic function of the interval
$[\tau^{-1}, \tau]$ for some large $\tau$.
Note that 
$P_t^B M_B^{-1}t\ud^*_B = i Q_t^B \PP_B^1$ as in
Corollary~\ref{cor:opfamilies}, and therefore
$\tb{\wt Q_t^2}_\op=\tb{\wt Q_t^1 \PP_B^1 }_\op\lesssim\tb{\wt Q_t^1 }_\op$.
To use the hypothesis on the terms (\ref{eq:bad3}-\ref{eq:bad5}) we
note that the last factors are
\begin{align*}
  P_t^{B_0} &= I +i M_{B_0}^{-1}(t\ud Q_t^{B_0})+ iM_{B_0}^{-1}
  (t\ud^*_{B_0} Q_t^{B_0}), \\
  t T_{B_0}Q_t^{B_0} &= -i M_{B_0}^{-1}(t\ud Q_t^{B_0})- iM_{B_0}^{-1}
  (t\ud^*_{B_0} Q_t^{B_0}), \\
  -it\ud^*_{B_0} Q_t^{B_0}&=-i (t\ud^*_{B_0} Q_t^{B_0})
\end{align*}
respectively. 
Thus we get from (\ref{eq:good1}-\ref{eq:good6}), after multiplication
with $\chi(t)$, that
$$
  \tb{\chi(t)(Q_t^B-Q_t^{B_0})}_\op \le C(\tb{\chi(t) Q_t^B}_\op+1)\|B-B_0\|_\infty,
$$
and thus if $\|B-B_0\|_\infty\le \epsilon:= 1/(2C)$ that
$$
  \tb{\chi Q_t^B}_\op \le 
  \frac{C\|B-B_0\|_\infty+ \tb{Q_t^{B_0}}_\op}{1-C\|B-B_0\|_\infty}\le C',
$$
since $\tb{\chi(t) Q_t^B}_\op < \infty$.
Since this estimate is independent of $\tau$, it follows that 
$\tb{Q_t^Bf}\lesssim \|f\|$.
\end{proof}
Before turning to the proofs of (\ref{eq:perttildeest1})
and (\ref{eq:perttildeest2}), we summarise fundamental techniques
from harmonic analysis that we shall need.
We use the following dyadic decomposition of $\R^n$. Let
$\dyadic= \bigcup_{j=-\infty}^\infty\dyadic_{2^j}$ where
$\dyadic_t:=\{ 2^j(k+(0,1]^n) :k\in\Z^n \}$ if $2^{j-1}<t\le
2^j$. For a dyadic cube $Q\in\dyadic_{2^j}$, denote by $l(Q)=2^j$
its \emph{sidelength}, by $|Q|= 2^{nj}$ its {\em volume} and 
by $R_Q:=Q\times(0,2^j]\subset \R^{n+1}_+$ the associated
\emph{Carleson box}. Let the {\em dyadic averaging operator} $A_t:
\mH \rightarrow \mH $ be given by
$$
  A_t u(x) := u_Q:= \barint_{\hspace{-6pt}Q} u(y)\,  dy = \frac{1}{|Q|} \int_Q
  u(y)\,  dy
$$
for every $x \in \R^n$ and $t>0$, where  $Q \in \dyadic_t$ is
the unique dyadic cube that contains $x$.

We now survey known results for a family of operators
$\Theta_t:\mH\rightarrow\mH$, $t>0$.
For the proofs we refer to \cite{AHLMcT} and \cite{AKMc}.

\begin{defn}  \label{defn:princpart}
By the {\em principal part} of $(\Theta_t)_{t>0}$
we mean the multiplication operators $\gamma_t$ defined by
$$
   \gamma_t(x)w:= (\Theta_t w)(x)
$$
for every $w\in \wedge$. We view $w$ on the right-hand side
of the above equation as the constant function defined on $\R^n$ 
by $w(x):=w$. 
We identify $\gamma_t(x)$ with the (possibly unbounded) multiplication
operator $\gamma_t: f(x)\mapsto \gamma_t(x)f(x)$.
\end{defn}
\begin{lem}
Assume that $\Theta_t$ has $L_2$ off-diagonal bounds 
with exponent $M>n$.
Then $\Theta_t$ extends to a bounded operator 
$L_\infty\rightarrow L_2^{\text{loc}}$.
In particular we have well defined functions 
$\gamma_t\in L_2^{\text{loc}}(\R^n; \mL(\wedge))$ with bounds
$$
   \barint_{\hspace{-6pt}Q} |\gamma_t(y)|^2 \, d y \lesssim  
   \|\Theta_t\|_\off^2
$$
for all $Q\in\dyadic_t$.  
Moreover $\|\gamma_t A_t\|\lesssim \|\Theta_t\|_\off$ uniformly for all $t>0$.
\end{lem}
We have the following principal part approximation $\Theta_t\approx\gamma_t$.
\begin{lem}  \label{lem:ppa}
Assume that $\Theta_t$ has $L_2$ off-diagonal bounds 
with exponent $M>3n$ and let $F_t:\R^n\rightarrow\wedge$ be a 
family of functions. Then
$$
  \tb{(\Theta_t-\gamma_t A_t)F_t}
  \lesssim \|\Theta_t\|_\off \tb{t\nabla F_t},
$$
where $\nabla f=\nabla\otimes f= \sum_{j=1}^n e_i\otimes (\partial_j f)$ denotes the full differential of $f$.
Moreover, if $P_t$ is a standard Fourier mollifier
(we shall use $P_t= (1+ t^2\Pi^2)^{-1}$ where $\Pi= \Gamma+ \Gamma^*$
and $\Gamma$ is an exact nilpotent, homogeneous first order partial
differential operator with constant coefficients as in 
\cite{AKMc})
and $F_t= P_t f$ for some $f\in\mH$, then
$\tb{\gamma_t A_t(P_t-I)f}\lesssim \|\Theta_t\|_\off \|f\|$
and $\tb{t\nabla P_t}_\op\le C$.
Thus
$$
  \tb{(\Theta_t P_t -\gamma_t A_t)f} \lesssim \|\Theta_t\|_\off \|f\|.
$$
\end{lem}

\begin{defn}
A function $\gamma(t,x):\R^{n+1}_+\rightarrow\wedge$ is called
a {\em Carleson function} if there exists $C<\infty$
such that
$$
  \iint_{R_Q} |\gamma(t,x)|^2 \frac{dxdt}t \le C^2 |Q|
$$
for all cubes $Q\subset\R^n$.
Here $R_Q:= Q\times (0,l(Q)]$ is the Carleson box over $Q$.
We define the Carleson norm $\|\gamma_t\|_C$ to be the smallest
constant $C$.
\end{defn}
We use Carleson's lemma in the following form.
\begin{lem}  \label{lem:Carleson}
Let $\gamma_t(x)=\gamma(t,x):\R^{n+1}_+\rightarrow\wedge$ be a 
Carleson function and let
$F_t(x)= F(t,x):\R^{n+1}_+\rightarrow\wedge$ be a family of functions.
Then
$$
  \tb{\gamma_t F_t} \lesssim \|\gamma_t\|_C \| N_*(F_t)\|,
$$
where $N_*(F_t)(x):= \sup_{|y-x|<t}|F_t(y)|$ denotes the non-tangential maximal function of $F_t$.
In particular, if $F_t= A_t f$ for some $f\in \mH$, then
$\|N_*(A_t f)\|\lesssim \|M(f)\|\lesssim \|f\|$, where 
$M(f)(x):= \sup_{r>0} \barint_{\hspace{-3pt} B(x,r)} |f(y)| dy$ 
denotes the Hardy--Littlewood maximal function, and thus
$$
  \tb{\gamma_t A_t f} \lesssim \|\gamma_t\|_C \|f\|.
$$
\end{lem}
\begin{lem}  \label{lem:sqfcntoCarleson}
Assume that $\Theta_t$ has $L_2$ off-diagonal bounds 
with exponent $M>n$.
Then 
$$
  \|\Theta_t f\|_C \lesssim (\tb{\Theta_t}_\op + \|\Theta_t\|_\off)\|f\|_\infty,
$$
for every $f\in L_\infty(\R^n;\wedge)$.
In particular, choosing $f=w=\text{constant}$ we obtain
$$
  \|\gamma_t\|_C \lesssim \tb{\Theta_t}_\op + \|\Theta_t\|_\off.
$$
\end{lem}

\subsection{Perturbation of block coefficients}  \label{section8.1}

In this section we assume that $B_0$ is a block matrix, i.e.
we assume that
$$
  B_0=
     \begin{bmatrix}
       (B_0)_{\no\no} &0 \\ 0 & (B_0)_{\ta\ta}
     \end{bmatrix},
$$
in the splitting $\mH= N^-\mH\oplus N^+\mH$.
Our goal is to prove Theorem~\ref{thm:transmain} by 
verifying the hypothesis of Lemma~\ref{lem:pertests}.
We recall from Lemma~\ref{lem:blockasswapping} that 
$$
  T_{B_0}= \Pi_{B_0}= \Gamma+ B_0^{-1}\Gamma^* B_0, 
  \qquad \text{where } \Gamma:= -i N \ud, 
$$
is an operator of the form treated in \cite{AKMc}.
Thus with a slight change of notation for $\ep$, we need to prove the
following.
\begin{thm}  \label{thm:pertestsforblock}
If $B_0$ is a block matrix and if $(\wt Q_t)_{t>0}$
is an operator family with $L_2$ off-diagonal bounds 
$\|\wt Q_t\|_\off \le C$ as in Definition~\ref{defn:offdiag},
then with notation as above
\begin{align}
  \tb{ \wt Q_t\ep (t^2\Gamma\Gamma^*_{B_0} P_t^{B_0}) }_\op &\le C( \tb{\wt Q_t}_\op+ 1)\|\ep\|_\infty, \label{eq:perttildeest1block}  \\
  \tb{ \wt Q_t\ep (t^2\Gamma^*_{B_0}\Gamma P_t^{B_0}) }_\op &\le C( \tb{\wt Q_t}_\op+ 1 )\|\ep\|_\infty,   \label{eq:perttildeest2block}
\end{align}
where $P_t^{B_0}= (1+t^2\Pi_{B_0}^2)^{-1}$.
\end{thm}
We note that if (\ref{eq:perttildeest1block}) holds with $\Gamma$ 
replaced by $\Gamma^*$ and
with $B_0$ replaced by $B_0^{-1}$, then also (\ref{eq:perttildeest2block}) holds.
This follows from the conjugation formula
\begin{multline*}
  \wt Q_t\ep (t^2\Gamma^*_{B_0}\Gamma P_t^{B_0}) \\
  = 
  B_0^{-1} ( B_0\wt Q_tB_0^{-1} ) ( B_0\ep B_0^{-1} )
  ( t^2\Gamma^*\Gamma_{B_0^{-1}} 
      (1+t^2(\Gamma^*+ B_0\Gamma B_0^{-1} )^2)^{-1} ) B_0.
\end{multline*}
Thus it suffices to prove (\ref{eq:perttildeest1block}), as long as we
only use properties of $(\Gamma, B_0)$ shared with $(\Gamma^*, B_0^{-1})$.
To this end, we let $\Theta_t$ be the operator
$$
  \Theta_t:= \wt Q_t\ep (t^2\Gamma\Gamma^*_{B_0} P_t^{B_0}),
$$
and denote by $\gamma_t(x)$ its principal part as in 
Definition~\ref{defn:princpart}.
We note that we have a Hodge type splitting
$\mH = \nul(\Gamma)\oplus \nul(\Gamma^*_{B_0})$ by 
Lemma~\ref{lem:hodge}(i),
and since $\Theta_t|_{\nul(\Gamma^*_{B_0})}=0$
it suffices to bound $\tb{\Theta_t f}$ for $f\in \nul(\Gamma)$. 
We do this by writing
\begin{equation}  \label{eq:blockmaindecomposition}
  \Theta_t f= \Theta_t(I-P_t)f+ (\Theta_t P_t-\gamma_t A_t )f+ 
  \gamma_t A_t f,
\end{equation}
where $\Pi:= \Gamma +\Gamma^*$ is
the corresponding unperturbed operator and 
$P_t:= (1+t^2 \Pi^2)^{-1}$ and $Q_t:= t\Pi(1+t^2 \Pi^2)^{-1}$.
\begin{lem}  \label{lem:firstterminblockest}
We have, for all $f\in\nul(\Gamma)$, the estimate
$$
  \tb{\Theta_t(I-P_t)f} \le 
    C( \tb{\wt Q_t}_\op+1 )\|\ep\|_\infty\|f\|.
$$
\end{lem}
\begin{proof}
If $f\in\nul(\Gamma)$, then
$(I-P_t)f= t\Gamma Q_tf \in\nul(\Gamma)$, which shows that
$t^2\Gamma\Gamma^*_{B_0}P_t^{B_0}(I-P_t)f= (I-P_t^{B_0})(I-P_t)f$.
To prove the estimate, we write
$$
  \Theta_t(I-P_t)f = \wt Q_t\ep (I- P_t^{B_0}) t\Gamma Q_t f 
  = \wt Q_t\ep f- \wt Q_t\ep P_t f- 
    \wt Q_t\ep Q_t^{B_0} ( \Pi_{B_0}^{-1}\Gamma ) Q_t f,
$$
where we recall that $Q_t^{B_0}= t \Pi_{B_0} P_t^{B_0}$.
Clearly $\tb{\wt Q_t\ep f} \lesssim \tb{\wt Q_t} \|\ep\|_\infty \|f\|$.
For the second term, we write $\wt\Theta_t:= \wt Q_t \ep$ with 
principal part $\wt\gamma_t$ as in Definition~\ref{defn:princpart}
and estimate
\begin{multline*}
  \tb{\wt \Theta_t P_t f} \le 
  \tb{(\wt \Theta_t  P_t - \wt\gamma_t A_t) f }+ \tb{ \wt\gamma_t A_t f } \\
  \lesssim \|\ep\|_\infty \|f\| + \|\wt\gamma_t \|_C \|f\| 
  \lesssim \|\ep\|_\infty\|f\| +  (\tb{\wt Q_t}_\op +1 ) \|\ep\|_\infty\|f\|.
\end{multline*}
In the second step, we used for the first term that 
$\|\wt\Theta_t\|_\off \lesssim \|\ep\|_\infty$ in Lemma~\ref{lem:ppa}
and for the second term we used Lemma~\ref{lem:Carleson}.
In the last step, we used Lemma~\ref{lem:sqfcntoCarleson} 
on the last term and that
$\tb{\wt\Theta_t}_\op\lesssim \tb{\wt Q_t}_\op\|\ep\|_\infty$.
Finally we note that
$$
  \tb{\wt Q_t\ep Q_t^{B_0} ( \Pi_{B_0}^{-1}\Gamma ) Q_t f}
  \lesssim \|\ep\|_\infty\tb{Q_t f}
  \lesssim \|\ep\|_\infty\|f\|,
$$
since the Hodge projection $\Pi_{B_0}^{-1}\Gamma$ is 
bounded by Lemma~\ref{lem:hodge}(i).
\end{proof}
To estimate the last term in (\ref{eq:blockmaindecomposition})
we shall apply a local $T(b)$ theorem as in \cite{AHLMcT,AKMc}.
We here give an alternative construction of test functions to those
used in \cite{AKMc}, more in the spirit of the original proof of the
Kato square root problem \cite{AHLMcT}.
\begin{lem}  \label{lem:newtestfcns}
  Let $\Gamma$ be a nilpotent operator in $\mH$, which
  is a homogeneous, first order partial differential operator
  with constant coefficients.
  Denote by $\wedge_\Gamma\subset \wedge$ the image of the linear functions
  $u:\R^n\rightarrow \wedge$ under $\Gamma$, where we identify $\wedge$ with the
  constant functions $\R^n\rightarrow \wedge$.
  
  Then for each $w\in \wedge_\Gamma$ with $|w|=1$, each cube $Q\subset\R^n$ and each $\epsilon>0$, 
  there exists a test function
  $f^w_{Q,\epsilon}\in\mH$ such that
  $f^w_{Q,\epsilon}\in\ran(\Gamma)$,
  $\|f^w _{Q,\epsilon} \| \lesssim |Q|^{1/2}$,
$$  
  \|\Gamma^*_{B_0} f^w _{Q,\epsilon}\| \lesssim \tfrac 1{\epsilon l(Q)} |Q|^{1/2} 
  \qquad\text{and}\qquad 
  \left| \barint_{\hspace{-6pt}Q} f^w_{Q,\epsilon} - w \right| \lesssim
     \epsilon^{1/2}.
$$
\end{lem}
\begin{proof}
Let $u(x)$ be a linear function such that 
$w= \Gamma u$ and $\sup_{3Q}|u(x)|\lesssim l(Q)$, and
define $w_Q := \Gamma(\eta_Q u)$, where $\eta_Q$
is a smooth cutoff such that $\eta_Q|_{2Q}=1$, $\supp(\eta_Q)\subset 3Q$
and $\|\nabla \eta_Q\|_\infty\lesssim 1/l(Q)$.
It follows that 
$$
  w_Q\in\ran(\Gamma),\quad
  w_Q|_{2Q}= w,\quad
  \supp w_Q\subset 3Q\quad\text{and}\quad 
  \|w_Q\|_\infty\le C.
$$
Next we define the test function
$f^w_{Q,\epsilon}:= P_{\epsilon l}^{B_0}w_Q$,
where we write $l=l(Q)$.
Using Corollary~\ref{cor:opfamilies}, it follows that
$\|f^w _{Q,\epsilon} \| \lesssim |Q|^{1/2}$ and
$\|\Gamma^*_{B_0} f^w _{Q,\epsilon}\|\lesssim \tfrac 1{\epsilon l(Q)} |Q|^{1/2}$
and since $\Gamma$ commutes with $P_{\epsilon l}^{B_0}$, it follows that
$f^w_{Q,\epsilon}\in\ran(\Gamma)$.
To verify the accretivity property, we make use of \cite[Lemma 5.6]{AKMc}
which shows that
\begin{multline*}
\left| \barint_{\hspace{-6pt}Q} f^w_{Q,\epsilon} - w \right| =
\left| \barint_{\hspace{-6pt}Q} (I-P_{\epsilon l}^{B_0})w_Q \right| 
=\left| \barint_{\hspace{-6pt}Q} \epsilon l\Gamma (Q_{\epsilon l}^{B_0}w_Q) \right| \\ 
\lesssim\epsilon^{1/2}\left( \barint_{\hspace{-6pt}Q} |Q_{\epsilon l}^{B_0}w_Q|^2 \right)^{1/4} 
\left( \barint_{\hspace{-6pt}Q} |\epsilon l\Gamma Q_{\epsilon l}^{B_0}w_Q|^2 \right)^{1/4} \lesssim \epsilon^{1/2},
\end{multline*}
where we used that $w_Q\in\ran(\Gamma)$ in the second step.
\end{proof}

\begin{proof}[Proof of Theorem~\ref{thm:pertestsforblock}]
We have seen that it suffices to prove (\ref{eq:perttildeest1block}),
and to bound each term in (\ref{eq:blockmaindecomposition}) for
$f\in \nul(\Gamma)$.
The first term is estimated by Lemma~\ref{lem:firstterminblockest}
and the second by Lemma~\ref{lem:ppa}.
To prove that the last term has estimate
$\tb{\gamma_t A_t f} \lesssim (\tb{\wt Q_t}_\op +1)\|\ep\|_\infty \|f\|$,
it suffices by Lemma~\ref{lem:Carleson} to prove that
$$
   \|\gamma_t \|_C \lesssim (\tb{\wt Q_t}_\op +1)\|\ep\|_\infty.
$$
To this end, we apply the local $T(b)$ argument and stopping time argument
in \cite[Section 5.3]{AKMc}.
Note that $\ran(\Gamma)\subset L_2(\R^n;\wedge_\Gamma)$ and
thus, since $\Gamma$ is an exact nilpotent operator,
it follows that $\nul(\Gamma)\subset L_2(\R^n;\wedge_\Gamma)$.
Furthermore $A_t f\in L_2(\R^n;\wedge_\Gamma)$ if 
$f\in L_2(\R^n;\wedge_\Gamma)$.
Thus it suffices to bound the Carleson norm of $\gamma_t(x)$ 
seen as a linear operator
$\gamma_t(x): \wedge_\Gamma \rightarrow \wedge$.
The conical decomposition $\bigcup_{\nu\in \V} K_\nu$ of the space of matrices
performed in \cite[Section 5.3]{AKMc}, here decomposes
the space $\mL(\wedge_\Gamma;\wedge)$ and for the fixed unit matrix 
$\nu\in\mL(\wedge_\Gamma;\wedge)$ we choose $w\in \wedge_\Gamma$ and 
$\hat w\in \wedge$ such that $|\hat w|= |w|=1$ and $\nu^*(\hat w)= w$.
With the stopping time argument in \cite[Section 5.3]{AKMc}, using the
new test functions $f^w_Q$ from Lemma~\ref{lem:newtestfcns}, we obtain
$$
  \|\gamma_t \|_C^2 \lesssim 
  \sup_{Q, w\in\wedge_\Gamma, |w|=1}\frac 1{|Q|} \iint_{R_Q} |\gamma_t(x) A_t f^w_Q(x)|^2\frac{dxdt}t,
$$
where $f^w_Q= f^w_{Q,\epsilon}$ for a small enough but fixed $\epsilon$.
To estimate the right hand side we use 
(\ref{eq:blockmaindecomposition}) with $f$ replaced with the
test function $f^w_Q$, estimate the first two terms with
Lemma~\ref{lem:firstterminblockest} (which works since 
$f^w_Q\in \ran(\Gamma)$) and Lemma~\ref{lem:ppa},
and obtain
\begin{multline*}
  \iint_{R_Q} |\gamma_t A_t f^w_Q|^2\frac{dxdt}t \\
  \lesssim 
  \iint_{R_Q} |\gamma_t A_t f^w_Q- \Theta_t f^w_Q|^2\frac{dxdt}t +
  \iint_{R_Q} |\Theta_t f^w_Q|^2\frac{dxdt}t \\
  \lesssim (\tb{\wt Q_t}_\op +1)^2\|\ep\|^2_\infty |Q| +
  \iint_{R_Q} |\Theta_t f^w_Q|^2\frac{dxdt}t.
\end{multline*}
Using that 
$\|\Gamma^*_{B_0} f^w _{Q}\| \lesssim \tfrac 1{ l(Q)} |Q|^{1/2}$
we then get
$$
  \|\Theta_t f^w _{Q}\|=
  \|\wt Q_t\ep (t\Gamma P_t^{B_0}) t
  (\Gamma^*_{B_0} f^w_{Q} )\| 
  \lesssim \|\ep\|_\infty \tfrac t{l(Q)}|Q|^{1/2}.
$$
This yields
$$
  \iint_{R_Q} |\Theta_t f^w_{Q}|^2 \, \frac{dxdt}{t}
\lesssim 
|Q|\|\ep\|_\infty^2
\int_0^{l(Q)} \left( \frac t{l(Q)}\right)^2\frac {dt}t
\lesssim
|Q|\|\ep\|_\infty^2,
$$
which proves that 
$\iint_{R_Q} |\gamma_t A_t f^w_Q|^2\frac{dxdt}t 
  \lesssim 
 (\tb{\wt Q_t}_\op +1)^2\|\ep\|^2_\infty |Q|$.
\end{proof}
\begin{rem}

\begin{itemize}
\item[{\rm (i)}]
Note that using the new test function from Lemma~\ref{lem:newtestfcns}
simplifies the estimate of the term $(\gamma_t A_t- \Theta_t) f^w_Q$ in the above 
proof as compared with the proof of \cite[Proposition 5.9]{AKMc}.
The useful new property of the test functions
from Lemma~\ref{lem:newtestfcns} is that they
belong to $\ran(\Gamma)$.
\item[{\rm (ii)}]
In the proof of Theorem~\ref{thm:pertestsforblock}, we only 
estimate the Carleson norm of the restriction of the 
matrix $\gamma_t$ to the subspace $\wedge_\Gamma$ as this
suffices since we want to bound
the quadratic norm of $\gamma_t A_t f$, and $A_t f$ is always
$\wedge_\Gamma$ valued.
However, 
to prove (\ref{eq:perttildeest1block}) and (\ref{eq:perttildeest2block}),
we use $\Gamma$ being either $N\m d$ or $N \m d^*$.
In these two cases, the space $\wedge_\Gamma$ is either the orthogonal complement
of $\text{span}\{1, e_0\}$ or $\text{span}\{e_{0,1,\ldots,n}, e_{1,\ldots,n}\}$ respectively.
Note also that block matrices preserve these spaces $\wedge_\Gamma^\perp$. It is seen that in these two cases 
$\gamma_t=0$ on $\wedge_\Gamma^\perp$, so actually we do get an estimate
of the Carleson norm of the whole matrix $\gamma_t$.
\end{itemize}
\end{rem}

\subsection{Perturbation of vector coefficients}  \label{section8.2}

In this section we assume that the unperturbed coefficients $B_0$
are of the form 
$$
  B_0= I\oplus A_0\oplus I \oplus\ldots\oplus I,
$$
i.e. $B_0$ only acts nontrivially on the vector part,
and that $A_0$ is a matrix such that $T_{B_0}$ has quadratic estimates.
Note that this hypothesis is true if $A_0$ is either real symmetric, constant
or of block form, by 
Theorem~\ref{thm:mainforrealsymm}, Proposition~\ref{prop:constantquadraticests} 
and Theorem~\ref{thm:unpertblockmain} respectively.
Our goal is to prove Theorem~\ref{thm:main}, by verifying 
the hypothesis of Lemma~\ref{lem:pertests}, as well as proving Theorem~\ref{thm:mainpert}.
We start by reformulating Lemma~\ref{lem:pertests} in 
terms of $e^{-t|T_{B_0}|}$, acting only on functions $f$ in one 
of the two Hardy spaces $E^\pm_{B_0}\mH$, instead of $P_t^{B_0}$.
\begin{thm}  \label{thm:pertestsforrealsymm}
If $B_0$ is as above and if $\|\wt Q_t\|_\off \le C$, then
for all $f\in E^+_{B_0}\mH$ we have estimates
\begin{align}
  \tb{ \wt Q_t\ep (F_t-f) }
  &\le C( \tb{\wt Q_t}_\op+ 1 )\|\ep\|_\infty\|f\|,   \label{eq:pertestsymm2} \\
  \tbb{ \wt Q_t\ep \ud \int_0^t F_s ds }
    &\le C( \tb{\wt Q_t}_\op+ 1 )\|\ep\|_\infty\|f\|, \label{eq:pertestsymm1}  
\end{align}
where $F_t:= e^{-t|T_{B_0}|}f$ is the extension of $f$ as in 
Lemma~\ref{lem:characthardyfcns}.
The corresponding estimates for $f\in E^-_{B_0}\mH$
also holds.
\end{thm}
\begin{proof}[Proof that Theorem~\ref{thm:pertestsforrealsymm} implies 
(\ref{eq:perttildeest1}) and (\ref{eq:perttildeest2})]
We first note that it suffices to prove 
(\ref{eq:perttildeest1}) and (\ref{eq:perttildeest2})
for all $f\in E^+_{B_0}\mH$ and all $f\in E^-_{B_0}\mH$
since we have a Hardy space splitting 
$\mH= E^+_{B_0}\mH\oplus E^-_{B_0}\mH$.
We only consider $f\in E^+_{B_0}\mH$ since the proof for
$f\in E^-_{B_0}\mH$ is similar.
 
Now let $f\in E^+_{B_0}\mH$ and use 
Proposition~\ref{prop:differentpsi}, which
shows that if $\psi\in \Psi(S^o_\nu)$, then
$\tb{\psi_t(T_{B_0})}_\op\lesssim \tb{Q_t^{B_0}}_\op\le C$.
For the estimate (\ref{eq:perttildeest1}), we write
$$
  \wt Q_t\ep t\ud Q_t^{B_0}=
    \wt Q_t\ep (\ud T_{B_0}^{-1})(I- e^{-t|T_{B_0}|})
  + \wt Q_t\ep (\ud T_{B_0}^{-1}) \psi_t(T_{B_0}),
$$
where $\psi(z)=e^{-|z|}-(1+z^2)^{-1}$. 
Note for the first term that 
$T_{B_0}^{-1}(I-e^{-t|T_{B_0}|})f= T_{B_0}^{-1}(f-F_t)
= -T_{B_0}^{-1} \int_0^t \partial_s F_s ds= \int_0^t F_s ds$.
Therefore Theorem~\ref{thm:pertestsforrealsymm},
the boundedness of $\ud T_{B_0}^{-1}$ and
Proposition~\ref{prop:differentpsi} give the estimate
(\ref{eq:perttildeest1}).

For the estimate (\ref{eq:perttildeest2}), note that
it suffices to estimate 
$$
  \wt Q_t\ep t T_{B_0} Q_t^{B_0}= \wt Q_t\ep (I-P_t^{B_0})
$$ 
since $iM_{B_0} T_{B_0}= \ud+ \ud^*_{B_0}$.
But this follows immediately from (\ref{eq:pertestsymm2})
since 
$$
  \wt Q_t\ep (I-P_t^{B_0})f = \wt Q_t\ep (f-F_t)
  +\wt Q_t\ep \psi_t(T_{B_0})f,
$$
with the same $\psi$ as above.
\end{proof}
\begin{rem}
In the case when $B_0$ is a constant matrix,
we can estimate $\wt Q_t\ep (I-P_t^{B_0})f$ directly.
We prove that
$$
  \tb{\wt Q_t\ep (I-P_t^{B_0})f} 
  \lesssim( \tb{\wt Q_t}_\op+ 1 )\|\ep\|_\infty\|f\|
$$
as follows.
Clearly $\tb{\wt Q_t\ep}_\op \lesssim \tb{\wt Q_t}_\op \|\ep\|_\infty$.
For the second term we write $\Theta_t:= \wt Q_t\ep P_t^{B_0}$.
Inserting a standard Fourier mollifier $P_t$, we write
$\Theta_t = \Theta_t(I-P_t) + \Theta_t P_t$.
Here $\tb{\Theta_t(I-P_t)f} \lesssim \tb{P_t^{B_0}(I-P_t)f} \lesssim \|f\|$
is easily verified using the Fourier transform.
On the other hand
$$
  \tb{ \Theta_t P_t f } \le 
  \tb{ (\Theta_t P_t- \gamma_tA_t) f } + \tb{\gamma_t A_t f} 
  \lesssim \|f\| + \|\gamma_t\|_C \|f\|,
$$
using Lemma~\ref{lem:ppa} and Lemma~\ref{lem:Carleson}.
However, since $B_0$ is constant, we have that $T_{B_0}w=0$ if
$w$ is a constant function, and therefore
$$
  \gamma_t w= \wt Q_t\ep (P_t^{B_0}w) = \wt Q_t \ep w.
$$
Therefore Lemma~\ref{lem:sqfcntoCarleson} shows that 
$\|\gamma_t\|_C\lesssim ( \tb{\wt Q_t}_\op+ 1 )\|\ep\|_\infty$.
\end{rem}
We now set up some notation for the proof of 
Theorem~\ref{thm:pertestsforrealsymm}.
We decompose the function $F_t:= e^{-t|T_{B_0}|}f$, where
$f\in E^+_{B_0}\mH$, as
\begin{equation}  \label{eq:splittingF_t}
  F_t= F^0_t+ (F^{1,\no}_t+F^{1,\ta}_t)+F^2_t+\ldots+ F^{n+1}_t,
\end{equation}
and similarly for $f= \lim_{t\rightarrow 0^+} F_t$.
It is important to note the special property that the normal
component of the vector part $F^{1,\no}_t= F^{1,0}_t e_0$ has by 
Lemma~\ref{lem:LaplaceF_0^1}:
it satisfies the divergence form second order equation
$\divv_{t,x} A_0(x) \nabla_{t,x} F^{1,0}=0$.
Furthermore, we decompose the matrix $A_0$ 
as
$$
  A_0=
     \begin{bmatrix}
       a_{\no\no} & a_{\no\ta} \\ a_{\ta\no} & a_{\ta\ta}
     \end{bmatrix},
$$
in the splitting $\mH= N^-\mH\oplus N^+\mH$.
We view the components 
$a_{\no\no}$, $a_{\no\ta}$, $a_{\ta\no}$ and $a_{\ta\ta}$
as operators, and write
$a_{\no\no}(f^{1,0}e_0)= (a_{00}f^{1,0})e_0$,
$a_{\no\ta} f^{1,\ta}= (a_{0\ta} \cdot f^{1,\ta})e_0$ and
$a_{\ta\no}(f^{1,0}e_0)= f^{1,0}a_{\ta 0}$, where
$a_{00}$ is a scalar and $a_{0\ta}$ and $a_{\ta 0}$ are
vectors.

We introduce an auxiliary block matrix
$$
  \hB_0= I\oplus \hA_0\oplus I \oplus\ldots\oplus I, \qquad
  \hA_0=
     \begin{bmatrix}
       a_{\no\no} & 0 \\ 0 & a_{\ta\ta}
     \end{bmatrix}.
$$
\begin{lem}   \label{lem:wedge1formulae}
Let $F_t:= e^{-t|T_{B_0}|}f$, where $f\in E^+_{B_0}\mH$,
so that $(\partial_t + T_{B_0}) F_t=0$.
Then
$$
  (\partial_t + T_{\hB_0})F_t 
  = -i (a_{00}^{-1} a_{\no\ta}) \ud F_t^{1,\no} 
    -i \ud^*_{\hB_0} (a_{\ta\ta}^{-1} a_{\ta\no}) F_t^{1,\no}
  = \partial_t F_t^{1,\no} + i \ud^*_{\hB_0} F_t^{1,\ta}
$$
\end{lem}
\begin{proof}
To prove the first identity, note that
\begin{multline*}
  (\partial_t + T_{\hB_0})F_t=
   (T_{\hB_0}-T_{B_0})F_t \\
   =
   -i (\hB_0 N^+- N^-\hB_0)^{-1} (\hB_0\ud+\ud^*\hB_0)F_t+
   i( B_0 N^+- N^- B_0)^{-1} ( B_0\ud+\ud^* B_0)F_t \\
   = i (\hB_0 N^+- N^-\hB_0)^{-1}
   ((B_0-\hB_0)\ud + \ud^*(B_0- \hB_0))F_t
\end{multline*}
since
$\hB_0 N^+- N^-\hB_0= B_0 N^+- N^- B_0$.
The vector part of this matrix is
$
     \begin{bmatrix}
       -a_{\no\no} & 0 \\ 0 & a_{\ta\ta}
     \end{bmatrix}
$.
Furthermore
$$
  B_0- \hB_0= 0\oplus 
     \begin{bmatrix}
       0 & a_{\no\ta} \\ a_{\ta\no} & 0
     \end{bmatrix}
  \oplus 0 \oplus \ldots\oplus 0,
$$
which shows that
$$
  (B_0-\hB_0)\ud F_t= a_{\no\ta}\ud F_t^{1,\no}
  \qquad\text{and}\qquad
  \ud^*(B_0- \hB_0)F_t= \ud^* a_{\ta\no} F_t^{1,\no},
$$
using the mapping properties of $\ud$ and $\ud^*$ from 
Remark~\ref{rem:mappingud}.
Thus
\begin{multline*}
  (\partial_t + T_{\hB_0})F_t=
  -i a_{\no\no}^{-1}( a_{\no\ta}\ud F_t^{1,\no}+ \ud^* a_{\ta\no} F_t^{1,\no} ) \\
  = -i (a_{00}^{-1} a_{\no\ta}) \ud F_t^{1,\no} 
    -i \ud^*_{\hB_0} (a_{\ta\ta}^{-1}a_{\ta\no}) F_t^{1,\no},
\end{multline*}
since 
$\hB_0 N^+- N^-\hB_0= - a_{\no\no}$
on normal vector fields.

To prove the second identity, we multiply the Dirac equation
$(\partial_t + T_{B_0}) F_t=0$ by $B_0 N^+- N^- B_0$ to obtain
$$
  (B_0 N^+- N^- B_0)\partial_t F_t -i(B_0\ud+\ud^* B_0)F_t= 0.
$$
The normal component of the vector part on the left hand side is
$$
  - a_{\no\no} \partial_t F^{1,\no}_t 
  -i( a_{\no\ta} \ud F^{1,\no}_t + 
       \ud^*( a_{\ta\no}F^{1,\no}_t+a_{\ta\ta}F^{1,\ta}_t ) )=0.
$$
Here we have used the expression for the vector part of $B_0 N^+- N^- B_0$
above for the first term.
For the other terms we recall from Remark~\ref{rem:mappingud} 
that the vector part of 
$\ud F_t$ is $\ud F_t^{1,\no}$ which is tangential, and
that the normal vector part of $\ud^* B_0 F_t$ is 
$\ud^* (B_0 F_t)^{1,\ta}$.
Therefore, after multiplying the equation with $a_{00}^{-1}$, we
obtain
$$
  -i (a_{00}^{-1} a_{\no\ta}) \ud F_t^{1,\no} 
    -i \ud^*_{\hB_0} (a_{\ta\ta}^{-1} a_{\ta\no}) F_t^{1,\no}
  = \partial_t F_t^{1,\no} + i \ud^*_{\hB_0} F_t^{1,\ta},
$$
which proves the lemma.
\end{proof}
As in Lemma~\ref{lem:blockasswapping} we note that since $\hB_0$ 
is a block matrix
$$
  T_{\hB_0}= \Pi_{\hB_0}= \Gamma+ \hB_0^{-1}\Gamma^* \hB_0, 
  \qquad \text{where } \Gamma:= -i N \ud.
$$
To prove Theorem~\ref{thm:pertestsforrealsymm} we shall need the
following corollary to Theorem~\ref{thm:pertestsforblock}.
\begin{cor}   \label{cor:harmanalestforsymm}
  Let $\hB_0$ be the block matrix defined above, 
assume that $\|\wt Q_t\|_\off \le C$ and let 
$v\in L_\infty(\R^n;\wedge)$ be a function with norm
$\|v\|_\infty\le C$.
Then for all $f\in E^+_{B_0}\mH$ we have estimates
\begin{align}
  \tb{ \wt Q_t \ep t^2\Gamma \Gamma^*_{\hB_0} P_t^{\hB_0}( F^{1,0}_t v ) }
    &\lesssim ( \tb{\wt Q_t}_\op+ 1 )\|\ep\|_\infty\|f\|,   \\
  \tb{ \wt Q_t\ep P_t^{\hB_0}( F^{1,0}_t v ) }
  &\lesssim ( \tb{\wt Q_t}_\op+ 1 )\|\ep\|_\infty\|f\|,  
\end{align}
where $F_t:= e^{-t|T_{B_0}|}f$ and $F^{1,0}_t = (F_t, e_0)$.
The corresponding estimates for $f\in E^-_{B_0}\mH$
also holds.
\end{cor}
We defer the proof until the end of this section, and turn to 
a lemma in preparation for the proof of
Theorem~\ref{thm:pertestsforrealsymm}.
\begin{lem}  \label{lem:SchurHodge}
If $f\in E^+_{B_0}\mH$ and $F_t:= e^{-t|T_{B_0}|}f$, then
\begin{align}
  \tbb{ \int_0^t \frac st(s\partial_s F_s)\frac {ds}s } 
  & \le \tb{t\partial_t F_t} \lesssim \|f\|, \label{eq:Schurterm} \\
  \tb{t \ud F_t} &\lesssim \tb{t\partial_t F_t} \lesssim \|f\|. 
  \label{eq:Hodgeterm}
\end{align}
The corresponding estimates for $f\in E^-_{B_0}\mH$ also holds.
\end{lem}
\begin{proof}
The proof of (\ref{eq:Schurterm}) uses Schur estimates.
Applying Cauchy--Schwarz inequality, we estimate the square
of the left hand side by
\begin{multline*}
  \int_0^\infty  \left\| \int_0^t 
     \frac st (s\partial_s F_s)\frac {ds}s \right\|^2  \frac{dt}t
  \le 
  \int_0^\infty  \left( \int_0^t \frac st \frac {ds}s \right)  
  \left( \int_0^t 
     \frac st \| s\partial_s F_s \|^2 \frac {ds}s \right)^2  \frac{dt}t
    \\
  =\int_0^\infty  \left( \int_s^\infty \frac st \frac {dt}t \right)  
     \| s\partial_s F_s \|^2 \frac {ds}s 
  = \tb{ \psi_t(T_{B_0})f }^2 \lesssim \|f\|^2,
\end{multline*}
where $\psi(z):= z e^{-|z|}$.

To prove (\ref{eq:Hodgeterm}), we note that 
$iM_{B_0} T_{B_0}= \ud+\ud^*_{B_0}$
and write $\PP^1_{B_0}$, $\PP^2_{B_0}$ for the two Hodge projections
corresponding to the splitting $\mH= \nul(\ud)\oplus \nul(\ud^*_{B_0})$.
Note that these projections are bounded by Lemma~\ref{lem:hodge}(i).
We get
$$
  t\ud F_t = i \PP^1_{B_0} M_{B_0}(tT_{B_0} F_t)= i \PP^1_{B_0} M_{B_0} \psi_t(T_{B_0})f,
$$
from which (\ref{eq:Hodgeterm}) follows.
\end{proof}
\begin{proof}[Proof that Corollary~\ref{cor:harmanalestforsymm} implies 
the estimate (\ref{eq:pertestsymm2})]
By inserting $I= P_t^{\hB_0} + (I- P_t^{\hB_0} )$ we write the
left hand side in (\ref{eq:pertestsymm2}) as 
$$
  \wt Q_t \ep P_t^{\hB_0} (F_t -f)+
  \wt Q_t \ep (I-P_t^{\hB_0} )F_t -
  \wt Q_t \ep (I-P_t^{\hB_0} )f =: X_1+X_2-X_3.
$$
For $X_3$ we get from Theorem~\ref{thm:pertestsforblock} the estimate
$\tb{X_3}\lesssim  (\tb{\wt Q_t}_\op+1) \|\ep\|_\infty\|f\|$
since $I-P_t^{\hB_0}= t^2\Gamma\Gamma^*_{\hB_0}P_t^{\hB_0}+
t^2\Gamma^*_{\hB_0}\Gamma P_t^{\hB_0}$.
For the term $X_2$ we use the first identity in 
Lemma~\ref{lem:wedge1formulae} to obtain
\begin{multline*}
  X_2= \wt Q_t \ep Q_t^{\hB_0} tT_{\hB_0} F_t =   
  -\wt Q_t \ep Q_t^{\hB_0} (t\partial_t F_t)
  -i\wt Q_t \ep Q_t^{\hB_0} a_{00}^{-1} a_{\no\ta} ( t\ud F_t^{1,\no})\\
  -i\wt Q_t \ep ( Q_t^{\hB_0} t\ud^*_{\hB_0})
   (F_t^{1,0}(a_{\ta\ta}^{-1} a_{\ta\no}e_0) )
   =: -X_4-iX_5 -iX_6.
\end{multline*}
We have the estimate $\tb{X_4} \lesssim 
\|\ep\|_\infty \tb{t\partial_t F_t}\lesssim 
\|\ep\|_\infty\|f\|$.
For $X_5$, we see from Remark~\ref{rem:mappingud}
that $\ud F_t^{1,\no}$ is the vector part of $\ud F_t$. 
Thus $\tb{X_5} \lesssim \|\ep\|_\infty \tb{t\ud F_t}\lesssim \|\ep\|_\infty\|f\|$ by 
(\ref{eq:Hodgeterm}).
To handle the term $X_6$ we note that
$Q_t^{\hB_0}t \ud^*_{\hB_0}= -Q_t^{\hB_0}t \Gamma^*_{\hB_0}N=
-t^2\Gamma\Gamma^*_{\hB_0}P_t^{\hB_0} N$
using
Remark~\ref{rem:gammaintertwining}.
Thus we obtain from Corollary~\ref{cor:harmanalestforsymm}, with
$v= a_{\ta\ta}^{-1} a_{\ta\no}e_0$, the estimate
$\tb{X_6}\lesssim  (\tb{\wt Q_t}_\op+1) \|\ep\|_\infty\|f\|$.

It remains to estimate the term $X_1$.
To handle this, we separate the normal vector part as
$$
  X_1= \wt Q_t \ep P_t^{\hB_0} (F^{1,0}_t e_0)-
  \wt Q_t \ep P_t^{\hB_0} (f^{1,0}e_0)+
  \wt Q_t \ep P_t^{\hB_0} (G_t -g)=: X_7-X_8+X_9,
$$
where $G_t:= F_t - F_t^{1,\no}$ and $g= f-f^{1,\no}$.
From Corollary~\ref{cor:harmanalestforsymm}, with
$v= e_0$, we get the estimate
$\tb{X_7}\lesssim  (\tb{\wt Q_t}_\op+1) \|\ep\|_\infty\|f\|$.
For the term $X_8$, we write
$P_t^{\hB_0}= I- t^2\Gamma\Gamma^*_{\hB_0}P_t^{\hB_0}-
  t^2\Gamma^*_{\hB_0}\Gamma P_t^{\hB_0}$.
From Theorem~\ref{thm:pertestsforblock} we obtain the estimate
$\tb{X_8}\lesssim  (\tb{\wt Q_t}_\op+1) \|\ep\|_\infty\|f\|$.
For the term $X_9$, we integrate by parts to obtain
$$
  X_9 =\wt Q_t \ep P_t^{\hB_0}(t\partial_t G_t)-
    \wt Q_t \ep P_t^{\hB_0}\left( \int_0^t s\partial_s^2 G_s ds \right)
    =: X_{10}-X_{11}.
$$
We have the estimate $\tb{X_{10}} \lesssim 
\|\ep\|_\infty \tb{t\partial_t F_t}\lesssim 
\|\ep\|_\infty\|f\|$.
For the term $X_{11}$ we apply $\partial_t$ to the last expression for 
$(\partial_t + T_{\hB_0})F_t$ in Lemma~\ref{lem:wedge1formulae} and get
\begin{equation}  \label{eq:dofdoblack}
 \partial_s^2 G_s + \partial_s T_{\hB_0}F_s = 
   i \ud^*_{\hB_0} \partial_s F_s^{1,\ta}.
\end{equation}
Thus
\begin{multline*}
  X_{11}= 
  -\wt Q_t \ep (P_t^{\hB_0}tT_{\hB_0})\left( \int_0^t \frac st
    (s\partial_s F_s)\frac{ds}s \right) \\
   +
    i \wt Q_t \ep (P_t^{\hB_0} t\ud^*_{\hB_0})
    \left( \int_0^t \frac st (s \partial_s F_s^{1,\ta}) \frac{ds}s \right)
  =: -X_{12} + i X_{13}.
\end{multline*}
Both $\tb{X_{12}}$ and $\tb{X_{13}}$ can now be estimated with
$\|\ep\|_\infty\|f\|$ by (\ref{eq:Schurterm}) since
$P_t^{\hB_0}tT_{\hB_0}= Q_t^{\hB_0}$ and 
$P_t^{\hB_0} t\ud^*_{\hB_0}$ are uniformly bounded by 
Corollary~\ref{cor:opfamilies}.
This proves the estimate (\ref{eq:pertestsymm2}).
\end{proof}
\begin{proof}[Proof that Corollary~\ref{cor:harmanalestforsymm} implies the estimate (\ref{eq:pertestsymm1})]
We write the left hand side in (\ref{eq:pertestsymm1}),
using integration by parts, as
$$
  \wt Q_t \ep \ud \int_0^t F_s ds =
  \wt Q_t \ep t\ud  F_t - \wt Q_t \ep \ud \int_0^t s\partial_s F_s ds 
  =: X_1 - X_2.
$$
For $X_1$ we have 
$\tb{\wt Q_t \ep (t\ud  F_t)}\lesssim \|\ep\|_\infty\|f\|$
by (\ref{eq:Hodgeterm}).
For $X_2$, we write $I= P_t^{\hB_0} + (I- P_t^{\hB_0})$ and get
$$
  X_2= \wt Q_t \ep (t\ud P_t^{\hB_0})\int_0^t \frac st (s\partial_s F_s)\frac{ds}s
  +
  \wt Q_t \ep t\ud Q_t^{\hB_0} \int_0^t s\partial_s T_{\hB_0} F_s ds
  =: X_3+X_4.
$$
Using that $\|t\ud P_t^{\hB_0}\|\le C$ by Corollary~\ref{cor:opfamilies}, 
and (\ref{eq:Schurterm}) shows that 
$\tb{X_3}\lesssim \|\ep\|_\infty\|f\|$.
To handle $X_4$, we use the identity (\ref{eq:dofdoblack}), which gives
$$
  X_4= \wt Q_t \ep t\ud Q_t^{\hB_0}
  \left( - \int_0^t s\partial_s^2 G_s ds +
    i \int_0^t s \ud^*_{\hB_0} \partial_s F_s^{1,\ta} ds \right)
  =: -X_5 + i X_6.
$$
For $X_6$ we note that 
$\ud Q_t^{\hB_0}\ud^*_{\hB_0}= N\ud Q_t^{\hB_0}N\ud^*_{\hB_0}
= -\Gamma Q_t^{\hB_0}\Gamma^*_{\hB_0}= -\Gamma^2 Q_t^{\hB_0}=0$
using Remark~\ref{rem:gammaintertwining}, and thus $X_6=0$.
To handle $X_5$, we rewrite this with an integration by parts
as
$$
  X_5= \wt Q_t \ep (t\ud Q_t^{\hB_0})(t\partial_t G_t)
  - \wt Q_t \ep t\ud Q_t^{\hB_0} G_t
  + \wt Q_t \ep (t\ud Q_t^{\hB_0} g)=: X_6-X_7+X_8,
$$
where $g= f- f^{1,\no}$.
Using Lemma~\ref{lem:SchurHodge} and that $\|t\ud Q_t^{\hB_0}\|\le C$ by Corollary~\ref{cor:opfamilies}, 
we get
$\tb{X_6}\lesssim \|\ep\|_\infty\|f\|$.
For $X_8$, we note that 
$t\ud Q_t^{\hB_0}= iN t^2 \Gamma\Gamma^*_{\hB_0}P_t^{\hB_0}$.
Thus we can apply Theorem~\ref{thm:pertestsforblock} to obtain
$\tb{X_8}\lesssim  (\tb{\wt Q_t}_\op+1)\|\ep\|_\infty \|f\|$.

We now write $G_t= F_t- F^{1,0}_t e_0$ and get
$$
  X_7= \wt Q_t \ep t\ud P_t^{\hB_0} (tT_{\hB_0}F_t)-
  \wt Q_t \ep t\ud Q_t^{\hB_0} (F^{1,0}_te_0) =: X_9-X_{10}.
$$
Again noting that
$t\ud Q_t^{\hB_0}= iN t^2 \Gamma\Gamma^*_{\hB_0}P_t^{\hB_0}$
we obtain from Corollary~\ref{cor:harmanalestforsymm}, with
$v= e_0$, the estimate
$\tb{X_{10}}\lesssim  (\tb{\wt Q_t}_\op+1)\|\ep\|_\infty \|f\|$.
The term $X_9$ remains, on which we use
the first identity in Lemma~\ref{lem:wedge1formulae}
to obtain
\begin{multline*}
  X_9= -\wt Q_t \ep t\ud P_t^{\hB_0} (t\partial_t F_t)
  -i\wt Q_t \ep( t\ud P_t^{\hB_0})a_{00}^{-1} a_{\no\ta}( t\ud F_t^{1,\no})\\
  -i\wt Q_t \ep (t\ud P_t^{\hB_0} t\ud^*_{\hB_0})
   (F_t^{1,0}(a_{\ta\ta}^{-1} a_{\ta\no}e_0) )
   =: -X_{11}-iX_{12} -iX_{13}.
\end{multline*}
Using that $\|t\ud P_t^{\hB_0}\|\le C$ by Corollary~\ref{cor:opfamilies}, 
shows that 
$\tb{X_{11}}\lesssim \|\ep\|_\infty\|f\|$.
For $X_{12}$ we see from Remark~\ref{rem:mappingud} 
that $\ud F_t^{1,\no}$ is the vector part of $\ud F_t$. 
Thus $\tb{X_{12}} \lesssim \|\ep\|_\infty \tb{t\ud F_t}\lesssim \|\ep\|_\infty\|f\|$ by 
(\ref{eq:Hodgeterm}).
To handle the final term $X_{13}$ we note that
$\ud P_t^{\hB_0} \ud^*_{\hB_0}= \ud N P_t^{\hB_0} N \ud^*_{\hB_0}
= -\Gamma P_t^{\hB_0}\Gamma^*_{\hB_0}= -\Gamma \Gamma^*_{\hB_0}P_t^{\hB_0}$
using Remark~\ref{rem:gammaintertwining}.
Thus we obtain from Corollary~\ref{cor:harmanalestforsymm}, with
$v= a_{\ta\ta}^{-1} a_{\ta\no}e_0$, the estimate
$\tb{X_{13}}\lesssim  (\tb{\wt Q_t}_\op+1) \|\ep\|_\infty\|f\|$.
This proves the estimate (\ref{eq:pertestsymm1}).
\end{proof}
\begin{proof}[Proof of Corollary~\ref{cor:harmanalestforsymm}]
Let $f\in E^+_{B_0}\mH$ and consider the functions $F_t=e^{-t|T_{B_0}|}f$
and $G_t:= P_t^{B_0}f$.
We note that $F_t-G_t= \psi_t(T_{B_0})f$, where 
$\psi(z)= e^{-|z|}- (1+z^2)^{-1}\in \Psi(S^o_\nu)$.
Thus it suffices
to prove the estimate 
$\tb{\Theta^i_t(G_t^{1,0})}\lesssim (\tb{\wt Q_t}_\op+1) \|\ep\|_\infty\|f\|$, 
$i= 1,2$, for the two families of operators
\begin{align}
  \Theta_t^1 &:=   \wt Q_t \ep t^2\Gamma \Gamma^*_{\hB_0} P_t^{\hB_0} M_v,   \\
  \Theta_t^2 &:=  \wt Q_t\ep P_t^{\hB_0} M_v,
\end{align}
where $M_v$ denotes the multiplication operator $M_v (f) := v f$.
To this end, we write 
$$
  \Theta_t^i(G_t^{1,0})=( \Theta_t^i-\gamma_t^i A_t)G_t^{1,0} + \gamma_t^i A_t G_t^{1,0},
$$
where $\gamma_t^i(x)$ is the principal part of the operator family $\Theta_t^i$
as in Definition~~\ref{defn:princpart}.
Using the principal part approximation Lemma~\ref{lem:ppa} and Carleson's 
Lemma~\ref{lem:Carleson} we obtain the estimate
$$
  \tb{ \Theta_t^i(G_t^{1,0}) }= \| \Theta_t^i \|_\off \tb{t\nabla (G_t^{1,0}) } + 
    \| \gamma_t^i \|_C \|N_*( A_t G_t^{1,0} )\|.
$$
By Proposition~\ref{pseudoloc} and Lemma~\ref{lem:offdiagcomposition}, we have 
$\|\Theta_t^i \|_\off \lesssim \|\ep\|_\infty$.
Furthermore, by Lemma~\ref{lem:sqfcntoCarleson} and Theorem~\ref{thm:pertestsforblock}
we have
$$
  \| \gamma_t^i \|_C \lesssim \tb{\Theta_t^i}_\op + \|\Theta_t^i \|_\off \lesssim 
    (\tb{\wt Q_t}_\op+1)\|\ep\|_\infty.
$$
For $\Theta_t^2$ we have used that
$\Theta_t^2=   \wt Q_t\ep(I- t^2\Gamma \Gamma^*_{\hB_0} P_t^{\hB_0} -  
  t^2 \Gamma^*_{\hB_0}\Gamma P_t^{\hB_0} ) M_v $.

To bound $\tb{t\nabla (G_t^{1,0}) }$, we note that the vector part of $\ud G_t$ is
$$
  \ud G_t^{1,\no}= i\m d (e_0 G_t^{1,0})= -i d(G_t^{1,0})= -i \nabla(G_t^{1,0}).
$$
Thus, similar to the proof of (\ref{eq:Hodgeterm}), it follows that 
$$
  \tb{t\nabla (G_t^{1,0}) }\lesssim \tb{t\ud G_t }\lesssim 
\tb{t T_{B_0}P_t^{B_0}f}\lesssim \|f\|.
$$
Finally, to bound $\|N_*( A_t G_t^{1,0} )\|$ we write 
$G_t = \tfrac 12(H_t+ H_{-t})$, where $H_t:= (1+ it T_{B_0})^{-1} f$.
We now observe that the divergence form equation and estimate for $H_t^{1,0}$
in Lemma~\ref{lem:H_t} in fact holds for all $f\in \mH$, by inspection 
of the proof.
We get 
\begin{multline*}
  N_*(A_t H_t^{1,0})(x) = \sup_{|y-x|<t} 
  \left| \barint_{\hspace{-6pt} Q(y,t) } H_t^{1,0}(z) dz  \right| \\
  \lesssim \sup_{t>0}\left( \barint_{\hspace{-6pt} B(x, r_0 t) } 
     |H_t^{1,0}(z)|^q dz\right)^{1/q}
  \lesssim ( M(|f|^p)(x) )^{1/p},
\end{multline*}
where $Q(y,t)$ denotes the dyadic cube $Q\in \dyadic_t$ which contains $y$.
Using the boundedness of the Hardy--Littlewood maximal function
on $L_{2/p}(\R^n)$, we obtain
$$
  \| N_*(A_t H_t^{1,0}) \|_2 \lesssim \| ( M(|f|^p) )^{1/p} \|_2
  = \| M(|f|^p) \|_{2/p}^{1/p} \lesssim \| |f|^p \|_{2/p}^{1/p}
  = \| f\|_2.
$$
A similar argument shows that $\|N_*( A_t H_{-t}^{1,0} )\|\lesssim \|f\|$,
and thus $\|N_*( A_t G_t^{1,0} )\| \lesssim \|f\|$.
We have proved that 
$\tb{\Theta^i_t(G_t^{1,0})}\lesssim (\tb{\wt Q_t}_\op+1) \|\ep\|_\infty\|f\|$,
and therefore Corollary~\ref{cor:harmanalestforsymm}.
\end{proof}

\subsection{Proof of main theorems}   \label{section8.3}

We are now in position to prove Theorems~\ref{thm:main},
\ref{thm:mainpert} and \ref{thm:transmain} stated in the
introduction.
\begin{proof}[Proof of Theorem~\ref{thm:transmain}]
Given a perturbation $B^k\in L_\infty(\R^n; \mL(\wedge^k))$
of the unperturbed coefficients $B_0^k\in L_\infty(\R^n; \mL(\wedge^k))$,
we introduce
$B:= I\oplus \ldots \oplus B^k \oplus\ldots\oplus I
\in L_\infty(\R^n; \mL(\wedge))$ and 
$B_0:= I\oplus \ldots \oplus B_0^k \oplus\ldots\oplus I
\in L_\infty(\R^n; \mL(\wedge))$ acting in all $\mH$.
By Theorem~\ref{thm:unpertblockmain}(i), $T_{B_0}$ satisfies
quadratic estimates and by Lemma~\ref{lem:pertests}
and Theorem~\ref{thm:pertestsforblock} there exists
$\epsilon>0$ such that we have quadratic estimates
$$
  \tb{Q_t^B f}\approx \|f\|, \qquad\text{whenever } 
  \|B-B_0\|_\infty\le\epsilon, \,\, f\in\mH.
$$
Therefore, by Proposition~\ref{prop:QtoFCalc} and Proposition~\ref{prop:lipcont} 
we have when $\|B-B_0\|_\infty\le\epsilon/2$
well defined and bounded operators
$E_B= \sgn(T_B)$ which depend Lipschitz continuously on $B$, i.e.
$$
  \| E_{B_2} - E_{B_1} \| \le C \|B_2 - B_1\|_\infty, 
  \qquad \text{when } \|B_i- B_0\|<\epsilon/2,\,\, i=1,2.
$$
To prove that (Tr-$B^k\alpha^\pm$) is well posed, note that
by Lemma~\ref{lem:optransmain} it suffices to show that $\lambda- E_B N_B$
is invertible since then in particular
$\lambda- E_{B^k} N_{B^k} = (\lambda- E_B N_B)|_{\hat\mH_B^k}$
is invertible.
Here the spectral parameter is $\lambda= (\alpha+1)/(\alpha-1)$ and
$\alpha:= \alpha^+/\alpha^-$.
By Theorem~\ref{thm:unpertblockmain}(ii), the unperturbed 
operator $\lambda- E_{B_0}N_{B_0}$ is invertible when 
$\lambda^2+1 \ne 0$.
For the perturbed operator we write
$$
  \lambda- E_B N_B = (\lambda- E_{B_0}N_{B_0})
  \Big(I+ \tfrac 1{\lambda^2+1}(\lambda - N_{B_0}E_{B_0})
  ( E_{B_0}N_{B_0} - E_BN_B ) \Big).
$$
Here $\| E_{B_0}N_{B_0} - E_BN_B \| \lesssim \|B- B_0\|_\infty$
since we clearly have Lipschitz continuity 
$\|N_{B_2} - N_{B_1}\| \lesssim \|B_2- B_1\|_\infty$.
It follows that $\lambda- E_B N_B$ 
is invertible when $\|B- B_0\|_\infty \le C |\lambda^2+1|$.
Under the assumption $\|B- B_0\|_\infty < \epsilon/2$ with $\epsilon$
small we can replace 
$\lambda^2+1 = 2(\alpha^2+1)/(\alpha-1)^2$ with $\alpha^2+1$.

This proves that for each boundary function $g$ there exists
a unique solution $f=f^+ + f^-$ satisfying the jump conditions
in (Tr-$B^k\alpha^\pm$) and $\|f^+\|+\|f^-\|\approx\|f\|\approx \|g\|$.
The norm estimates for $F^\pm$ in Theorem~\ref{thm:transmain} 
now follows from Lemma~\ref{lem:characthardyfcns}.
Finally it follows from Proposition~\ref{prop:lipcont} that the solution operator
$$
  F^\pm(t,x)= 2 e^{\mp t|T_B|}E_B^\pm\big((\alpha^++\alpha^-)E_B - (\alpha^+-\alpha^-)N_B\big)^{-1} g(x)
$$
depends Lipschitz continuously on $B$.
Indeed, from above it suffices to prove Lipschitz continuity for $e^{\mp t|T_B|}$.
For the trippel bar norm, we here take
$\psi(z)= ze^{\mp|z|}$ in Proposition~\ref{prop:lipcont}, and for $\|F^\pm_t\|_2$ we take
$b(z)= e^{\mp t|z|}$.
This completes the proof of Theorem~\ref{thm:transmain}.
\end{proof}
Before turning to the proofs of Theorems~\ref{thm:main} and
\ref{thm:mainpert}, we note some corollaries of Theorem~\ref{thm:transmain}.
First, if we let $k=1$ in Theorem~\ref{thm:transmain}, then it
proves that the following Neumann--regularity transmission
problem is well posed for small $L_\infty$ perturbations $A$
of a block matrix $A_0$.

\vspace{2mm}
\noindent THE TRANSMISSION PROBLEM (Tr-$A\alpha^\pm$).

Let $\alpha^\pm\in\C$ be given jump parameters.
Given scalar functions $\psi, \phi:\R^n\rightarrow\C$ with 
$\nabla_x \psi\in L_2(\R^n;\C^n)$ and $\phi\in L_2(\R^n;\C)$,
find gradient vector fields $F^\pm(t,x)= \nabla_{t,x}U^\pm(t,x)$ 
in $\R^{n+1}_\pm$ such that 
$F_t^\pm\in C^1(\R_\pm; L_2(\R^n;\C^{n+1}))$ and
$F^\pm$ satisfies (\ref{eq:Laplacein1order}) for $\pm t>0$,
and furthermore $\lim_{t\rightarrow \pm\infty}F^\pm_t=0$ and 
$\lim_{t\rightarrow 0^\pm}F^\pm_t=f^\pm$
in $L_2$ norm, 
where the traces $f^\pm$ satisfy the jump conditions
\begin{equation*}  
  \left\{
  \begin{array}{rcl}
    \alpha^- \nabla_x U^+(0,x) - \alpha^+ \nabla_x U^-(0,x) &=& \nabla_x\psi(x), \\
    \alpha^+ \tdd{U^+}{\nu_A}(0,x) - \alpha^- \tdd{U^-}{\nu_A}(0,x) &=& \phi(x),
  \end{array}
  \right.
\end{equation*}
where $\nabla_x U^\pm (0,x)= f_\ta^\pm(x)$ and
$\tdd {U^\pm} {\nu_A}= (A f^\pm , e_0)$ denotes the conormal derivative.

\vspace{2mm}

Secondly, Theorem~\ref{thm:transmain} give
perturbation results for the following boundary value problems
for $k$-vector fields.

\vspace{2mm}
\noindent THE NORMAL BVP (Nor-$B^k$).

Given a $k$-vector field $g\in \hat\mH^k_B$,
find a $k$-vector field $F(t,x)$ in $\R^{n+1}_+$
such that $F_t\in C^1(\R_+; L_2(\R^n;\wedge^k))$ and
$F$ satisfies (\ref{eq:diracwedgek}) for $t>0$,
and furthermore $\lim_{t\rightarrow\infty}F_t=0$ and 
$\lim_{t\rightarrow 0}F_t=f$ in $L_2$ norm, 
where $f$ satisfies
$$
  e_0\lctr (B^k f) = e_0 \lctr (B^k g)
  \qquad \text{on }\R^n =\partial \R^{n+1}_+.
$$

\vspace{2mm}
\noindent THE TANGENTIAL BVP (Tan-$B^k$).

Given a $k$-vector field $g\in \hat\mH^k_B$,
find a $k$-vector field $F(t,x)$ in $\R^{n+1}_+$ 
such that $F_t\in C^1(\R_+; L_2(\R^n;\wedge^k))$ and
$F$ satisfies (\ref{eq:diracwedgek}) for $t>0$,
and furthermore $\lim_{t\rightarrow\infty}F_t=0$ and 
$\lim_{t\rightarrow 0}F_t=f$ in $L_2$ norm, 
where $f$ satisfies
$$
  e_0\wedg f = e_0\wedg g
  \qquad \text{on }\R^n =\partial \R^{n+1}_+.
$$

\begin{cor}  \label{cor:formBVP}
Let $B_0^k=B_0^k(x)\in L_\infty(\R^n;\mL(\wedge^k))$ be accretive
and assume that $B_0^k$ is a block matrix.
Then there exists $\epsilon>0$ depending only on the
constants $\|B_0^k\|_\infty$ and $\kappa_{B^k_0}$ and dimension $n$, such that if 
$B^k\in L_\infty(\R^n;\mL(\wedge^k))$ satisfies 
$\|B^k-B^k_0\|_\infty <\epsilon$, then
the normal and tangential boundary value problems (Nor-$B^k$) and (Tan-$B^k$) above
are well posed.
\end{cor}
\begin{proof}
(i)
For (Nor-$B^k$) in $\R^{n+1}_+$, we let 
$\alpha^-=0$ and $\alpha^+= 1$ in (Tr-$B^k\alpha^\pm$).
Then we obtain two decoupled jump conditions
\begin{equation*}  
  \left\{
  \begin{array}{rcl}
    -e_0\wedg f^-  &=& e_0\wedg g \\
    e_0\lctr(B^k f^+) &=& e_0\lctr(B^k g).
  \end{array}
  \right.
\end{equation*}
Discarding the solution $F^-$, we obtain a unique solution
$F= F^+$ to (Nor-$B^k$).

(ii) 
For (Tan-$B^k$) in $\R^{n+1}_+$, we let 
$\alpha^-=1$ and $\alpha^+= 0$ in (Tr-$B^k\alpha^\pm$).
Then we obtain two decoupled jump conditions
\begin{equation*}  
  \left\{
  \begin{array}{rcl}
    e_0\wedg f^+  &=& e_0\wedg g \\
    -e_0\lctr(B^k f^-) &=& e_0\lctr(B^k g).
  \end{array}
  \right.
\end{equation*}
Discarding the solution $F^-$, we obtain a unique solution
$F= F^+$ to (Tan-$B^k$).
\end{proof}
Note that well posedness of (Neu-$A$) and (Reg-$A$) for
Theorem~\ref{thm:main}(b) is 
the special case $k=1$ of Corollary~\ref{cor:formBVP},
as well as the special cases $(\alpha^+,\alpha^-)=(1,0)$
and $(0,1)$ respectively of (Tr-$A\alpha^\pm$).

\begin{proof}[Proof of Theorem~\ref{thm:mainpert}]
Let $A_0$ be such that $T_{B_0}$ has quadratic estimates in $\mH$, where
$B_0= I\oplus A_0\oplus I\oplus\ldots\oplus I$.
Thus by Lemma~\ref{lem:pertests} and Theorem~\ref{thm:pertestsforrealsymm}
there exists $\epsilon>0$ such that we have quadratic estimates
$$
  \tb{Q_t^B f}\approx \|f\|, \qquad\text{whenever } 
  \|B-B_0\|_\infty <\epsilon, \,\, f\in\mH.
$$
Therefore, by Proposition~\ref{prop:QtoFCalc}, we have when 
$\|B-B_0\|_\infty <\epsilon$ well defined and bounded operators
$E_B= \sgn(T_B)$. 
With Lemma~\ref{lem:preservehats}, these restricts to bounded 
operators $E_A$ in $\hat\mH^1$.
In particular $f\in\hat\mH^1$ can be decomposed as $f=f^++f^-$,
where $f^\pm:= E^\pm_A f$ and $\|f\|\approx \|f^+\|+ \|f^-\|$.
Moreover, Lemma~\ref{lem:characthardyfcns} and Proposition~\ref{prop:modnontang}
proves the stated norm equivalences for $F^\pm$.
\end{proof}

\begin{proof}[Proof of Theorem~\ref{thm:main}]
Given a perturbation $A\in L_\infty(\R^n; \mL(\wedge^1))$
of the unperturbed coefficients $A_0\in L_\infty(\R^n; \mL(\wedge^1))$,
which we assume are either of block form, real symmetric or constant,
we introduce
$B:= I\oplus A \oplus I\oplus\ldots\oplus I
\in L_\infty(\R^n; \mL(\wedge))$ and 
$B_0:= I\oplus A_0 \oplus I \oplus\ldots\oplus I
\in L_\infty(\R^n; \mL(\wedge))$ acting in all $\mH$.
That $T_{B_0}$ satisfies quadratic estimates follows from
Theorem~\ref{thm:unpertblockmain}(i),
Theorem~\ref{thm:mainforrealsymm} and Proposition~\ref{prop:constantquadraticests} 
respectively.
Theorem~\ref{thm:mainpert} now shows that we have quadratic estimates
for $T_B$ when $\|B-B_0\|_\infty <\epsilon$.
In case $A_0$ is a block matrix, we note that this result follows already from 
Theorem~\ref{thm:pertestsforblock}.
By Proposition~\ref{prop:QtoFCalc} and Proposition~\ref{prop:lipcont} we have when $\|B-B_0\|_\infty <\epsilon/2$
well defined and bounded operators $E_B= \sgn(T_B)$ which depend Lipschitz 
continuously on $B$, so that
$$
  \| E_{B_2} - E_{B_1} \| \le C \|B_2 - B_1\|_\infty, 
  \qquad \text{when } \|B_i- B_0\|<\epsilon/2,\,\, i=1,2.
$$
To prove that (Neu-$A$) and (Reg-$A$) are well posed, note that
by Lemma~\ref{lem:opmain} it suffices to show that 
$I\pm E_A N_A:\hat\mH^1\rightarrow\hat\mH^1$
are invertible, where $E_A= E_B|_{\hat\mH^1}$ and $N_A= N_B|_{\hat\mH^1}$.
By Theorem~\ref{thm:unpertblockmain}(ii), Theorem~\ref{thm:mainforrealsymm} and 
Theorem~\ref{thm:unpertconstrellich} respectively, the unperturbed 
operators $I\pm E_{A_0}N_{A_0}:\hat\mH^1\rightarrow\hat\mH^1$ 
are invertible.
For the perturbed operator we write
$$
  I\pm E_A N_A = (I\pm E_{A_0}N_{A_0})
  \Big(I\pm (I\pm E_{A_0}N_{A_0})^{-1}
  ( E_AN_A - E_{A_0}N_{A_0} ) \Big).
$$
Here $\| E_AN_A - E_{A_0}N_{A_0} \| \lesssim \|A- A_0\|_\infty$
since we clearly have Lipschitz continuity 
$\|N_{A_2} - N_{A_1}\| \lesssim \|A_2- A_1\|_\infty$.
It follows that $I\pm E_A N_A$ are invertible when $\|A- A_0\|_\infty < \epsilon'$.

The well posedness of (Neu$^\perp$-$A$) is a consequence of Proposition~\ref{prop:regneueq},
since our hypothesis is stable when taking adjoints $A\mapsto A^*$.
Alternatively, we can replace $N_{A_0}$ with $N$ above, proving that $I\pm E_AN$ is an isomorphism,
using Theorem~\ref{thm:unpertblockmain}(ii), Remark~\ref{rem:Nforrealsym} and \ref{rem:Nforconst}.
This proves that for (Neu-$A$), (Reg-$A$) and (Neu$^\perp$-$A$) and each boundary function $g$ 
(being $\phi$ and $\nabla_x\psi$ respectively), there exists
a unique solution $f= E^+_A \hat\mH^1$ satisfying the boundary condition and $\|f\|\approx \|g\|$.
Lemma~\ref{lem:characthardyfcns} and Proposition~\ref{prop:modnontang} proves that the stated norms 
of $f$ and $F(t,x) := (e^{- t|T_A|}f)(x)$ are equivalent.

The well posedness of (Dir-$A$), as well as the first three norm estimates, follows from 
Lemma~\ref{lem:characterisePoisson} and (Neu$^\perp$-$A$).
To show that $\tb{t\nabla_x U_t}\approx\tb{t\pd_t U_t}$, we consider the gradient vector field
$G_t= \nabla_{t,x}U_t$ as in the proof of the lemma.
From (Reg-$A$) and (Neu$^\perp$-$A$), it follows that for all $t>0$, we have
$\|\pd_t U_t\| = \|N^-G_t\|\approx \|N^+G_t\|= \|\nabla_x U_t\|$, from which
the square function estimate for $\nabla_x U_t$ follows.
To show that $\|u\|\approx\|\widetilde N_* (U)\|$, we consider the vector field $F_t$
of conjugate functions from the proof of the lemma.
Proposition~\ref{prop:modnontang} shows that 
$\|u\|\approx\|f\| \approx \|\widetilde N_* (F)\|\gtrsim \|\widetilde N_* (U)\|$.
Moreover, the proof of the reverse estimate $\|u\| \lesssim \|\widetilde N_* (U)\|$
is similar to the proof of $\|f\| \lesssim \|\widetilde N_* (F)\|$ in
Proposition~\ref{prop:modnontang}, using the uniform boundedness of $\mP_t$.

Finally we note that the solution operators for 
(Neu-$A$), (Neu$^\perp$-$A$), (Reg-$A$) and (Dir-$A$) are
\begin{alignat*}{2}
  F_t &= 2e^{-t|T_A|}(E_A-N_A)^{-1}(a_{00}^{-1}\phi e_0), &\qquad
  F_t &= 2e^{-t|T_A|}(E_A-N)^{-1}(\phi e_0), \\
  F_t &= 2e^{-t|T_A|}(E_A+N)^{-1}(\nabla_x\psi), &\qquad
  U_t &= 2\big( e^{-t|T_A|}(E_A-N)^{-1}(ue_0), e_0 \big).
\end{alignat*}
The Lipschitz continuity of $u\mapsto U$ is a consequence of
the corresponding result for (Neu$^\perp$-$A$).
For the norms $\sup_{t>0}\|F_t\|$ and $\tb{t\partial_t F_t}$, Lipschitz continuity
for the solution operators follows from Proposition~\ref{prop:lipcont} as in the proof of
Theorem~\ref{thm:transmain}.
It remains to show Lipschitz continuity for the norm $\|F\|_\mX= \|\wt N_*(F)\|_2$.
To this end, we consider
$$
  f(x)\longmapsto F_{A_z}= (e^{-t|T_{A_z}|}E_{A_z}^+f)(x): \hat\mH^1\longrightarrow \mX
$$
and the truncations $f(x)\mapsto F^k_{A_z}(t,x)=\chi_k(t)F_{A_z}(t,x)$, where
$\chi_k$ denotes the characteristic function for $(1/k,k)$ as in the proof of Lemma~\ref{lem:qgivesanal}(iii).
We claim that it suffices to show that, for each fixed $k$, the operator
$f(x)\mapsto F^k_{A_z}(t,x): \hat\mH^1\rightarrow \mX$ depends holomorphically on $z$.
Indeed, using Schwarz' lemma as in Proposition~\ref{prop:lipcont}, we obtain the
Lipschitz estimate
$$
  \|F^k_{A_2}-F^k_{A_1}\|_\mX\le C \|A_2-A_1\|_\infty \|f\|_2,
$$
uniformly for all $k$, since $\|F^k\|_\mX\le \|F\|_\mX\le C\|f\|$ by Proposition~\ref{prop:modnontang}.
Furthermore, by the monotone convergence theorem
we have 
$$
  \|F^k_{A_2}-F^k_{A_1}\|_\mX\nearrow \|F_{A_2}-F_{A_1}\|_\mX, \qquad k\longrightarrow \infty,
$$
so the desired Lipschitz continuity follows after taking limits.

To prove that $f(x)\mapsto F^k_{A_z}(t,x)$ is holomorphic, we note that 
\begin{multline*}
\|F^k\|^2_\mX\lesssim \int_{\R^n}\left( \sup_{t>0} 
  \int_{|s-t|<c_0t}\int_{|y-x|< c_1s/(1-c_0) } |F^k(s,y)|^2 \,\frac{dsdy}{s^{n+1}} \right)dx \\
\lesssim \int_{\R^n}\left(
  \int_{1/k}^{k}\int_{|y-x|< c_1s/(1-c_0) } |F^k(s,y)|^2 \,\frac{dsdy}{s^{n+1}} \right)dx 
\approx \int_{1/k}^{k} \|F^k_s\|^2_2 \, ds,
\end{multline*}
for fixed $k$,
where in the last step we use that $1/k\le s^{-1}\le k$.
Since $f(x)\mapsto F^k(t,x):\hat\mH^1\rightarrow L_2(\R^n\times(1/k,k))$ 
is holomorphic by Lemma~\ref{lem:qgivesanal}(ii)
and the embedding $L_2(\R^n\times(1/k,k)) \hookrightarrow \mX$ is continuous and 
independent of $z$, it follows that
$f(x)\mapsto F^k(t,x): \hat\mH^1\rightarrow \mX$ is holomorphic
for each fixed $k$. 
This completes the proof of Theorem~\ref{thm:main}.
\end{proof}


\bibliographystyle{acm}

\end{document}